\documentclass[11pt,reqno,a4paper]{amsart}
\usepackage{amsmath,amsthm,verbatim,amscd,amssymb,setspace,hyperref,esint,mathtools}
\usepackage{exscale,color}
\usepackage[colorinlistoftodos,prependcaption]{todonotes}

\usepackage{mathrsfs}

\usepackage{enumitem}

\renewcommand{\epsilon}{\varepsilon}
\newcommand{\N}{\mathbb{N}}
\newcommand{\Z}{\mathbb{Z}}

\newcommand{\R}{\mathbb{R}}
\newcommand{\C}{\mathbb{C}}

\renewcommand{\Re}{\operatorname{Re}}

\newcounter{mtheorem}
\newtheorem{mtheorem}[mtheorem]{Theorem}

\newcommand{{\vol}}{\rm vol}
\newcommand{\p}{\partial}
\newcommand{\norm}[1]{\Vert #1 \Vert}

\newcommand{\Ric}{\operatorname{Ric}}

\newcommand{\Rm}{\operatorname{Rm}}

\providecommand{\norm}[1]{\lVert#1\rVert}

\def\tr{\operatorname{tr}}

\def\inj{\operatorname{inj}}

\def\Id{\operatorname{Id}}

\def \Cstarn{(\mathbb{C}^*)^n}

\def \t {\mathfrak{t}}
\def \bp {\bar{\partial}}
\def \Pol {P_{-K_M}}

\def\tr{\operatorname{tr}}
\def\Ric{\operatorname{Ric}}
\def\tr{\operatorname{tr}}

\def\inj{\operatorname{inj}}

\def\Id{\operatorname{Id}}

\def\vol{\operatorname{vol}}

\newtheoremstyle{fancy}{}{}{\itshape}{}{\textbf\bgroup}{.\egroup}{ }{}
\newtheoremstyle{fancy2}{}{}{\rm}{}{\textbf\bgroup}{.\egroup}{ }{}

\theoremstyle{fancy}
\newtheorem{theorem}{Theorem}[section]
\newtheorem{lemma}[theorem]{Lemma}
\newtheorem{corollary}[theorem]{Corollary}

\newtheorem{prop}[theorem]{Proposition}

\theoremstyle{fancy2}
\newtheorem{definition}[theorem]{Definition}

\newtheorem{remark}[theorem]{Remark}

\newtheorem{claim}[theorem]{Claim}

\setlist{leftmargin=*}

\numberwithin{equation}{section}

\begin{document}
\title{An Aubin continuity path for asymptotically conical toric shrinking gradient K\"ahler-Ricci solitons: openness and a solution for $t=0$}
\date{\today}

\author{Ivin Babu}
\address{Department of Mathematical Sciences, The University of Texas at Dallas, Richardson, TX 75080}
\email{ivin.babu@utdallas.edu}

\author{Ronan J.~Conlon}
\address{Department of Mathematical Sciences, The University of Texas at Dallas, Richardson, TX 75080}
\email{ronan.conlon@utdallas.edu}

\author{Alix Deruelle}
\address{Universit\'e Paris-Saclay, CNRS, Laboratoire de math\'ematiques d'Orsay, 91405, Orsay, France}
\email{alix.deruelle@universite-paris-saclay.fr}

\date{\today}

\begin{abstract}
We show that any toric asymptotically conical
shrinking gradient K\"ahler-Ricci soliton on an anti-canonically polarised resolution of a K\"ahler cone
satisfies a complex Monge-Amp\`ere equation. We then set up an Aubin continuity path to
solve the resulting equation and show that it has a solution at the initial value of the path parameter
in the toric case. This we do by implementing another continuity method.
Finally, we prove openness of the initial value of the path
parameter independent of the toricity.
\end{abstract}

\maketitle

\markboth{Ivin Babu, Ronan J.~Conlon and Alix Deruelle}{An Aubin continuity path for AC toric shrinking gradient K\"ahler-Ricci solitons}

\tableofcontents

\section{Introduction}

\subsection{Overview}\label{overview}
A \emph{Ricci soliton} is a triple $(M,\,g,\,X)$, where $M$ is a Riemannian manifold endowed with a complete Riemannian metric $g$
and a complete vector field $X$, such that
\begin{equation}\label{hot}
\Ric_{g}+\frac{1}{2}\mathcal{L}_{X}g=\lambda g
\end{equation}
for some $\lambda\in\mathbb{R}$. The vector field $X$ is called the
\emph{soliton vector field}. If $X=\nabla^{g} f$ for some smooth real-valued function $f$ on $M$,
then we say that $(M,\,g,\,X)$ is \emph{gradient}. In this case, the soliton equation \eqref{hot}
becomes $$\Ric_{g}+\operatorname{Hess}_{g}(f)=\lambda g,$$
and we call $f$ the \emph{soliton potential}. In the case of gradient Ricci solitons, the completeness of $X$ is guaranteed by the completeness of $g$
\cite{zhang12}.

Let $(M,\,g,\,X)$ be a Ricci soliton. If $g$ is K\"ahler and $X$ is real holomorphic, then we say that $(M,\,g,\,X)$ is a \emph{K\"ahler-Ricci soliton}. Let $\omega$ denote the K\"ahler
form of $g$. If $(M,\,g,\,X)$ is in addition gradient, then \eqref{hot} may be rewritten as
\begin{equation}\label{krseqn}
\rho_{\omega}+i\partial\bar{\partial}f=\lambda\omega,
\end{equation}
where $\rho_{\omega}$ is the Ricci form of $\omega$ and $f$ is the soliton potential.

Finally, a Ricci soliton and a K\"ahler-Ricci soliton are called \emph{steady} if $\lambda=0$, \emph{expanding}
if $\lambda<0$, and \emph{shrinking} if $\lambda>0$ in \eqref{hot}.
One can always normalise $\lambda$, when non-zero, to satisfy $|\lambda|=1$. We henceforth assume that this is the case.

Ricci solitons are interesting both from the point of view of canonical metrics and of the Ricci flow. On one hand, they represent one direction in which the notion of an Einstein manifold can be generalised. On compact manifolds, shrinking Ricci solitons are known to exist in several instances where there are obstructions to the existence of Einstein metrics; see for example \cite{soliton}. By the maximum principle, there are no nontrivial expanding or steady Ricci solitons on compact manifolds. However, there are many examples on noncompact manifolds; see for example \cite{Cao-KR-sol, FIK, Wang, con-der, schafer2, schaffer, heather, conlon33} among others. On the other hand, one can associate to a Ricci soliton a self-similar solution of the Ricci flow, and gradient shrinking Ricci solitons in particular provide models for finite-time Type I singularities of the flow \cite{topping, naber}. From this perspective, it is an important problem to classify and construct examples of such solitons in order to better understand singularity formation in the Ricci flow.

In this article, we are concerned with the construction of toric shrinking gradient K\"ahler-Ricci solitons. In essence, we set up an Aubin continuity path for a complex Monge-Amp\`ere equation to construct such solitons in the asymptotically conical toric setting that allows for control at infinity on the data of the equation. We then show that there is a solution to this equation for the initial value of the path parameter in the toric case. This we do by implementing another continuity path. Finally, we prove openness of the initial value of the path parameter in a proof that does not involve the toricity.

\subsection{Main results}\label{sec:main}

Before stating our main results, recall that a ``toric manifold'' is an $n$-dimensional manifold $M$ endowed with an effective holomorphic action of the complex
torus $(\mathbb{C}^{*})^{n}$ with compact fixed point set admitting an orbit which is open and dense in $M$. The $(\mathbb{C}^{*})^{n}$-action of course determines the holomorphic action of a real torus $T^{n}\subset(\mathbb{C}^{*})^{n}$. A ``toric K\"ahler cone'' is a K\"ahler cone endowed with an effective
holomorphic action of the complex torus $(\mathbb{C}^{*})^{n}$, admits an open dense orbit, has fixed point set only the apex of the cone, and the action of the real torus $T^{n}\subset(\mathbb{C}^{*})^{n}$ is isometric (and holomorphic). We further assume that the Reeb vector field of the cone lies in the Lie algebra $\mathfrak{t}$ of $T^{n}$.
For a given cone metric $g_{0}$ with Levi-Civita connection $\nabla^{g_{0}}$ and radial function $r$, and
for a tensor $\alpha$ on the cone, we say that
``$\alpha=O(r^{\lambda})$ with $g_{0}$-derivatives'' whenever $|(\nabla^{g_{0}})^{k}\alpha|_{g_{0}}=O(r^{\lambda-k})$ for every $k\in\mathbb{N}_{0}$.

Our first result can now be stated as follows.

\begin{mtheorem}[Equation set-up]\label{mainthm}
Let $C_{0}$ be a toric K\"ahler cone of complex dimension $n$ with complex structure $J_{0}$, and let
$\pi:M\to C_{0}$ be a toric-equivariant quasi-projective resolution of $C_{0}$ with complex structure $J$ and exceptional
set $E$ such that $-K_{M}$ is $\pi$-ample. Let $T^{n}$ denote the real $n$-torus acting (equivariantly) on $\pi:M\to C_{0}$
and write $\mathfrak{t}$ for the Lie algebra of $T^{n}$. Recall by assumption that the Reeb vector field of $C_{0}$ lies in $\mathfrak{t}$.
Then:
\begin{enumerate}
\item There exists a unique complete real holomorphic vector field $JX\in\mathfrak{t}$ such that $X$ is the soliton vector field of any complete toric shrinking gradient K\"ahler-Ricci soliton on $M$.
 \item There exists a $T^{n}$-invariant K\"ahler cone metric $g_{0}$ on $C_{0}$ with radial function $r$ such that
$J_{0}r\partial_{r}\in\mathfrak{t}$ and $d\pi(X)=r\partial_{r}$.
      \item  There exists a complete K\"ahler metric $\omega$ on $M$ invariant under the action of $T^{n}$ such that
outside a compact subset of $M$ containing $E$, $\omega=\pi^{*}(\omega_{0}+\rho_{\omega_{0}})$. Here, $\omega_{0}$ is the K\"ahler form of $g_{0}$ and $\rho_{\omega_{0}}$ is the associated Ricci form. In particular, $\omega-\pi^{*}\omega_{0}=O(r^{-2})$ with $g_{0}$-derivatives.
\item There exists a unique $T^{n}$-invariant function $f\in C^{\infty}(M)$ with
$-\omega\lrcorner JX=df$ such that outside a compact subset of $M$ containing $E$, $f=\pi^{*}\left(\frac{r^{2}}{2}-n\right)$ and
$\Delta_{\omega}f+f-\frac{X}{2}\cdot f=O(r^{-2})$ with $g_{0}$-derivatives.
\item Every shrinking gradient K\"ahler-Ricci soliton on $M$ with soliton vector field $X$ is of the form
  $\omega+i\partial\bar{\partial}\varphi$ for some smooth real-valued function $\varphi\in C^{\infty}(M)$ with $\omega+i\partial\bar{\partial}\varphi>0$
  satisfying the complex Monge-Amp\`ere equation
\begin{equation*}
(\omega+i\partial\bar{\partial}\varphi)^{n}=e^{F+\frac{X}{2}\cdot\varphi-\varphi}\omega^{n},
\end{equation*}
where $F\in C^{\infty}(M)$ is equal to a constant outside a compact subset of $M$ and is determined by the fact that
$$\rho_{\omega}+\frac{1}{2}\mathcal{L}_{X}\omega-\omega=i\partial\bar{\partial}F\qquad\textrm{and}\qquad\int_{M}(e^{F}-1)e^{-f}\omega^{n}=0.$$
Here, $\rho_{\omega}$ denotes the Ricci form of $\omega$. Moreover, if the shrinking soliton is $T^{n}$-ivnariant, then 
so is $\varphi$.
\end{enumerate}
\end{mtheorem}

Note that since in our geometric set-up the torus action on $M$ contains a fixed point, $M$ is simply connected \cite[Theorem 12.1.10]{cox},
a necessary condition for the existence of a shrinking gradient K\"ahler-Ricci soliton \cite{carlos, sun2024}. Our assumption of
quasi-projectivity of the resolution is also necessary \cite{sun2024} and is used in the proof of Theorem \ref{mainthm}(iii).
If $M$ compactifies to be a Fano orbifold, then the condition of quasi-projectiveness and $\pi$-ampleness of
$-K_{M}$ are automatically satisfied, with the latter yet another necessary condition resulting from the shrinking soliton equation.
 Note that throughout, our convention for the K\"ahler Laplacian $\Delta_{\omega}$ is that with respect to the K\"ahler form $\omega$, $\Delta_{\omega}h=\operatorname{tr}_{\omega}\left(i\partial\bar{\partial}h\right)$ for $h$ a smooth real-valued function, so that the eigenvalues of minus the Laplacian are non-negative on a compact Riemannian manifold.

Part (i) of the theorem determines the soliton vector field of any complete toric shrinking gradient K\"ahler-Ricci soliton on $M$ and follows immediately
from \cite[Theorem A]{charlie}, where it is asserted that a complete toric shrinking gradient K\"ahler-Ricci soliton is unique up to biholomorphism. The vector field $JX$ is characterised by the fact that it is the point in a specific open convex subset of $\mathfrak{t}$ at which a certain strictly convex functional attains its minimum. More precisely, because $M$ is simply connected
and toric, the action of $T^{n}$ is Hamiltonian and there exists a strictly convex functional $\mathcal{F}_{\omega}:\Lambda_{\omega}\to\mathbb{R}_{>0}$, the ``weighted volume functional'' \cite[Definition 5.16]{cds}, defined on an open convex cone $\Lambda_{\omega}\subset\mathfrak{t}$ uniquely determined by the image of $M$ under the moment map defined by the action of $T^{n}$ and the choice of $\omega$ \cite[Proposition 1.4]{wu} and well-defined by the non-compact version of the Duistermaat-Heckman formula \cite{wu} (see also \cite[Theorem A.3]{cds}). Because $T^{n}$ provides  full-dimensional torus symmetry, the domain $\Lambda_{\omega}$ of $\mathcal{F}_{\omega}$ and $\mathcal{F}_{\omega}$ itself only depend on the torus action \cite[Lemmas 2.26 and 2.27]{ccd} so that both are independent of the choice of $\omega$. In addition, in the toric case $\mathcal{F}_{\omega}$ is proper on $\Lambda$ and so must attain a unique minimum on $\Lambda$.
This minimum defines a distinguished point in $\mathfrak{t}$, namely the only vector field in $\mathfrak{t}$ that can admit a complete toric shrinking gradient K\"ahler-Ricci soliton \cite[Theorem 4.6]{charlie}. This is precisely the vector field $JX$ of Theorem \ref{mainthm}(i). Since all the data is explicit and determined by the torus action, one can a priori determine this vector field for a given $M$; see for example \cite[Section A.4]{cds}.

Part (ii) gives a reference K\"ahler cone metric on $C_{0}$ compatible with the hypothetical soliton vector field, whereas part (iii) details a reference metric on $M$ that is
asymptotic to this cone metric. With respect to this background metric, part (iv) gives a normalised Hamiltonian potential $f$ for the soliton vector field $X$, and
(v) gives a complex Monge-Amp\`ere equation \eqref{cmaa} that any complete toric shrinking gradient K\"ahler-Ricci soliton on $M$ must satisfy, notably with control on the asymptotics of the data $F$ of the equation. By \cite{charlie}, we know that there is at most one such soliton on $M$. Our ultimate aim is to show that this equation has a solution, resulting in an asymptotically conical (AC) complete toric shrinking gradient K\"ahler-Ricci soliton on $M$.
By \cite{Esp25}, we would then know that this is the only AC shrinking gradient K\"ahler-Ricci soliton on $M$. Our attempt to solve equation \eqref{cmaa} is by implementing the Aubin continuity path that was introduced for K\"ahler-Einstein manifolds \cite[Section 7.26]{Aubin}. Specifically in our case, we consider the path
\begin{equation}\label{ast-t}
\left\{
\begin{array}{rl}
(\omega+i\partial\bar{\partial}\varphi_{t})^{n}=e^{F+\frac{X}{2}\cdot\varphi_{t}-t\varphi_{t}}\omega^{n},&\quad\varphi\in C^{\infty}(M),\quad\mathcal{L}_{JX}\varphi=0,\quad\omega+i\partial\bar{\partial}\varphi>0,\quad t\in[0,\,1],\\
\int_{M}e^{F-f}\omega^{n}=\int_{M}e^{-f}\omega^{n}. &
\end{array} \right.\tag{$\ast_{t}$}
\end{equation}
Then we have:

\begin{mtheorem}[Existence when $t=0$]\label{mainthm2}
In the setting of Theorem \ref{mainthm}, there exists a $T^{n}$-invariant smooth function $\psi\in \mathcal{M}_{X}^{\infty}(M)$ (cf.~Section \ref{function-spaces-subsection}) inducing a K\"ahler form $\omega+i\partial\bar{\partial}\psi$ on $M$ such that
\begin{equation}\label{ronan1}
(\omega+i\partial\bar{\partial}\psi)^{n}=e^{F+\frac{X}{2}\cdot\psi}\omega^{n},
\end{equation}
where $\int_{M}\psi\,e^{-f}\omega^{n}=0$. In particular, outside a compact subset, $\psi=c\log f+\vartheta$ for some constant $c\in\mathbb{R}$ and
a smooth real-valued function $\vartheta:M\to\mathbb{R}$ satisfying $\vartheta=O(1)$ with $g_{0}$-derivatives.
\end{mtheorem}
\noindent This provides the asserted solution to $(\ast_{0})$. We obtain this solution by implementing another continuity path. In the compact case,
this was achieved by Zhu \cite{Zhu-KRS-C1}, and in the non-compact toric case, this was achieved in \cite{ccd2} in a different asymptotic setting to that considered here.
The proof of Theorem \ref{mainthm2} involves working in function spaces
suitably adapted to the conical geometry at infinity. Their definition is given at the beginning of Section \ref{function-spaces-subsection}.
Here (and in \cite{ccd2}), the toric assumption is only used in two places, namely:
\begin{itemize}
\item To find the soliton vector field $X$ via properness of the weighted volume functional in the toric setting \cite{charlie} (cf.~Lemma \ref{vector-field});
\item To achieve a uniform infimum bound of the solution along the continuity path (cf.~Section \ref{infmyballs}).
\end{itemize}
None of the other a priori estimates require toricity. As in \cite[Lemma 7.6]{ccd2}, the infimum of the solution is localised along the continuity path; cf.~Lemma \ref{lemma-loc-crit-pts}. However,
contrary to \cite[Lemma 7.6]{ccd2}, the localisation of the infimum here relies crucially on the fact that the normalisation of the Hamiltonian potential of $X$
with respect to the solution metric as dictated in Lemma \ref{lemma-tr-star-star} surprisingly depends only the initial data, rather than on the unknown solution. In order to keep the geometry at infinity fixed along the path, new weighted a priori estimates are also derived.

By showing that the linearisation of $(\ast_{0})$ is invertible between polynomially weighted function spaces (cf.~Section \ref{opennesss} for the precise definitions), we further prove:
\begin{mtheorem}[Openness when $t=0$]\label{mainthm3}
In the setting of Theorem \ref{mainthm}, there exists $\varepsilon>0$ such that \eqref{ast-t} has a $T^{n}$-invariant polynomially growing smooth solution for all $t\in[0,\,\varepsilon)$.
\end{mtheorem}
\noindent The proof of this theorem requires function spaces different from those considered in the proof of Theorem \ref{mainthm2} for a good reason--the linearised operator associated to \eqref{ast-t} has a kernel which contains functions that grow polynomially at infinity. In addition, the decay rate of the solution provided by Theorem \ref{mainthm3} degenerates
as $t\in[0,\,\varepsilon)$ increases. This is consistent with the fact that, as noted on \cite[p.19]{sun2024}, the construction of non-compact complete shrinking K\"ahler-Ricci solitons is a free boundary problem, and so in the conical setting one is unable to determine in advance the asymptotic cone of the shrinking soliton. Indeed, the cone is the singular
end time of the associated K\"ahler-Ricci flow, which results in an ill-posed problem for parabolic systems.

\subsection{Outline of paper}

We begin in Section \ref{sec-srs} by recalling the definitions of Riemannian and K\"ahler cones, and of manifolds for which these serve as asymptotic models.
This makes up Sections \ref{cones1}--\ref{cones2}. We then recall the definition of a metric measure space in Section \ref{metricmeasure}. In Section \ref{pooly}, we digress and define
polyhedrons and polyhedral cones before moving on to the definition of a Hamiltonian action in Section \ref{hamilton}.
Section \ref{toric-geom} then comprises the background material on toric geometry that we require.

In Section \ref{sec-construction-back-metric}, namely in Proposition \ref{mainprop}, we construct a background metric with the desired properties, set up the 
complex Monge-Amp\`ere equation, and provide the normalisation for the Hamiltonian of $JX$,
resulting in the proof of Theorem \ref{mainthm}. Our background metric is asymptotic to a K\"ahler cone at infinity, which itself is a shrinking 
gradient K\"ahler-Ricci soliton compatible with $X$ up to terms of order $O(r^{-2})$. This property is what allows us to set up the complex Monge-Amp\`ere equation with decaying data.

Section \ref{sec-poin-inequ} establishes sufficient geometric conditions for an $L^p$-Poincar\'e inequality, $p\geq 1$, to hold with respect to a given weighted measure. The result is
Proposition \ref{poincare} and is of independent interest. The proof is purely Riemannian in nature and is independent of the rest of the paper.
Notice that this proposition plays a crucial role in deriving the a priori weighted energy estimate for the complex Monge-Amp\`ere equation \eqref{ronan1} with decaying data
in Proposition \ref{prop-a-priori-ene-est}.

From Section \ref{linear-theory-section} onwards, the content takes on a more analytic flavour with the proof of Theorem \ref{mainthm2} making up Sections \ref{linear-theory-section}--\ref{sec-a-priori-est}. We prove Theorem \ref{mainthm2} by implementing the continuity method, the specific path of which is
outlined at the beginning of Section \ref{sec-a-priori-est}. 

In Section \ref{linear-theory-section}, we study the properties of the drift Laplacian of our background metric acting on polynomially weighted function spaces. More precisely, in Section \ref{function-spaces-subsection} we introduce polynomially weighted function spaces,
the elements of which are invariant under the flow of $JX$. We follow this up in Section \ref{linear} by showing that the drift Laplacian of our
background metric is an isomorphism between such spaces. This result is the content of Theorem \ref{iso-sch-Laplacian-pol}.
Using it, we then prove Theorem \ref{Imp-Def-Kah-Ste} in Section \ref{invert-poly} that serves as the openness part of the continuity argument.
The closedness part involves a priori estimates and these make up Section \ref{sec-a-priori-est}.

Contrary to \cite{conlon33}, the presence of the unbounded vector field $X$ makes the analysis here much more involved. An a priori lower bound for the radial derivative $X\cdot \psi$, where $\psi$ solves \eqref{ronan1}, has to be proved \emph{before} the a priori $C^{0}$-bound in order to avoid a circular argument; see Section \ref{sec-low-bd-rad-der}. An a priori energy bound is established in Section \ref{sec-a-priori-energy} through the use of the so-called Aubin-Tian-Zhu's functionals, and leads to an a priori upper bound on a solution to the complex Monge-Amp\`ere equation \eqref{ronan1}; cf.~Propositions \ref{prop-a-priori-ene-est} and \ref{prop-bd-abo-uni-psi}. As explained above, the invariance of the solution under the whole torus action (and not just under the flow of $JX$) is crucial in deriving an a priori lower bound on the infimum; cf.~Proposition \ref{prop-bd-bel-uni-psi}. Then and only then is an a priori upper bound on the radial derivative of a solution to \eqref{ronan1} achieved; cf.~Proposition \ref{prop-bd-uni-X-psi}. Section \ref{sec-high-der} is devoted to proving a local bootstrapping phenomenon for \eqref{ronan1} that results in a priori bounds for higher derivatives of the solution. Finally, Section \ref{sec-wei-bd} comprises new a priori weighted estimates for \eqref{ronan1}, leading to the completion of the proof of Theorem \ref{mainthm2} in Section \ref{sec-proof-main-thm}.

We conclude the paper in Section \ref{opennesss} with the proof of Theorem \ref{mainthm3}.

\subsection{Acknowledgments}
This material is based upon work supported by the National Science Foundation under Grant No.~DMS-1928930, while the second author was in residence
at the Simons Laufer Mathematical Sciences Institute (formerly MSRI) in Berkeley, California, during the Fall 2024 semester. He wishes to thank SLMath for their excellent working conditions and hospitality during this time. He is also supported by a Simons Travel Grant.

The third author is partially supported by grants from the French National Research Agency  ANR-24-CE40-0702 (Project OrbiScaR) and the Charles Defforey Fondation-Institut de France via the project ``KRIS''. He also benefits from a Junior Chair from the Institut Universitaire de France.

The authors wish to thank Charles Cifarelli for providing the proof of Lemma \ref{normpolytope-tof} and for useful discussions, and Emmanuel Milman for enlightening discussions concerning $L^{p}$-Poincar\'e inequalities on metric measure spaces and for providing us with the crucial reference \cite{Bak-Bar-Cat-Gui}.

\section{Preliminaries}\label{sec-srs}

\subsection{Riemannian cones}\label{cones1}
 For us, the definition of a Riemannian cone will take the following form.
\begin{definition}\label{cone}
Let $(S, g_{S})$ be a closed Riemannian manifold. The \emph{Riemannian cone} $C_{0}$ with \emph{link} $S$ is defined to be $\R_{>0}\times S$ with metric $g_0 = dr^2 \oplus r^2g_{S}$ up to isometry. The radius function $r$ is then characterized intrinsically as the distance from the apex in the metric completion.
\end{definition}

Suppose that we are given a Riemannian cone $(C_0,g_{0})$ as above. Let $(r,x)$ be polar coordinates on $C_{0}$, where $x\in S$, and for $t>0$, define a map
$$\nu_{t}: [1,2]\times S \ni (r,x) \mapsto (tr,x) \in  [t,2t]\times S.$$ One checks that $\nu_{t}^{*}(g_{0})=t^{2}g_{0}$ and $\nu^{*}_{t}\circ\nabla^{g_0}=\nabla^{g_0}\circ\nu_{t}^{*}$, where $\nabla^{g_0}$ is the  Levi-Civita connection of $g_{0}$.

\begin{lemma}\label{simple321}
Suppose that $\alpha\in\Gamma((TC_0)^{\otimes p}\otimes (T^{*}C_0)^{\otimes q})$ satisfies $\nu_{t}^{*}(\alpha)=t^{k}\alpha$ for every $t>0$ for some $k\in\R$. Then $|(\nabla^{g_0})^{l}\alpha|_{g_{0}}=O(r^{k+p-q-l})$ for all $l\in\N_0$.
\end{lemma}

We shall say that ``$\alpha=O(r^{\lambda})$ with $g_{0}$-derivatives'' whenever $|(\nabla^{g_0})^{k}\alpha|_{g_{0}}=O(r^{\lambda-k})$ for every $k \in \N_0$.
We will then also say that $\alpha$ has ``rate at most $\lambda$'', or sometimes, for simplicity, ``rate $\lambda$'', although it should be understood that (at least when $\alpha$ is purely polynomially behaved and does not contain any $\log$ terms) the rate of $\alpha$ is really the infimum of all $\lambda$ for which this holds.

\subsection{K\"ahler cones}\label{conekahler}
We may further impose that a Riemannian cone is K\"ahler, as the next definition demonstrates.
\begin{definition}A \emph{K{\"a}hler cone} is a Riemannian cone $(C_{0},g_0)$ such that $g_0$ is K{\"a}hler, together with a choice of $g_0$-parallel complex structure $J_0$. This will in fact often be unique up to sign. We then have a K{\"a}hler form $\omega_0(X,Y) = g_0(J_0X,Y)$, and $\omega_0 = \frac{i}{2}\p\bar{\p} r^2$ with respect to $J_0$.
\end{definition}

The vector field $r\partial_{r}$ is real holomorphic and $\xi:=J_{0}r\partial_r$ is real holomorphic and Killing \cite[Appendix A]{Yau}. This latter vector field is known as the \emph{Reeb vector field}. The closure of its flow in the isometry
group of the link of the cone generates the holomorphic isometric action of a real torus on $C_{0}$ that
fixes the apex of the cone. The real $(1,\,1)$-form $\omega^{T}:=i\partial\bar{\partial}\log(r)$ defines a K\"ahler form on 
the local leaf space of the foliation of the link of the cone induced by the flow of $\xi$ and is called the 
\emph{transverse K\"ahler form}.

Every K\"ahler cone is affine algebraic.
\begin{theorem}\label{t:affine}
For every K{\"a}hler cone $(C_{0},g_0,J_0)$, the complex manifold $(C_{0},J_0)$ is isomorphic to the smooth part of a normal algebraic variety $V \subset \C^N$ with one singular point. In addition, $V$ can be taken to be invariant under a $\C^*$-action $(t, z_1,\ldots,z_N) \mapsto (t^{w_1}z_1,\ldots,t^{w_N}z_N)$ such that all $w_i $ are positive.
\end{theorem}
\noindent This can be deduced from arguments written down by van Coevering in \cite[\S 3.1]{vanC4}.

The holomorphic torus action on a K\"ahler cone leads to the notion of an \emph{equivariant resolution}.
\begin{definition}\label{equivariantt}
Let $C_0$ be a K\"ahler cone and let $\pi:M\to C_0$ be a resolution of $C_0$. We say that
$\pi:M\to C_0$ is an \emph{equivariant resolution} with respect to the holomorphic action of a Lie group on $C_{0}$ if that group
action lifts via $\pi$ to a holomorphic action on $M$.
\end{definition}
Such a resolution of a K\"ahler cone always exists; see \cite[Proposition 3.9.1]{kollar}. We recall:

\begin{lemma}[\protect{\textnormal{\cite[Lemma 2.4]{conlon33}}}]\label{nice}
Let $(C_{0},\,g_{0})$ be a K\"ahler cone with Reeb vector field $\xi$ and let $K\subseteq C_{0}$ be a compact subset containing the apex of $C_{0}$ such that $C_{0}\setminus K$ is connected.
If $u:C_{0}\setminus K\to\mathbb{R}$ is a smooth real-valued function defined on $C_{0}\setminus K$
that is pluriharmonic (meaning that $\partial\bar{\partial}u=0$) and invariant under the flow of
$\xi$, then $u$ is a real constant.
\end{lemma}

We conclude this section with a gluing lemma, analogous to \cite[Lemma 2.4]{ccd2}.
\begin{lemma}[Gluing lemma]\label{glue}
Let $(C_{0},\,g_{0})$ be a K\"ahler cone with radius function $r$ and let $\pi:(M,\,g)\to(C_{0},\,g_{0})$ be an asymptotically conical K\"ahler metric
on a resolution $\pi:M\to C_{0}$ with exceptional set $E$ with tangent cone $(C_{0},\,g_{0})$. Consider $r$ as a function on $M\setminus E$.

Let $\phi\in C^{\infty}(M\setminus E)$ be such that $\phi=O(\log(r))$, $|d\phi|_{g_{0}}=O(r^{-1})$, and
$|i\partial\bar{\partial}\phi|_{g_{0}}=O(r^{-2})$ for $r>\frac{1}{2}$. Then for all $R\geq1$,
there exists a cut-off function $\chi_{R}:M\to\mathbb{R}$ supported on $M\setminus\{r\leq R\}$ with $\chi_{R}(x)=1$
if $r(x)>2R$ such that $$|i\partial\bar{\partial}(\chi_{R}\cdot\phi)|_{g}\leq\frac{C}{R}\left(\|(\log(r)+1)^{-1}\cdot\phi\|_{C^{0}(M\setminus\{r\,\leq\,1\},\,g)}
+\|r\cdot d\phi\|_{C^{0}(M\setminus\{r\,\leq\,1\},\,g)}+\|r^{2}\cdot i\partial\bar{\partial}\phi\|_{C^{0}(M\setminus\{r\,\leq\,1\},\,g)}\right)$$
for  some $C>0$ independent of $R$. In particular, $\chi_{R}\cdot\phi=\phi$ on $\{r(x)>2R\}$.
\end{lemma}

\begin{proof}
Let $\chi:\mathbb{R}\rightarrow\mathbb{R}$ be a smooth function satisfying
$\chi(x)=0$ for $x\leq 1$, $\chi(x)=1$ for $x\geq 4$, and $|\chi(x)|\leq 1$ for all $x$, and with it, define a function $\chi_{R}:M\to\mathbb{R}$ by
$$\chi_{R}(x)=\chi\left(\frac{r(x)^{2}}{R^{2}}\right)\qquad\textrm{for $R\geq1$ as in the statement of the lemma}.$$
Then $\chi_{R}$ is identically zero on $\{x\in M\setminus E\,|\,r(x)<R\}\cup E$ and identically equal to one on the set $\{x\in M\,|\,r(x)>2R\}$.
Define $\phi_{R}:=\chi_{R}\cdot\phi$. Then the closed real $(1,\,1)$-form $i\partial\bar{\partial}(\chi_{R}\cdot\phi)$ on $M$ is given by
\begin{equation*}
\begin{split}
i\partial\bar{\partial}(\chi_{R}\cdot\phi)=\chi_{R}(r)\cdot i\partial\bar{\partial}\phi
&+\chi'\left(\frac{r^{2}}{R^2}\right)\cdot
i\frac{\partial r^{2}}{R}\wedge\frac{\bar{\partial}\phi}{R}+\frac{\phi}{R^{2}}\cdot\chi'\left(\frac{r^{2}}{R^{2}}\right)
\cdot i\partial\bar{\partial}r^{2}\\
&+\chi'\left(\frac{r^{2}}{R^{2}}\right)\cdot\frac{i\partial\phi}{R}\wedge\frac{\bar{\partial}r^{2}}{R}+\frac{\phi}{R^{2}}\cdot\chi''\left
(\frac{r^{2}}{R^{2}}
\right)\cdot i\frac{\partial r^{2}}{R}\wedge \frac{\bar{\partial}r^{2}}{R}.
\end{split}
\end{equation*}
The assumptions on $\phi$ and its derivatives then imply for example that
\begin{equation*}
\begin{split}
|\chi_{R}(x)\cdot i\partial\bar{\partial}\phi|_{g_{0}}&\leq\sup_{r\,\in\,[R,\,\infty)}|i\partial
\bar{\partial}\phi|_{g_{0}}\\
&\leq\left(\sup_{r\,\in\,[R,\,\infty)}r^{-2}\right)
\left(\sup_{r\,\in\,[R,\,\infty)}r^{2}\cdot|i\partial
\bar{\partial}\phi|_{g_{0}}\right)\leq R^{-2}\|r^{2}\cdot i\partial\bar{\partial}\phi\|_{C^{0}(M\setminus\{r\,<\,R\},\,g_{0})}
\end{split}
\end{equation*}
and that
\begin{equation*}
\left|\chi'\left(\frac{r^{2}}{R^{2}}\right)\cdot
i\frac{\partial r^{2}}{R}\wedge\frac{\bar{\partial}\phi}{R}\right|_{g_{0}}\leq\frac{C}{R^{2}}\left(\sup_{r\,\in\,[R,\,2R]}\left|
r\partial r\wedge\bar{\partial}\phi\right|_{g_{0}}\right)\leq CR^{-2}\|r\cdot d\phi\|_{C^{0}(M\setminus\{r\,<\,R\},\,g_{0})}.
\end{equation*}
The estimate of the lemma is now clear.
\end{proof}

\subsection{Asymptotically conical Riemannian manifolds}

We have the following definition.
\begin{definition}\label{d:AC}
Let $(M,g)$ be a complete Riemannian manifold and let $(C_0,g_0)$ be a Riemannian cone. We call $M$ \emph{asymptotically conical} (AC) with tangent cone $C_{0}$ if there exists a diffeomorphism $\Phi: C_0\setminus K \to M \setminus K'$ with $K,K'$ compact, such that $\Phi^*g - g_0 = O(r^{-\epsilon})$ with $g_0$-derivatives for some $\epsilon > 0$. A \emph{radius function} is a smooth function $\rho: M \to [1,\infty)$ with $\Phi^*\rho = r$ away from $K$.
\end{definition}

\subsection{Asymptotically conical K\"ahler manifolds}\label{cones2}

We also have:
\begin{definition}\label{d:ACK}
Let $(M,\,g)$ be a complete K\"ahler manifold with complex structure $J$ and let $(C_0,g_0)$ be a K\"ahler cone with a choice of $g_0$-parallel complex structure $J_0$. We call $M$ \emph{asymptotically conical} (AC) \emph{K\"ahler} with tangent cone $C_{0}$ if there exists a diffeomorphism $\Phi: C_0\setminus K \to M \setminus K'$ with $K,K'$ compact, such that $\Phi^*g - g_0 = O(r^{-\epsilon})$ with $g_0$-derivatives
and $\Phi^*J - J_0 = O(r^{-\epsilon})$ with $g_0$-derivatives for some $\epsilon > 0$. In particular, $(M,\,g)$ is AC with tangent cone $C_{0}$.
\end{definition}

We implicitly only allow for one end in Definitions \ref{d:AC} and \ref{d:ACK}. This is because shrinking gradient K\"ahler-Ricci solitons only have one end \cite{munteanu}.
Furthermore, for our applications, the map $\Phi$ in Definition \ref{d:ACK} will always be a biholomorphism.

\subsection{Basics of metric measure spaces}\label{metricmeasure}

We take the following from \cite{fut}.

A smooth metric measure space is a Riemannian manifold endowed with a weighted volume.
\begin{definition}
A \emph{smooth metric measure space} is a triple $(M,\,g,\,e^{-f}dV_{g})$, where $(M,\,g)$ is a complete Riemannian manifold with Riemannian metric $g$,
$dV_{g}$ is the volume form associated to $g$, and $f:M\to\mathbb{R}$ is a smooth real-valued function.
\end{definition}
\noindent A shrinking gradient Ricci soliton $(M,\,g,\,X)$ with $X=\nabla^g f$ naturally defines a smooth metric measure space $(M,\,g,\,e^{-f}dV_{g})$.
On such a space, we define the weighted Laplacian $\Delta_{f}$ by
$$\Delta_{f}u:=\Delta u-g(\nabla^g f,\,\nabla u)$$
on smooth real-valued functions $u\in C^{\infty}(M)$. There is a natural $L^{2}$-inner product $\langle\cdot\,,\,\cdot\rangle_{L^{2}_{f}}$ on the space $L^{2}_{f}$ of square-integrable smooth real-valued functions on $M$
with respect to the measure $e^{-f}dV_g$ defined by $$\langle u,\,v\rangle_{L_{f}^{2}}:=\int_{M}uv\,e^{-f}dV_{g},\qquad u,\,v\in L_{f}^{2}.$$
As one can easily verify, the operator $\Delta_{f}$ is symmetric with respect to $\langle\cdot\,,\,\cdot\rangle_{L_{f}^{2}}$.

\subsection{Polyhedrons and polyhedral cones}\label{pooly}

We take the following from \cite{cox}.

Let $E$ be a real vector space of dimension $n$ and let $E^{*}$ denote the dual. Write $\langle\cdot\,,\,\cdot\rangle$ for the evaluation $E^{*}\times E\to\mathbb{R}$. Furthermore, assume that we are given a \emph{lattice} $\Gamma \subset E$, that is, an additive subgroup $\Gamma \simeq \Z^n$. This gives rise to a dual lattice $\Gamma^* \subset E^*$. For any $\nu\in E$, $c\in\mathbb{R}$, let
$K(\nu,\,c)$ be the (closed) half space $\{x\in E\:|\:\langle\nu,\,x\rangle\geq c\}$ in $E$. Then we have:

\begin{definition}
A \emph{polyhedron} $P$ in $E$ is a finite intersection of half spaces,
i.e., $$P=\bigcap_{i=1}^{r}K(\nu_{i},\,c_{i})\qquad\textrm{for $\nu_{i}\in E^{*},\,c_{i}\in\mathbb{R}$}.$$
It is called a \emph{polyhedral cone} if all $c_{i}=0$, and moreover a \emph{rational polyhedral cone} if all $\nu_i \in \Gamma^*$ and $c_i = 0$. In addition, a polyhedron is called \emph{strongly convex} if it does not contain any affine subspace of $E$.
\end{definition}

The following definition will be useful.

\begin{definition}
A polyhedron $P \subset E^{*}$ is called \emph{Delzant} if its set of vertices is non-empty and each vertex $v \in P$ has the property that there are precisely $n$ edges $\{e_1, \dots e_n\}$ (one-dimensional faces) emanating from $v$ and there exists a basis $\{\varepsilon_1, \dots, \varepsilon_n\}$ of $\Gamma^*$ such that $\varepsilon_i$ lies along the ray $\R (e_i - v) $.
\end{definition}

\noindent Note that any such $P$ is necessarily strongly convex.

The asymptotic cone of a polyhedron contains all the directions going off to infinity in the polyhedron.

\begin{definition}\label{recession}
Let $P$ be a polyhedron in $E$. Its \emph{asymptotic cone}, denoted
by $\mathcal{C}(P)$, is the set of vectors $\alpha\in E$ with the property that there exists $\alpha^{0}\in E$
such that $\alpha^{0}+t\alpha\in P$ for sufficiently large $t>0$.
\end{definition}

The asymptotic cone may be identified as follows.

\begin{lemma}[{\cite[Lemma A.3]{wu}}]\label{alpha}
If $P=\bigcap_{i=1}^{r}K(\nu_{i},\,c_{i})$, then $\mathcal{C}(P)=\bigcap_{i=1}^{r}K(\nu_{i},\,0)$.
\end{lemma}
\noindent In particular, the asymptotic cone of a polyhedron is a polyhedral cone.

We also have
\begin{definition}
The \emph{dual} of a polyhedral cone $C$ is the set $C^{\vee}=\{x\in E^{*}\:|\:\langle x,\,C\rangle\geq0\}$.
\end{definition}

\subsection{Hamiltonian actions}\label{hamilton}

Recall what it means for an action to be Hamiltonian.

\begin{definition}
Let $(M,\,\omega)$ be a symplectic manifold and let $T^{n}$ be a real torus acting by symplectomorphisms on $(M,\,\omega)$.
Denote by $\mathfrak{t}$ the Lie algebra of $T^{n}$ and by $\mathfrak{t}^{*}$ its dual. Then we say that the action of $T^{n}$ is \emph{Hamiltonian}
if there exists a smooth map $\mu_{\omega}:M\to\mathfrak{t}^{*}$ such that for all $\zeta\in\mathfrak{t}$,
\begin{equation*}
-\omega\lrcorner\zeta=du_{\zeta},
\end{equation*}
where $u_{\zeta}(x)=\langle\mu_{\omega}(x),\,\zeta\rangle$ for all $\zeta\in\mathfrak{t}$ and $x\in M$
and $\langle\cdot\,,\cdot\rangle$ denotes the dual pairing between $\mathfrak{t}$ and $\mathfrak{t}^{*}$.
We call $\mu_{\omega}$ the \emph{moment map} of the $T^{n}$-action and we call $u_{\zeta}$ a \emph{Hamiltonian (potential)} of $\zeta$.
This is defined up to a constant.
\end{definition}

Define
\begin{equation*}
\Lambda_{\omega}:=\{Y\in\mathfrak{t}\:|\: \textrm{$\langle\mu_{\omega},\,Y\rangle$ is proper and bounded below}\}\subseteq\mathfrak{t}.
\end{equation*}
This set can be identified through the image of $\mu_{\omega}$ in the following way.
\begin{prop}[{\cite[Proposition 1.4]{wu}}]\label{identify}
Let $(M,\,\omega)$ be a (possibly non-compact) symplectic manifold of real dimension $2n$ with symplectic form $\omega$ on which there is a Hamiltonian
action of a real torus $T^{n}$ with moment map $\mu_{\omega}:M\to\mathfrak{t}^{*}$, where $\mathfrak{t}$ is the Lie algebra of $T^{n}$ and $\mathfrak{t}^{*}$
its dual. Assume that the fixed point set of $T^{n}$ is compact and that $\Lambda_{\omega}\neq\emptyset$. Then $\Lambda_{\omega}=\operatorname{int}(\mathcal{C}(\mu_{\omega}(M))^{\vee})$.
\end{prop}

\subsection{Toric geometry}\label{toric-geom}

In this section, we collect together some standard facts from toric geometry as well as recall those results from \cite{charlie} that we require.
We begin with the following definition.
\begin{definition}\label{toricmanifold}
A \emph{toric manifold} is an $n$-dimensional complex manifold $M$ endowed
with an effective holomorphic action of the algebraic torus $\Cstarn$ such that the following hold true.
\begin{itemize}
  \item The fixed point set of the $\Cstarn$-action is compact.
  \item There exists a point $p\in M$ with the property that the orbit $\Cstarn \cdot p \subset M$ forms a dense open subset of $M$.
\end{itemize}
\end{definition}
We will often denote the dense orbit simply by $\Cstarn \subset M$ in what follows.
The $\Cstarn$-action of course determines the action of the real torus $T^n \subset \Cstarn$.

\subsubsection{Divisors on toric varieties and fans}

Let $T^n \subset \Cstarn$ be the real torus with Lie algebra $\t$ and denote the dual pairing between $\t$ and the dual space $\mathfrak{t}^{*}$ by $\langle \cdot\,,\cdot\rangle$. There is a natural integer lattice $\Gamma \simeq \Z^n \subset \t$ comprising all $\lambda \in \t$ such that $\operatorname{exp}(\lambda) \in T^n$ is the identity. This then induces a dual lattice $\Gamma^* \subset \t^*$. We have the following combinatorial definition.

\begin{definition}
 A \emph{fan} $\Sigma$ in $\t$ is a finite set of rational polyhedral cones $\sigma$ satisfying:

	\begin{enumerate}
		\item For every $\sigma \in \Sigma$, each face of $\sigma$ also lies in $\Sigma$.
		\item For every pair $\sigma_1, \sigma_2 \in \Sigma$, $\sigma_1 \cap \sigma_2$ is a face of each.
	\end{enumerate}
\end{definition}

To each fan $\Sigma$ in $\t$, one can associate a toric variety $X_\Sigma$. Heuristically, $\Sigma$ contains all the data necessary to produce a partial equivariant compactification of $\Cstarn$, resulting in $X_\Sigma$. More concretely, one obtains $X_\Sigma$ from $\Sigma$ as follows. For each $n$-dimensional cone $\sigma \in \Sigma$, one constructs an affine toric variety $U_\sigma$ which we first explain. We have the dual cone $\sigma^{\vee}$ of $\sigma$. Denote by $S_\sigma$ the semigroup of those lattice points which lie in $\sigma^{\vee}$ under addition. Then one defines the semigroup ring, as a set, as all finite sums of the form
	\begin{equation*}
		\C[S_\sigma] = \left\{ \left.\sum \lambda_s s \, \right| \, s \in S_\sigma \right\},
	\end{equation*}
with the ring structure defined on monomials by $\lambda_{s_1}s_1\cdot \lambda_{s_2}s_2  = (\lambda_{s_1}\lambda_{s_2})(s_1+ s_2)$ and extended in the natural way. The affine variety $U_\sigma$ is then defined to be $\text{Spec}(\C[S_\sigma])$. This automatically comes endowed with a $\Cstarn$-action with a dense open orbit. This construction can also be applied to the lower dimensional cones $\tau \in \Sigma$. If $\sigma_1 \cap \sigma_2 = \tau$, then there is a natural way to map $U_\tau$ into $U_{\sigma_1}$ and $U_{\sigma_2}$ isomorphically. One constructs $X_\Sigma$ by declaring the collection of all $U_\sigma$ to be an open affine cover of $X_{\Sigma}$ with transition functions determined by $U_\tau$. This identification is also reversible.

\begin{prop}[{\cite[Corollary 3.1.8]{cox}}]\label{fann}
Let $M$ be a smooth toric manifold. Then there exists a fan $\Sigma$ such that $M \simeq X_\Sigma$.
\end{prop}

 \begin{prop}[{\cite[Theorem 3.2.6]{cox}, Orbit-Cone Correspondence}]\label{orbitcone} The $k$-dimensional cones $\sigma \in \Sigma$ are in a natural one-to-one correspondence with the $(n-k)$-dimensional orbits $O_\sigma$ of the $\Cstarn$-action on $X_\Sigma$.
 \end{prop}
\noindent In particular, each ray $\sigma \in \Sigma$ determines a unique torus-invariant divisor $D_\sigma$. As a consequence, a torus-invariant Weil divisor $D$ on $X_\Sigma$ naturally determines a polyhedron $P_D \subset \mathfrak{t}^{*}$. Indeed, we can decompose $D$ uniquely as $D = \sum_{i=1}^N a_i D_{\sigma_i}$, where $\{\sigma_i\}_{i}\subset\Sigma$ is the collection of rays. Then by assumption, there exists a unique minimal lattice element $\nu_i \in \sigma_i \cap \Gamma$. The polyhedron $P_{D}$ is then given by

 \begin{equation} \label{eqnB2}
 	P_D = \left\{ x \in \mathfrak{t}^{*} \: | \: \langle \nu_i, x \rangle \geq - a_i \right\} = \bigcap_{i = 1}^N K(\nu_i, -a_i).
 \end{equation}

\subsubsection{K\"ahler metrics on toric varieties}\label{finito}

For a given toric manifold $M$ endowed with a Riemannian metric $g$ invariant under the action of the real torus $T^n \subset \Cstarn$ and K\"ahler with respect to the underlying complex structure
of $M$, the K\"ahler form $\omega$ of $g$ is also invariant under the $T^n$-action. We call such a manifold a \emph{toric K\"ahler manifold}.
In what follows, we always work with a fixed complex structure on $M$.

Hamiltonian K\"ahler metrics have a useful characterisation due to Guillemin.

\begin{prop}[{\cite[Theorem 4.1]{Guil}}]\label{propB6}
	Let $\omega$ be any $T^n$-invariant K\"ahler form on $M$. Then the $T^{n}$-action is Hamiltonian with respect to $\omega$ if and only if the restriction of $\omega$ to the dense orbit $\Cstarn \subset M$ is exact, i.e., there exists a $T^{n}$-invariant potential $\phi$ such that
	\begin{equation*}
		\omega = 2i\p\bp \phi.
	\end{equation*}
\end{prop}

Fix once and for all a $\Z$-basis $(X_1,\ldots,X_n)$ of $\Gamma \subset \t$. This in particular induces a background coordinate system $\xi=(\xi^1, \dots, \xi^n)$ on $\t$.
Using the natural inner product on $\t$ to identify $\t \cong \t^*$, we can also identify $\t^* \cong \R^n$.
For clarity, we will denote the induced coordinates on $\t^*$ by $x=(x^1,\ldots, x^n)$. Let $(z_1, \dots, z_n)$ be the natural
coordinates on $\Cstarn$ as an open subset of $\C^n$. There is a natural diffeomorphism $\text{Log}:
\Cstarn \to \t \times T^n$ which provides a one-to-one correspondence between $T^n$-invariant smooth functions on
 $\Cstarn$ and smooth functions on $\t$. Explicitly,
\begin{equation}\label{diffeoo}
(z_1, \dots, z_n)\xmapsto{\operatorname{Log}}(\log(r_1), \dots, \log(r_n), \theta_1, \dots, \theta_n)=(\xi_{1},\ldots,\xi_{n},\,\theta_{1},\ldots,\theta_{n}),
\end{equation}
where $z_j = r_j e^{i \theta_j}$,\,$r_{j}>0$. Given a function $H(\xi)$ on $\t$, we can extend $H$ trivially to $\t \times T^n$ and pull back by Log to
obtain a $T^n$-invariant function on $\Cstarn$. Clearly, any $T^n$-invariant function on $\Cstarn$ can be written in this form.

Choose any branch of $\log$ and write $w = \log(z)$. Then clearly $w = \xi + i \theta$, where $\xi=(\xi^1,\ldots,\xi^n)$ are real coordinates on $\t$
(or, more precisely, there is a corresponding lift of $\theta$ to the universal cover with respect to which this equality holds),
and so if $\phi$ is $T^n$-invariant and $\omega = 2i \p \bp \phi$, then we have that
\begin{equation}\label{e:T5}
	\omega = 2i\frac{\p^2 \phi}{ \p w^i \p\bar{w}^j} dw_i \wedge d\bar{w}_j = \frac{\p^2 \phi}{ \p \xi^i \p\xi^j} d\xi^i \wedge d\theta^j.
\end{equation}
In this setting, the metric $g$ corresponding to $\omega$ is given on $\t \times T^n$ by
\begin{equation*}
	g = \phi_{ij}(\xi)d\xi^i d\xi^j + \phi_{ij}(\xi)d\theta^i d\theta^j,
\end{equation*}
and the moment map $\mu$ as a map $\mu: \t \times T^n \to \t^*$ is defined by the relation
		\begin{equation*}
		\langle \mu(\xi, \theta), b \rangle = \langle \nabla \phi(\xi), b \rangle
	\end{equation*}
for all $b \in \t$, where $\nabla \phi$ is the Euclidean gradient of $\phi$.
The $T^n$-invariance of $\phi$ implies that it depends only on $\xi$ when considered as a function on $\t \times T^n$
via \eqref{diffeoo}. Since $\omega$ is K\"ahler, we see from \eqref{e:T5} that the Hessian of $\phi$ is positive-definite so that $\phi$ itself is strictly convex.
In particular, $\nabla \phi$ is a diffeomorphism onto its image.
Using the identifications mentioned above, we view $\nabla \phi$ as a map from $\t$ into an open subset of $\t^*$.

\subsubsection{K\"ahler-Ricci solitons on toric manifolds}

Next we define what we mean by a shrinking K\"ahler-Ricci soliton in the toric category.
\begin{definition}
A complex $n$-dimensional shrinking K\"ahler-Ricci soliton $(M,\,g,\,X)$ with complex structure $J$ and K\"ahler form $\omega$ is \emph{toric} if $(M,\,\omega)$ is a toric K\"ahler
manifold as in Definition \ref{toricmanifold} and $JX$ lies in the Lie algebra $\t$ of the underlying real torus $T^{n}$ that acts on $M$. In particular, the zero set of $X$ is compact.
\end{definition}

It follows from \cite{carlos, sun2024} that $\pi_{1}(M)=0$, hence the induced real $T^n$-action is automatically Hamiltonian with respect to $\omega$.
Working on the dense orbit $\Cstarn \subset M$, the condition that a vector field $JY$ lies in $\t$ is equivalent to saying that in the coordinate system $(\xi^1,\ldots,\xi^n,\,\theta_{1},\ldots,\theta_{n})$
from \eqref{diffeoo}, there is a constant $b_Y=(b_{Y}^{1},\ldots,b_{Y}^{n})\in \R^n$ such that
\begin{equation}\label{eqnY4}
	JY =  b_Y^i \frac{\p}{\p\theta^i}\qquad\textrm{or equivalently,}\qquad Y =   b_Y^i \frac{\p}{\p\xi^i}.
\end{equation}
From Proposition \ref{propB6}, we know that $\mathcal{L}_{X}\omega=2i\p\bp(X\cdot\phi)$. In addition,
the function $X\cdot\phi$ on $\Cstarn$ can be written as $\langle b_X, \nabla \phi \rangle = b_X^j \frac{\p\phi}{\p\xi^j}$,
where $b_{X}\in\mathbb{R}^{n}$ corresponds to the soliton vector field $X$ via \eqref{eqnY4}.
These observations allow us to write the shrinking soliton equation \eqref{krseqn} as a real Monge-Amp\`ere equation for $\phi$ on $\R^n$.

\begin{prop}[{\cite[Proposition 2.6]{charlie}}]
Let $(M,\,g,\,X)$ be a toric shrinking gradient K\"ahler-Ricci soliton with K\"ahler form $\omega$. Then
there exists a unique smooth convex real-valued function $\phi$ defined on the dense orbit $\Cstarn\subset M$ such that $\omega=2i\partial\bar{\partial}\phi$
and
\begin{equation} \label{realMA}
	\det(\phi_{ij})=e^{-2\phi+\langle b_{X},\,\nabla\phi\rangle}.
\end{equation}
\end{prop}

A priori, the function $\phi$ is defined only up to addition of a linear function.
However, \eqref{realMA} provides a normalisation for $\phi$ which in turn provides a normalisation for $\nabla\phi$, the moment map of the action.
The next lemma shows that this normalisation coincides with that for the moment map as defined in \cite[Definition 5.16]{cds}.

\begin{lemma}\label{normal}
Let $(M,\,g,\,X)$ be a toric complete shrinking gradient K\"ahler-Ricci soliton with complex structure $J$ and K\"ahler form $\omega$
with soliton vector field $X=\nabla^{g}f$ for a smooth real-valued function $f:M\to\mathbb{R}$.
Let $\phi$ be given by Proposition \ref{propB6} and normalised by \eqref{realMA}, let $JY\in\mathfrak{t}$, and let $u_{Y}=\langle\nabla\phi,\,b_{Y}\rangle$ be the Hamiltonian potential of $JY$
with $b_{Y}$ as in \eqref{eqnY4} so that $\nabla^{g}u_{Y}=Y$. Then $\mathcal{L}_{JX}u_{Y}=0$ and $\Delta_{\omega}u_{Y}+u_{Y}-\frac{1}{2}Y\cdot f=0$.
\end{lemma}
\noindent To see the equivalence with \cite[Definition 5.16]{cds},
simply replace $Y$ with $JY$ in this latter definition as here we assume that $JY\in\mathfrak{t}$, contrary to the convention in \cite[Definition 5.16]{cds} where it is assumed that
$Y\in\mathfrak{t}$.

Given the normalisation \eqref{realMA}, the next lemma identifies the image of the moment map $\mu=\nabla\phi$.

\begin{lemma}[{\cite[Lemmas 4.4 and 4.5]{charlie}}]\label{note}
Let $(M,\,g,\,X)$ be a complete toric shrinking gradient K\"ahler-Ricci soliton, let $\{D_i\}$ be the prime $\Cstarn$-invariant divisors in $M$, and let $\Sigma \subset \t$ be the fan determined by
Proposition \ref{fann}. Let $\sigma_i\in\Sigma$ be the ray corresponding to $D_i$ with minimal generator $\nu_i \in \Gamma$.
\begin{enumerate}
  \item There is a distinguished Weil divisor representing the anticanonical class $-K_{M}$ given by
  \begin{equation*}
		-K_M = \sum_i D_i
	\end{equation*}
	whose associated polyhedron (cf.~\eqref{eqnB2}) is given by
 	\begin{equation}\label{xmas}
		P_{-K_M} = \left\{ x \: | \: \langle \nu_i, x \rangle \geq -1 \right\}
	\end{equation}
which is strongly convex and has full dimension in $\mathfrak{t}^{*}$. In particular, the origin lies in the interior of $P_{-K_{M}}$.
  \item If $\mu$ is the moment map for the induced real $T^n$-action normalised by \eqref{realMA}, then the image of $\mu$ is precisely $P_{-K_{M}}$.
\end{enumerate}
	\end{lemma}

\subsubsection{The weighted volume functional}\label{weighted}

As a result of Lemma \ref{normal}, we can now define the weighted volume functional.

\begin{definition}[{Weighted volume functional, \cite[Definition 5.16]{cds}}]\label{weightedvol}
Let $(M,\,g,\,X)$ be a complex $n$-dimensional toric shrinking gradient K\"ahler-Ricci soliton with K\"ahler form $\omega=2i\partial\bar{\partial}\phi$
 on the dense orbit with $\phi$ strictly convex with moment map $\mu=\nabla\phi$ normalised by \eqref{realMA}. Assume that the fixed point set of the torus is compact and define the open convex cone $$\Lambda_{\omega}:=\{Y\in\mathfrak{t}\:|\:\textrm{$\langle\mu,\,Y\rangle$ is proper and bounded below}\}\subseteq\mathfrak{t}.$$ Then the \emph{weighted volume functional} $\mathcal{F}_{\omega}:\Lambda_{\omega}\to\mathbb{R}$ is defined by
	\begin{equation*}
		\mathcal{F}_{\omega}(v) = \int_M e^{-\langle \mu,\,v \rangle} \omega^n.
	\end{equation*}	
\end{definition}

As the fixed point set of the torus is compact by definition, $\mathcal{F}_{\omega}$ is well-defined by the non-compact version of the Duistermaat-Heckman formula \cite{wu}
(see also \cite[Theorem A.3]{cds}). It is moreover strictly convex on $\Lambda_{\omega}$ \cite[Lemma 5.17(i)]{cds}, hence has at most one critical point in this set.
This leads to two important lemmas concerning the weighted volume functional in the toric category, the independence of $\Lambda_{\omega}$ and $\mathcal{F}_{\omega}$ from the choice of shrinking soliton $\omega$.

\begin{lemma}[{\cite[Lemma 2.26]{ccd}}]\label{one}
The open convex cone $\Lambda_{\omega}$ is independent of the choice of toric shrinking K\"ahler-Ricci soliton $\omega$ in Definition \ref{weightedvol}.
\end{lemma}

\begin{lemma}[{\cite[Lemma 2.27]{ccd}}]\label{two}
The weighted volume functional $\mathcal{F}_{\omega}$ is independent of the choice of toric shrinking K\"ahler-Ricci soliton $\omega$ in Definition \ref{weightedvol}. Moreover,
after identifying $\Lambda_{\omega}$ with a subset of $\mathbb{R}^{n}$ via \eqref{eqnY4}, $\mathcal{F}_{\omega}$
is given by $\mathcal{F}_{\omega}(v)=(2\pi)^n \int_{P_{-K_M}} e^{-\langle v,\,x \rangle }\,dx$, where $x=(x^1,\ldots,x^n)$ denotes coordinates on $\mathfrak{t}^{*}$ dual to
the coordinates $(\xi^{1},\ldots,\xi^{n})$ on $\t$ introduced in Section \ref{finito}.
\end{lemma}
\noindent Thus, we henceforth drop the subscript $\omega$ from $\mathcal{F}_{\omega}$ and $\Lambda_{\omega}$ when working in the toric category. The
functional $\mathcal{F}:\Lambda\to\mathbb{R}$ is in addition proper in this category \cite[Proof of Proposition 3.1]{charlie}, hence attains a unique critical point in $\Lambda$.
This critical point characterises the soliton vector field of a complete toric shrinking gradient K\"ahler-Ricci soliton.
\begin{theorem}[{\cite[Theorem 4.6]{charlie}, \cite[Theorem 1.1]{caoo}}]\label{thmB13}
Let $(M,\,g,\,X)$ be a complete toric shrinking gradient K\"ahler-Ricci soliton with complex structure $J$.
Then $JX\in\Lambda$ and $JX$ is the unique critical point of $\mathcal{F}$ in $\Lambda$.
\end{theorem}

Having established in Lemmas \ref{one} and \ref{two} that in the toric category the weighted volume functional $F$ and its domain $\Lambda$ are determined solely by the polytope $P_{-K_{M}}$ which itself, by Lemma \ref{note}, depends only on the torus action on $M$ (i.e., is independent of the choice of shrinking soliton), and having an explicit expression for $\mathcal{F}$ given by Lemma \ref{two}, after using the torus action to identify $P_{-K_{M}}$ via \eqref{xmas}, we can determine explicitly the soliton vector field of a hypothetical toric shrinking gradient K\"ahler-Ricci soliton on $M$. Indeed, in light of Lemma \ref{two}, the unique minimiser $b_{X}\in\mathfrak{t}\simeq\mathbb{R}^{n}$ is characterised by the fact that
\begin{equation*}
0=d_{b_{X}}\mathcal{F}(v)=\int_{P_{-K_{M}}}\langle x,\,v\rangle\,e^{-\langle b_{X},\,x\rangle}dx\qquad\textrm{for all $v\in\mathbb{R}^{n}$.}
\end{equation*}
The fact that $JX \in \Lambda$ is reflected in the following property of the constant $b_X \in \t$.

\begin{lemma}\label{normpolytope-tof}
Let $JX\in\Lambda$. Then there exists a compact subset $K\subset P_{-K_{M}}$ and a constant $C> 0$ such that
	\begin{equation*}
		 C^{-1}\langle x ,\, b_X \rangle \leq |x| \leq  C \langle x ,\, b_X \rangle
	\end{equation*}
for all $x \in P_{-K_{M}}\backslash K$.
\end{lemma}

\begin{proof}
We clearly have that $\langle x ,\, b_X \rangle \leq |b_X||x|$. As for the other inequality,
notice that by Proposition \ref{identify}, since $b_X$ lies in the interior of the dual cone $$\mathcal{C}(P_{-K_{M}})^{\vee}= \{ b \in \t \: | \: \langle x , \, b \rangle \geq 0 \textnormal{ for all } x \in \mathcal{C}(P_{-K_{M}}) \}$$ with $\mathcal{C}(P_{-K_{M}})$ the asymptotic cone of $P_{-K_{M}}$, it holds that $\langle x, \, b_X \rangle>0$ for all $x \in \mathcal{C}(P_{-K_{M}})$.
Now, $P_{-K_M}$ is obtained as a Minkowski sum of $\mathcal{C}(P_{-K_{M}})$ and a compact polytope $Q$
\cite[Theorem 1.2]{ziegler}. This means that every $x\in P_{-K_M}$ can be written as $x=y+z$ for some $y\in\mathcal{C}(P_{-K_{M}})$
and $z\in Q$. Let $\alpha:=\max_{q\,\in\,Q}|q|$. Then for $x\in P_{-K_{M}}$ with $|x|>2\alpha$, we have that
$$|y|=|x-z|\geq||x|-|z||=|x|-|z|\geq\frac{1}{2}|x|.$$ Consequently, with $\beta:=\min_{w\in S^{n-1}\,\cap\,\mathcal{C}(P_{-K_{M}})}\langle b_{X},\,w\rangle>0$,
we find that for all $x\in P_{-K_{M}}$ with $|x|>2\max\left\{\alpha,\,\frac{2\alpha|b_{X}|}{\beta}\right\}$,
\begin{equation*}
\begin{split}
\langle b_{X},\,x\rangle&=\langle b_{X},\,y\rangle+\langle b_{X},\,z\rangle
\geq\frac{\beta}{2}|x|-|b_{X}||z|\geq\frac{\beta}{2}|x|-|b_{X}|\alpha\geq\frac{\beta|x|}{4},
\end{split}
\end{equation*}
as desired.
\end{proof}

\subsubsection{The Legendre transform}

Let $M$ be a toric manifold of complex dimension $n$ endowed with a complete K\"ahler form $\omega$ invariant under the induced real $T^{n}$-action and with respect to which
this action is Hamiltonian. Write $\omega=2i\partial\bar{\partial}\phi$ on the dense orbit for $\phi$ strictly convex as in Proposition \ref{propB6}. Then
$\nabla\phi(\mathbb{R}^{n})$ is a Delzant polytope $P$. Recall that we have coordinates $\xi$ on $\mathbb{R}^{n}\simeq\mathfrak{t}$, $x$ on $P$, and $\theta$ on $T^{n}$.
Given any smooth and strictly convex function $\psi$ on $\R^n$ such that $\nabla\psi(\R^n)=P$, there exists a unique smooth and strictly convex function $u_\psi(x)$ on $P$ defined by
	\begin{equation*}
		\psi(\xi) + u_\psi(\nabla \psi) = \langle\nabla\psi,\,\xi \rangle.
	\end{equation*}
This process is reversible; that is to say, $\psi$ is the unique function satisfying
		\begin{equation*}
		\psi(\nabla u_\psi) + u_\psi(x) = \langle x ,\,\nabla u_\psi \rangle,
	\end{equation*}
where $\nabla$ now denotes the Euclidean gradient with respect to $x$. The function $u_{\psi}$ is called the \emph{Legendre transform of $\psi$} and is sometimes denoted by $L(\psi)(x)$.
Clearly $L(L(\psi))(\xi)=\psi(\xi)$. The Legendre transform $u$ of $\phi$ is called the \emph{symplectic potential} of $\omega$, as the metric $g$ associated to $\omega$ is given by
$$g=u_{ij}(x)dx^{i}dx^{j}+u^{ij}(x)d\theta^{i}d\theta^{j}.$$

The following will prove useful.

\begin{lemma}[{\cite[Lemma 2.10]{charlie}}]\label{growth}
Let $\phi$ be any smooth and strictly convex function on an open convex domain $\Omega'\subset\mathbb{R}^{n}$
and let $u=L(\phi)$ be the Legendre transform of $\phi$ defined on $(\nabla\phi)(\Omega')=:\Omega$. If $0\in\Omega$, then there exists a constant $C > 0$ such that
$$\phi(\xi)\geq C^{-1}|\xi|-C.$$ In particular, $\phi$ is proper and bounded from below.
\end{lemma}

If $\phi \in C^\infty(\R^n)$ solves \eqref{realMA}, then the Legendre transform $u=L(\phi)$ satisfies
\begin{equation}\label{realMA2}
		2 \left( \langle \nabla u, x \rangle - u(x) \right) - \log\det(u_{ij}(x)) = \langle b_X, x \rangle\qquad\textrm{on $\Pol$}.
	\end{equation}
To study K\"ahler-Ricci solitons on $M$ via \eqref{realMA2} on $\Pol$, we need to understand when a strictly convex function on a Delzant polytope
defines a symplectic potential, i.e., is induced from a K\"ahler metric on $M$ via the Legendre transform.
To this end, consider a Delzant polytope $P$ obtained as the image of the moment map of a toric K\"ahler manifold.
Let $F_i$, $i = 1, \dots, d$ denote the $(n-1)$-dimensional facets of $P$ with inward-pointing normal vector $\nu_i \in \Gamma$, normalised so that $\nu_i$
is the minimal generator of $\sigma_i = \R_+ \cdot \nu_i$ in $\Gamma$, and let $\ell_i(x) = \langle \nu_i, x \rangle$ so that $\overline{P}$ is defined by the system of
inequalities $\ell_i(x) \geq - a_i$, $i = 1, \dots, N$, $a_i\in \R$. Then there exists a canonical metric $\omega_P$ on $M$ \cite[Proposition 2.7]{charlie}, the symplectic potential
$u_{P}$ of which is given explicitly by the formula \cite{BGL, Guil}
\begin{equation*}
	u_P(x) = \frac{1}{2}\sum_{i=1}^d (\ell_i(x) + a_i) \log\left( \ell_i(x) + a_i \right).
\end{equation*}
In particular, the Legendre transform $\phi_{P}$ of $u_{P}$ will define the K\"ahler potential on the dense orbit of a globally defined K\"ahler metric $\omega_{P}$ on $M$ \cite{BGL, Guil}.
In general, we have the following necessary condition for a convex function on the polytope $P$ to be induced by a toric K\"ahler metric on $M$.

\begin{lemma}[{\cite{Ab1}, \cite{ACGT2}}]\label{boundaryy}
A convex function $u$ on $P$ defines a K\"ahler metric $\omega_u$ on $M$ only if $u$ has the form
\begin{equation*}
u = u_{P} + v,
\end{equation*}
where $v \in C^\infty(\overline{P})$ extends past $\partial P$ to all orders.
\end{lemma}
\noindent In the case that $P=\Pol$, we read from Lemma \ref{note}(ii) that $a_{i}=-1$ for all $i$. Thus, in this case, the canonical metric on $\Pol$ has symplectic potential
		\begin{equation*}	
		u_{\Pol} = \frac{1}{2} \sum_i (\ell_i(x) + 1) \log(\ell_i(x) + 1).
	\end{equation*}

\subsubsection{The $\hat{F}$-functional}

We next define the $\hat{F}$-functional on toric K\"ahler manifolds.

\begin{definition}\label{fhat}
Let $(M,\,\omega)$ be a (possibly non-compact) toric K\"ahler manifold with complex structure $J$ endowed with a real holomorphic vector field $X$ such that
$JX\in\Lambda_{\omega}$. Write $T^{n}$ for the torus acting on $M$, identify the dense orbit with $\mathbb{R}^{n}$, let $\xi=(\xi_{1},\ldots,\xi_{n})$ denote coordinates on $\mathbb{R}^{n}$,
let $b_{X}$ be as in \eqref{eqnY4}, and write $\omega=2i\partial\bar{\partial}\phi_{0}$
on the dense orbit as in Proposition \ref{propB6}. Let $P:=(\nabla\phi_{0})(\mathbb{R}^{n})$ denote the image of the moment map associated to $\omega$ and let
$x=(x_{1},\ldots,x_{n})$ denote coordinates on $P$. Let $\varphi\in C^{\infty}(M)$ be a smooth function on $M$ invariant under
the action of $T^{n}$ such that $\omega+i\partial\bar{\partial}\varphi>0$ and assume that:
\begin{enumerate}[label=\textnormal{(\alph*)}]
  \item There exists a $C^{1}$-path of smooth functions $(\varphi_{s})_{s\in[0,\,1]}\subset C^{\infty}(M)$ invariant under the action of $T^{n}$ such that
$\varphi_{0}=0$, $\varphi_{1}=\varphi$, $\omega+i\partial\bar{\partial}\varphi_{s}>0$,
and $(\nabla\phi_{s})(\mathbb{R}^{n})=P$ for all $s\in[0,\,1]$, where $\phi_{s}:=\phi_{0}+\frac{\varphi_{s}}{2}$.
\item $\int_{0}^{1}\int_{\mathbb{R}^{n}}|\dot{\phi}_{s}|\,e^{-\langle b_{X},\,\nabla\phi_{s}\rangle}\det(\phi_{s,\,ij})\,d\xi\,ds<+\infty$.
\end{enumerate}
Then we define
\begin{equation*}
\begin{split}
\hat{F}(\varphi):=2\int_{P}(L(\phi_{1})-L(\phi_{0}))\,e^{-\langle b_{X},\,x\rangle}dx.
\end{split}
\end{equation*}
\end{definition}

The existence of the path $(\varphi_{s})_{s\in[0,\,1]}$ satisfying conditions (a) and (b) is required so that $\hat{F}(\varphi)$ is well-defined. To see this, first note:
\begin{lemma}[\protect{\textnormal{\cite[Lemma 2.28]{ccd2}}}]\label{converge1}
Under the assumptions of Definition \ref{fhat}, let $u_{s}:=L(\phi_{s})$, $\omega_{s}=\omega+i\partial\bar{\partial}\varphi_{s}$, and
write $f_{s}:=f+\frac{X}{2}\cdot \varphi_{s}$ for the Hamiltonian potential of $JX$ with respect to $\omega_{s}$, where $f$ is the Hamiltonian potential of $JX$ with respect to
$\omega$. Then the following are equivalent.
\begin{enumerate}
\item $\int_{0}^{1}\int_{\mathbb{R}^{n}}|\dot{\phi}_{s}|\,e^{-\langle b_{X},\,\nabla\phi_{s}\rangle}\det(\phi_{s,\,ij})\,d\xi\,ds<+\infty$.
\item $\int_{0}^{1}\int_{P}|\dot{u}_{s}|\,e^{-\langle b_{X},\,x\rangle}\,dx\,ds<+\infty$.
\item $\int_{0}^{1}\int_{M}|\dot{\varphi}_{s}|\,e^{-f_{s}}\omega^{n}_{s}\,ds<+\infty$.
\end{enumerate}
In particular when this is the case, $|\hat{F}(\varphi)|<+\infty$.
\end{lemma}

Under an additional assumption on the path $(\varphi_{s})_{s\in[0,\,1]}$, we recover the well-known expression for the
$\hat{F}$-functional given in \cite[p.702]{ctz}.

\begin{lemma}[\protect{\textnormal{\cite[Lemma 2.29]{ccd2}}}]\label{converge2}
If one (and hence all) of the conditions of Lemma \ref{converge1} hold true and if in addition it holds true that
$\int_{0}^{1}\int_{M}|\dot{\varphi}_{s}|\,e^{-f}\omega^{n}\,ds<+\infty$, then
\begin{equation*}
\hat{F}(\varphi)=\int_0^1\int_M\dot{\varphi}_{s}\left(e^{-f}\omega^n-e^{-f_{s}}\omega_{s}^n\right)\wedge ds
-\int_M\varphi\,e^{-f}\omega^n.
\end{equation*}
\end{lemma}

\subsubsection{Integrability and independence of the path}

In light of conditions (a) and (b) of Definition \ref{fhat} required to define the $\hat{F}$-functional, it remains to identify sufficient
conditions for the moment polytope to remain unchanged under a path of K\"ahler metrics and for each
summand in the integral of $\hat{F}$ to be finite. This will be important for achieving an a priori $C^{0}$-bound along our continuity path.

To this end, suppose that $(M,\,\omega)$ is a toric K\"ahler manifold,
i.e., $(M,\,\omega)$ is K\"ahler with K\"ahler form $\omega$ with respect to a complex structure $J$, endowed with
the holomorphic action of a complex torus of the same complex dimension as $(M, \, J)$ whose underlying real torus $T$ induces a Hamiltonian action,
and let $JX\in\mathfrak{t}$. Via \eqref{eqnY4}, we can identify $X$ with an element $b_{X}\in\mathbb{R}^{n}\simeq\mathfrak{t}$.
Using Proposition \ref{propB6}, we can also write $\omega=2i\partial\bar{\partial}\phi_{0}$ on the dense orbit for some strictly convex function
$\phi_{0}:\mathbb{R}^{n}\to\mathbb{R}$. Assume that:
\begin{itemize}
\item $JX\in\Lambda_{\omega}$ so that the Hamiltonian potential $f$ of $JX$ is proper and bounded from below.
  \item There exists a smooth bounded real-valued function $F$ on $M$ so that the Ricci form $\rho_{\omega}$ of $\omega$ satisfies
$\rho_\omega + \frac{1}{2}\mathcal{L}_X\omega - \omega = i \p\bp F$.
\end{itemize}
The equation in the second bullet point reads as
\begin{equation*}
\left(F+\log\det(\phi_{0,\,ij}) - \langle \nabla \phi_0, b_X \rangle + 2\phi_0\right)_{ij}=0\qquad\textrm{on $\mathfrak{t}\simeq\mathbb{R}^{n}$}
\end{equation*}
so that
$$F=-\log\det(\phi_{0,\,ij}) + \langle \nabla \phi_0, b_X \rangle - 2\phi_0+a(\xi)\qquad\textrm{on $\mathbb{R}^{n}$}$$
for some {affine} function $a(\xi)$ defined on $\mathbb{R}^{n}$. By considering
$2\phi_0+a+\langle\nabla a,\,b_X \rangle$, we can therefore assume that
\begin{equation*}
F=-\log\det(\phi_{0,\,ij})+\langle\nabla\phi_0,\,b_X \rangle-2\phi_0\qquad\textrm{on $\mathbb{R}^{n}$}.
\end{equation*}
We then have the following lemma.

\begin{lemma}[\protect{\textnormal{\cite[Lemma 2.30]{ccd2}}}]\label{expression}
Under the above assumptions, let $\varphi \in C^\infty(M)$ be a torus-invariant smooth real-valued function on $M$ such that
$\omega_\varphi:=\omega +  i \p \bp \varphi > 0$ and {$\sup_M|X\cdot\varphi| < \infty$}.
Define $\phi:=\phi_{0}+\frac{1}{2}\varphi$ so that $\omega + i \p\bp \varphi = 2i\p\bp \phi$ on the dense orbit. Then:
\begin{enumerate}[label=\textnormal{(\roman*)}]
\item The image of the moment map $\mu_{\omega_{\varphi}}:M \to \t^*$ with respect to $\omega_\varphi$ defined by the Euclidean gradient $\nabla\phi: \R^n \to \R^n$ is equal to $P_{-K_{M}}$.
In particular, $0\in\operatorname{int}\left(\mu_{\omega_{\varphi}}(M)\right)$.
\item $\int_{P}|L(\phi_0)|\,e^{-\langle b_{X},\,x\rangle}dx<+\infty$.
\end{enumerate}
\end{lemma}

\section{Proof of Theorem \ref{mainthm}\MakeLowercase{(ii)--(v)}: Setup of the complex Monge-Amp\`ere equation}\label{sec-construction-back-metric}

In this section, we set up the complex Monge-Amp\`ere equation whose solution will give us the shrinking soliton we desire. Our setup is as in Theorem \ref{mainthm}, namely the pair $(C_0,\,g_{0})$ is a toric K\"ahler cone with K\"ahler cone metric $g_{0}$, complex structure $J_{0}$, radial function $r$, and K\"ahler form $\omega_{0}=\frac{i}{2}\partial\bar{\partial}r^{2}$. We have a quasi-projective equivariant resolution of $C_0$ (with respect to the holomorphic torus action) $\pi:M\to C_0$ with exceptional set $E$ such that $-K_{M}$ is $\pi$-ample. We write $J$ for the complex structure on $M$ and $\mathfrak{t}$ for the Lie algebra of the real torus $T^{n}$ acting equivariantly on $\pi:M\to C_{0}$. We first identify the soliton vector field of any hypothetical toric shrinking K\"ahler-Ricci soliton on $M$.

\begin{lemma}\label{vector-field}
There exists a unique complete real holomorphic vector field $JX\in\mathfrak{t}$ such that $X$ is the soliton vector field of any complete toric shrinking gradient K\"ahler-Ricci soliton $g$ on $M$.
Moreover, if $g$ has quadratic curvature decay, then $g_{0}(d\pi(X),\,r\partial_{r})>0$.
\end{lemma}

\begin{proof}
The existence and uniqueness of $JX$ follows from \cite[Theorem A]{charlie}. Regarding the last assertion, we know
that $JX$ lies in the domain $\Lambda\subseteq\mathfrak{t}$ of the weighted volume functional of $g$. This is a non-empty open convex cone in $\mathfrak{t}$ that comprises
those vector fields in $\mathfrak{t}$ that admit a Hamiltonian potential with respect to $g$ that is proper and bounded from below. By \cite[Lemma 2.26]{ccd}, toricity of $g$ implies
that $\Lambda$ is independent of the choice of toric shrinking soliton. Assuming that $g$ has quadratic curvature decay, the soliton
$(M,\,g)$ has a tangent cone at infinity; cf.~\cite[Theorem 3.8]{cds} for the analogous statement in the expanding case. \cite[Theorem A.10]{cds} then stipulates that not only is $\Lambda$ independent of the choice of toric shrinking K\"ahler-Ricci soliton metric, it is furthermore independent of the choice of toric K\"ahler metric on $M$ asymptotic to a K\"ahler cone at infinity in the $C^{0}$-sense with Reeb vector field lying in $\mathfrak{t}$; cf.~\cite[(A-6)]{cds}. The proof of this theorem then allows us to identify $\Lambda$ as $\Lambda=\{Y\in\mathfrak{t}\,|\,g_{0}(r\partial_{r},\,-d\pi(JY))>0\}$.
Since $JX\in\Lambda$, the last assertion of the lemma follows.
\end{proof}

Given the above lemma, we can now perform a Type I deformation of $(C_{0},\,g_{0})$ \cite[Theorem A]{taka} (see also \cite[Appendix II]{conlon999}) to obtain a toric K\"ahler cone metric
on $C_{0}$ with radial coordinate $r$ satisfying $d\pi(X)=r\partial_{r}$ \cite[Lemma 2.2]{frank}. By abuse of notation, we still denote this K\"ahler cone metric by $g_{0}$
and the corresponding K\"ahler form by $\omega_{0}$. This proves Theorem \ref{mainthm}(ii).

With the model metric at infinity now determined, we construct a suitable background metric as demonstrated in the following proposition. This gives Theorem \ref{mainthm}(iii)--(v).
\begin{prop}\label{mainprop}
\begin{enumerate}
\item  There exists a complete K\"ahler metric $\omega$ on $M$ invariant under the action of $T^{n}$
such that outside a compact subset of $M$ containing $E$,
$\omega=\pi^{*}(\omega_{0}+\rho_{\omega_{0}})$. In particular, $\pi_{*}\omega-\omega_{0}=O(r^{-2})$ with $g_{0}$-derivatives.
In addition, there exists a smooth real-valued torus-invariant function $F=c_{0}-\frac{s_{\omega_{0}}}{2}+O(r^{-4})$ with $g_{0}$-derivatives such that
\begin{equation}\label{hutchins}
i\partial\bar{\partial}F=\rho_{\omega}+\frac{1}{2}\mathcal{L}_{X}\omega-\omega.
\end{equation}
Here, $c_{0}\in\mathbb{R}$, $\rho_{\omega_{0}}$ (respectively $\rho_{\omega}$) denotes the Ricci form of $\omega_{0}$ (resp.~$\omega$), and $s_{\omega_{0}}$ denotes the scalar curvature of $\omega_{0}$.
\item There exists a unique torus-invariant real-valued function $f\in C^{\infty}(M)$ with
$-\omega\lrcorner JX=df$ such that outside a compact subset of $M$ containing $E$,
$f=\pi^{*}\left(\frac{r^{2}}{2}-n\right)$ and
\begin{equation}\label{normal12}
\Delta_{\omega}f+f-\frac{X}{2}\cdot f=-\frac{X}{2}\cdot F.
\end{equation}
In particular,
\begin{equation}\label{normal2}
\Delta_{\omega}f+f-\frac{X}{2}\cdot f=O(r^{-2})\quad\textrm{with $g_{0}$-derivatives,}
\end{equation}
and $f\to+\infty$ as $r\to+\infty$, hence is proper. Moreover, there exists $C>0$ such that
\begin{equation}\label{bds-cov-der-f}
|X\cdot f-2f|\leq C\qquad\textrm{and}\qquad|\nabla^{g,\,k}(\nabla^{g,\,2}f-g)|_{g}\leq \frac{C_{k}}{(f+C)^{1+\frac{k}{2}}}\qquad\textrm{for all $k\geq 0$}.
\end{equation}
  \item Every shrinking gradient K\"ahler-Ricci soliton on $M$ with soliton vector field $X$ is of the form
  $\omega+i\partial\bar{\partial}\varphi$ for some smooth real-valued function $\varphi\in C^{\infty}(M)$ with $\omega+i\partial\bar{\partial}\varphi>0$
  satisfying the complex Monge-Amp\`ere equation
\begin{equation}\label{cmaa}
(\omega+i\partial\bar{\partial}\varphi)^{n}=e^{F+\frac{X}{2}\cdot\varphi-\varphi}\omega^{n},
\end{equation}
where $F$ is as in part (i). Moreover, if the shrinking soliton is $T^{n}$-invariant, then so is $\varphi$.
\end{enumerate}
\end{prop}
Note that since the torus action on $M$ contains a fixed point, $M$ is simply connected \cite[Theorem 12.1.10]{cox},
a necessary condition for the existence of a shrinking gradient K\"ahler-Ricci soliton \cite{carlos, sun2024}.

\begin{proof}[Proof of Proposition \ref{mainprop}]
\begin{enumerate}
\item As $M$ is quasi-projective and $\pi$ is $-K_{M}$-ample by assumption, the construction of the background metric follows by arguing first as in \cite[p.307]{cds}
with $-K_{M}$ in place of $K_{M}$ to show that \cite[(4-4)]{cds} holds true, and then as in \cite[Proposition 3.1]{con-der}, with $T^{n}$ in place of $T^{k}$, $c=1$, and again
with $-K_{M}$ in place of $K_{M}$. The existence of $F$ follows from the computations on \cite[p.319]{con-der} (see in particular \cite[(3.5)--(3.8)]{con-der}), together with an application of Lemma \ref{nice}.
  \item As we have already remarked, $M$ is simply connected, and so there exists a smooth real-valued function $f\in C^{\infty}(M)$, defined up to a constant, with $-\omega\lrcorner JX=df$.
Any such choice of $f$ is invariant under the action of $T^{n}$ by virtue of the fact that $\omega\lrcorner JX$ is invariant under this action.
Next, notice that $-\omega_{0}\lrcorner J_{0}r\partial_{r}=d\left(\frac{r^{2}}{2}\right)$, where recall $J_{0}$ is the complex structure on $C_{0}$,
and $\rho_{\omega_{0}}\lrcorner J_{0}r\partial_{r}=0$ because, as is well-known, the Ricci form of a K\"ahler cone metric is a basic $(1,\,1)$-form \cite[Section 2.4]{Boyer}.
Henceforth suppressing the pullback by $\pi$, we therefore see from the form of $\omega$ given in part (i) that on the complement of a compact subset $K\subseteq M$ containing $E$,
$$df=-\omega\lrcorner JX=-\left(\omega_{0}+\rho_{\omega_{0}}\right)\lrcorner JX=d\left(\frac{r^{2}}{2}\right),$$
so that $f$ differs from $\frac{r^{2}}{2}$ by a constant on this set, meaning that $f=\frac{r^{2}}{2}+\operatorname{const.}$
on $M\setminus K$. Normalise $f$ so that this constant is equal to $-n$. Then $f=\pi^{*}\left(\frac{r^{2}}{2}-n\right)$ outside a compact set.
What remains to show is that with respect to this normalisation, \eqref{normal2} holds true.

To this end, let $F$ be the smooth function from part (i) satisfying
$$ i \p \bp F = \rho_{\omega} + \frac{1}{2}\mathcal{L}_X \omega - \omega.$$
Using the $JX$-invariance of $F$ and $f$,
contract this equation with $X^{1,\,0}:=\frac{1}{2}(X-iJX)$ and use the Bochner formula to derive that
$$i\bar{\partial}\left(\Delta_{\omega}f-\frac{X}{2}\cdot f+f+\frac{X}{2}\cdot F\right)=0.$$
As a real-valued holomorphic function, we must have that $\Delta_{\omega}f-\frac{X}{2}\cdot f+f+\frac{X}{2}\cdot F$ is
constant on $M$. In light of the fact that on $M\setminus K$,
\begin{equation}\label{bonjour}
\begin{split}
\Delta_{\omega}f-\frac{X}{2}\cdot f+f&=(\Delta_{\omega}-\Delta_{\omega_{0}})f+\Delta_{\omega_{0}}f-\frac{X}{2}\cdot f+f\\
&=O(r^{-2})+\underbrace{\Delta_{\omega_{0}}\left(\frac{r^{2}}{2}-n\right)-\frac{r}{2}\frac{\partial}{\partial r}\left(\frac{r^{2}}{2}-n\right)+\left(\frac{r^{2}}{2}-n\right)}_{=\,0}\\
&=O(r^{-2})
\end{split}
\end{equation}
and $\frac{X}{2}\cdot F=O(r^{-2})=O(f^{-1})$, this constant must be zero so that globally on $M$, so that
\begin{equation*}
\Delta_{\omega}f+f-\frac{X}{2}\cdot f=-\frac{X}{2}\cdot F=O(f^{-1}).
 \end{equation*}
This proves \eqref{normal12} and \eqref{normal2}. Finally, as $f=\pi^{*}\left(\frac{r^{2}}{2}-n\right)$ outside a compact subset of $M$, the first estimate of
\eqref{bds-cov-der-f} is clear. As for the second, part (i) gives us that
$$|\nabla^{g,\,k}\left(\Ric(g)+\nabla^{g,2}f-g\right)|_g=O\left(f^{-2-\frac{k}{2}}\right)\qquad\textrm{for all $k\geq 0$}$$
outside a compact subset of $M$. But then $g$ is asymptotically conical at rate $-2$ by part (i) again, and so $|\nabla^{g,\,k}\Ric(g)|_g=O\left(f^{-1-\frac{k}{2}}\right)$ for all $k\geq 0$.

\item The fact that the shrinking soliton takes the form as stated follows from \eqref{hutchins} and the defining equation of a shrinking gradient K\"ahler-Ricci soliton.
One then derives the complex Monge-Amp\`ere equation as in the proof of \cite[Proposition 3.2]{con-der} using Lemma \ref{nice}. Finally, one can guarantee the $T^{n}$-invariance of
$\varphi$ by averaging over the corresponding action. 
\end{enumerate}
\end{proof}

In summary, we want to solve the complex Monge-Amp\`ere equation
\begin{equation*}
(\omega+i\partial\bar{\partial}\varphi)^{n}=e^{F+\frac{X}{2}\cdot\varphi-\varphi}\omega^{n},\quad\textrm{$\varphi\in C^{\infty}(M)$ torus-invariant},\quad\omega+i\partial\bar{\partial}\varphi>0.\\
\end{equation*}
A strategy to solve this equation is given by considering the Aubin continuity path:
\begin{equation}\label{ast-t-bis}
(\omega+i\partial\bar{\partial}\varphi_{t})^{n}=e^{F+\frac{X}{2}\cdot\varphi_{t}-t\varphi_{t}}\omega^{n},\quad\textrm{$\varphi_{t}\in C^{\infty}(M)$ torus-invariant},\quad\omega+i\partial\bar{\partial}\varphi_{t}>0,\quad t\in[0,\,1].\tag{$\ast_{t}$}
\end{equation}
The equation corresponding to $t=0$ we consider is given by
\begin{equation}\label{ast-0}
\left\{
\begin{array}{rl}
(\omega+i\partial\bar{\partial}\psi)^{n}=e^{F+\frac{X}{2}\cdot\psi}\omega^{n},&\quad\textrm{$\psi\in C^{\infty}(M)$ torus-invariant},\quad\omega+i\partial\bar{\partial}\psi>0,\\
\int_{M}e^{F-f}\omega^{n}=\int_{M}e^{-f}\omega^{n}. &
\end{array} \right.\tag{$\ast_{0}$}
\end{equation}
Notice that we introduce the integral condition to determine $c_{0}$ uniquely.
This equation we will solve by the continuity method, the particular
path of which will be introduced in Section \ref{continuitie}. This will yield Theorem \ref{mainthm2}.
Beforehand however, we prove some analytic results regarding the metric $\omega$ and those metrics that are asymptotic to it, beginning with a weighted Poincar\'e inequality.

\section{Poincar\'e inequality}\label{sec-poin-inequ}

In this section, we prove a Poincar\'e inequality
(in a general setting) that will be used in Proposition \ref{prop-a-priori-ene-est} to establish an a priori weighted $L^{p}$-estimate along the
continuity path that we are considering.

Our set-up for our Poincar\'e inequality is as follows. We consider a non-compact connected Riemannian manifold $(M,\,g)$ endowed with a smooth real-valued function $f:M\to\mathbb{R}$.
Recall our notation for the drift Laplacian $\Delta_{f}(\,\cdot\,):=\Delta_{g}(\,\cdot\,)-\nabla^g_{\nabla^{g}f}(\,\cdot\,)$ and its corresponding properties from Section \ref{metricmeasure}.
We write $d\mu$ for the volume form of $g$ and $d\mu_{f}$ for the weighted measure $e^{-f}d\mu$.
We assume that the triple $(M,\,g,\,f)$ has the following properties.
\begin{enumerate}
  \item $f$ is proper and bounded from below.
  \item The zero set of the vector field $X:=\nabla^{g}f$ is compact.
  \item There exists constants $a,\,b>0$ such that $|\nabla^{g}f|^2_{g}\leq af+b$ everywhere on $M$.
  \item $\Delta_{f}f=-2f-G$ for a smooth bounded real-valued function $G:M\to\mathbb{R}$.
  \item The weighted volume $\int_{M}d\mu_{f}$ is finite.
  \item The sublevel sets of $f$ are connected.
\end{enumerate}
By Proposition \ref{mainprop}(i)--(ii), these properties clearly hold in our setting.

Given this set-up, we work with the Lebesgue and Sobolev spaces $L^{p}(d\mu_{f})$ and $W^{1,\,p}(d\mu_{f})$ on $M$ respectively, defined in the obvious way for $p\geq1$. We denote
by $$\fint_{M}u\,d\mu_{f}:=\frac{1}{\int_{M}d\mu_{f}}\int_{M}u\,d\mu_{f}\qquad\textrm{for all $u\in L^{p}(d\mu_{f})$.}$$
By H\"older's inequality and property (v), the integral $\fint_{M}u\,d\mu_{f}$ is finite.

\begin{prop}[Poincar\'e inequality]\label{poincare}
Under assumptions (i)--(v) above, for all $p\geq 1$, there exists a constant $C(p)>0$ such that
$$\left\|u-\fint_{M}u\,d\mu_{f}\right\|_{L^{p}(d\mu_{f})}\leq C(p)\|\nabla^{g}u\|_{L^{p}(d\mu_{f})}\qquad\textrm{for all $u\in W^{1,\,p}(d\mu_{f})$}.$$
\end{prop}

\begin{proof}
We first prove an $L^{1}$-Poincar\'e inequality by invoking \cite{Bak-Bar-Cat-Gui}. Since the setting of \cite{Bak-Bar-Cat-Gui} is that of Euclidean space, for clarity we reproduce their proof here. 

Let $R\geq 1+\frac{1}{2}\max_M|G|$ be a positive constant such that the zero set of $X$ is compactly contained in the sublevel set $\{f<R\}$, so that $\{f=s\}$ is a smooth compact hypersurface for all $s\geq R$. Choosing such an $R$ is possible because of properties (i) and (ii) above. Next, let $\chi_R$ be a smooth cut-off function with $\chi_R=1$ on $\{f\leq R\}$ and $\chi_R=0$ on $\{f\geq 2R\}$. Then properties (iii) and (iv) above, together with the choice of $R$ ensures that the function $\rho_R:=(1-\chi_R)(f-\frac{1}{2}\max_M|G|)+\chi_R$ satisfies
\begin{equation}\label{prop-ad-hoc-rho}
\Delta_f\rho_R\leq -2\rho_R+b\chi_{\{f\,\leq\, 2R\}},\qquad \rho_R\geq 1,\qquad |\nabla^g\rho_R|^2_g\leq C\rho_R,
\end{equation}
for some constants $b,\,C>0$.

Let $\psi$ be a smooth compactly supported function on $M$. Then following the proof of \cite[Theorem $1.5$]{Bak-Bar-Cat-Gui} and invoking \eqref{prop-ad-hoc-rho}, we see that for all $c\in\R$,
\begin{equation*}
\begin{split}
\int_M|\psi-c|\,d\mu_f&\leq \int_M|\psi-c|\frac{\left(-\Delta_f\rho_R\right)}{2\rho_R}\,d\mu_f+\frac{b}{2}\int_M|\psi-c|\chi_{\{f\,\leq \,2R\}}\,d\mu_f\\
&= \frac{1}{2}\int_M\left\langle\nabla^g\left(\frac{|\psi-c|}{\rho_R}\right),\nabla^g\rho_R\right\rangle_g\,d\mu_f+\frac{b}{2}\int_M|\psi-c|\chi_{\{f\,\leq\, 2R\}}\,d\mu_f\\
&= \frac{1}{2}\int_M\left\langle\nabla^g|\psi-c|,\,\frac{\nabla^g\rho_R}{\rho_{R}}\right\rangle_g\,d\mu_f
-\frac{1}{2}\int_M |\psi-c|\frac{|\nabla^g\rho_R|^{2}_{g}}{\rho_{R}^{2}}\,d\mu_f\\
&\qquad+\frac{b}{2}\int_M|\psi-c|\chi_{\{f\,\leq\, 2R\}}\,d\mu_f\\
&\leq \frac{1}{2}\int_M|\nabla^g(\psi-c)|_g\frac{|\nabla^g\rho_R|_g}{\rho_R}\,d\mu_f+\frac{b}{2}\int_M|\psi-c|\chi_{\{f\,\leq\, 2R\}}\,d\mu_f\\
&\leq \frac{\sqrt{C}}{2}\int_M|\nabla^g(\psi-c)|_g\,d\mu_f+\frac{b}{2}\int_{\{f\,\leq\, 2R\}}|\psi-c|\,d\mu_f.
\end{split}
\end{equation*}
Here we have used the Cauchy-Schwarz inequality followed by Kato's inequality in the fifth line.

Since the domain $\{f\leq 2R\}$ is connected by (iv) and has smooth boundary,
it satisfies an $L^1$-Poincar\'e inequality with Poincar\'e constant $C_{P}(R)>0$. Thus, there exists a constant $C_P(R)>0$ 
such that for
$$c:=\frac{1}{\int_{\{f\,\leq\,2R\}}d\mu}\int_{\{f\,\leq\,2R\}} \psi\,d\mu,$$
we have
 \begin{equation*}
\begin{split}
\int_M|\psi-c|\,d\mu_f&\leq \frac{\sqrt{C}}{2}\int_M|\nabla^g(\psi-c)|_g\,d\mu_f+\frac{b\,C_P(R)}{2}\int_{\{f\,\leq\, 2R\}}|\nabla^g\psi|\,d\mu_f\\
&\leq \left(\frac{\sqrt{C}}{2}+\frac{b\,C_P(R)}{2}\right)\int_M|\nabla^g\psi|_g\,d\mu_f.
\end{split}
\end{equation*}
In particular, there exists $C_P>0$ such that for all compactly supported functions $\psi$,
\begin{equation*}
\int_M|\psi-M_{\mu}\psi|\,d\mu_f=\inf_{c\,\in\,\R}\int_M|\psi-c|\,d\mu_f\leq C_P\int_M|\nabla^g\psi|_g\,d\mu_f,
\end{equation*}
where $M_{\mu}\psi$ denotes the median of $\psi$, i.e., a number $a\in\mathbb{R}$ such that
$$\frac{1}{\int_{M}d\mu_{f}}\int_{\{\psi\,\geq\,a\}}d\mu_{f}\geq\frac{1}{2}
\qquad\textrm{and}\qquad\frac{1}{\int_{M}d\mu_{f}}\int_{\{\psi\,\leq\,a\}}d\mu_{f}\geq\frac{1}{2}.$$
Note that in our situation, a median exists and is unique \cite[Exercise 4.26]{balestro}, and we have used the fact that such a number minimises the $L^{1}$-norm.
From \cite[Lemma $2.1$]{milman}, we conclude that
\begin{equation*}
\int_M\left|\psi-\fint_M\psi\,d\mu_f\right|\,d\mu_f\leq 2\int_M\left|\psi-M_{\mu}f\right|\,d\mu_f\leq 2C_P\int_M|\nabla^g\psi|_g\,d\mu_f,
\end{equation*}
giving the desired $L^{1}$-Poincar\'e inequality. One can now invoke \cite[Proposition 2.5]{milman} to assert that the $L^{p}$-Poincar\'e inequality holds true for all $p>1$.
\end{proof}

\section{Linear theory}\label{linear-theory-section}

Recall the setup of Theorem \ref{mainthm}: $\pi:(M,\,\omega)\to (C_{0},\,\omega_{0})$ is a torus-equivariant resolution with exceptional set $E$ of the K\"ahler cone $(C_{0},\,\omega_{0})$ with radial function $r$, and
$\omega$ is as in Theorem \ref{mainthm}(iii). We write $g$ for the K\"ahler metric associated to $\omega$. Then in particular, $\pi_{*}g-g_{0}=O(r^{-2})$ with $g_{0}$-derivatives.
Recall that $X$ is the soliton vector field on $M$ with $d\pi(X)=r\partial_{r}$ and $X=\nabla^{g}f$ for a smooth real-valued function $f:M\to\mathbb{R}$ satisfying the conditions
of Proposition \ref{mainprop}(ii), and $J$ is the complex structure on $M$. The vector field $JX$ is then holomorphic and $g$-Killing on $M$.
Throughout, we identify $C_{0}$ with $M\setminus E$ via $\pi$.

In this section, we introduce the function spaces and set up the linear theory for K\"ahler metrics on $M$ asymptotic to $g_{0}$.
Openness along the continuity path \eqref{ast-0}
will then follow from the invertibility of the drift Laplacian between our function spaces. Although Theorem \ref{mainthm2} holds true for torus-invariant functions, in order
to remain as broad as possible, we present the linear theory under minimal assumptions,
namely for $JX$-invariant functions.

\subsection{Main setting and basic properties}\label{setup}
Let $(\varphi^{X}_{\tau})_{\tau\,<\,0}$ denote the flow of $\frac{X}{2(-\tau)}$ such that\linebreak $\varphi^{X}_{\tau}\big|_{\tau\,=\,-1}=\Id_M$. Let $\tilde{g}$ be any $JX$-invariant K\"ahler metric on $M$ with K\"ahler form $\tilde{\omega}$. For $\tau<0$, define $\tilde{\omega}(\tau):=(-\tau)(\varphi^{X}_{\tau})^*\tilde{\omega}$, let $\tilde{g}(\tau)$ denote the corresponding K\"ahler metric, and
set $g(\tau):=(-\tau)(\varphi^{X}_{\tau})^*g$. For $(x,\,\tau)\in M\times(-\infty,0)$, define the parabolic neighborhood $P_r(x,\tau)$ of $(x,\,\tau)$ of radius $r>0$ by
\begin{equation*}
P_r(x,\tau):=B_{\tilde{g}}(x,r)\times(\tau-r^2,\tau].
\end{equation*}
Let $\rho$ be any strictly positive, $JX$-invariant smooth function on $M$ equal to $f$ outside a compact subset and bounded below by $1$, and for
$x\in M$, define $r_x^2:=\iota_0^2\rho$, where $\iota_{0}>0$ is chosen sufficiently small so that for every $x\in M$,
$0<r_{x}\leq\frac{1}{2}\operatorname{inj}_{x}(g)$. This is possibly because $g$ is asymptotically conical. Our assumptions on $\tilde{\omega}$ (or equivalently on $\tilde{g}$) are then that
\begin{equation}\label{hyp-basic-ass}
\sup_{x\,\in\,M}\left[r_x^{i+2j}\cdot\sup_{P_{r_x}(x)}\left|\nabla^{g,\,i}\left(\partial_{\tau}^{(j)}\tilde{\omega}(\tau)\right)\right|_{g}\right]<\infty\quad\textrm{for all $i,\,j\geq 0$}\qquad\textrm{and}\qquad \lim_{f(x)\,\rightarrow\,+\infty}|\tilde{g}-g|_g(x)=0.
\end{equation}
This latter assumption implies that $\tilde{\omega}(\tau)\to\pi^{*}\omega_{0}$ pointwise on $M\setminus E$ as $\tau\to0^{-}$ because $g$ is asymptotically conical.
We have the following useful properties.
\begin{lemma}\label{lemma-basic-equiv-ass}
\begin{enumerate}
\item The K\"ahler form $\omega$ satisfies \eqref{hyp-basic-ass}.
\item Condition \eqref{hyp-basic-ass} is equivalent to the following:
\begin{equation}\label{hyp-basic-ass-II}
\sup_{x\,\in\,M}\left[r_x^{i+2j}\cdot\sup_{P_{r_x}(x)}\left|\nabla^{g(\tau),\,i}\left(\partial_{\tau}^{(j)}\tilde{\omega}\right)\right|_{g(\tau)}\right]<\infty
\quad\textrm{for all $i,\,j\geq 0$}\qquad\textrm{and}\qquad \lim_{f(x)\rightarrow +\infty}|\tilde{g}-g|_g(x)=0.
\end{equation}
\item Condition \eqref{hyp-basic-ass} implies that:
\begin{equation}\label{asy-cond-diff-spatial}
\sup_M\rho^{\frac{i}{2}+1}\left|\nabla^{g,\,i}\left(\tilde{\omega}-\omega\right)\right|_{g}<\infty\qquad\textrm{for all $i\geq 0$.}
\end{equation}
\end{enumerate}
\end{lemma}

\begin{remark}
Part (iii) of this lemma asserts that \eqref{hyp-basic-ass} gives a convergence rate of $-2$ with derivatives for the difference $\tilde{\omega}-\omega$ of the two K\"ahler forms.
\end{remark}

\begin{proof}[Proof of Lemma \ref{lemma-basic-equiv-ass}]
We henceforth drop the dependency of the flow $(\varphi_{\tau}^X)_{\tau\,<\,0}$ on $X$. Recall from Proposition \ref{mainprop}(i) that $\omega=\omega_0+\rho_{\omega_0}$ outside a fixed compact set $K$ of $M$ containing $E$.
\begin{enumerate}
\item  Since
  $(-\tau)\varphi_{\tau}^*\omega_0=\omega_0$ on $C_{0}$ for all $\tau<0$ so that $\varphi_{\tau}^*\rho_{\omega_0}=\rho_{\varphi_{\tau}^*\omega_0}=\rho_{(-\tau)^{-1}\omega_0}=\rho_{\omega_0}$ on
  $C_{0}$ for all $\tau<0$, we see that $\omega(\tau):=(-\tau)\varphi_{\tau}^*\omega=\omega_0+(-\tau)\rho_{\omega_0}$ on $P_{r_x}(x)$ for all $x\in M$ with $r(x)$ sufficiently large.
  Next notice that for all such $x$,
  $\partial_{\tau}^{(j)}\omega(\tau)=0$ for $j\geq 2$ and $\partial_{\tau}\omega(\tau)=-\rho_{\omega_0}$ on $P_{r_x}(x)$, the latter being time-independent and decaying quadratically with derivatives at (spatial) infinity. From these observations, (i) follows.

\item We begin with the following claim.
\begin{claim}\label{claim-cov-der-g}
For all $i\geq 0$, there exists $C_i>0$ such that for all $x\in M$,
\begin{equation*}
r_x^{i}\cdot\sup_{P_{r_x}(x)}|(\nabla^{g})^i(g(\tau)-g)|_{g}(x)\leq C_{i}.
\end{equation*}
\end{claim}

\begin{proof}[Proof of Claim \ref{claim-cov-der-g}]
Since $g-g_{0}=O(r^{-2})$ with $g_{0}$-derivatives by Proposition \ref{mainprop}(i), the computations on \cite[p.301]{cds}
tell us that for all $i\geq0$, there exists a constant $C_{i}>0$ such that for all $x\in M\setminus E$,
$$\textrm{$|(\nabla^{g_{0}})^i(g(\tau)-g_{0})|_{g_{0}}(x)\leq C_{i}(-\tau)r(x)^{-2-i}$ for all $\tau<0$}.$$
Because $g$ and $g_{0}$ are asymptotic with derivatives, this is equivalent to
the fact that for all $i\geq0$, there exist constants $C_{i}>0$ such that for all $x\in M\setminus E$,
$$\textrm{$r(x)^{i}|(\nabla^{g})^i(g(\tau)-g)|_{g}(x)\leq C_{i}(-\tau)r(x)^{-2}$ for all $\tau<0$.}$$
Finally, as we are working in $P_{r_{x}}(x)$ where $-1-r_{x}^{2}<\tau<-1$ so that
$-\tau<1+r_{x}^{2}$, and since $r^{2}_{x}=\iota_{0}\rho(x)=\iota_{0}f(x)=\iota_{0}\left(\frac{r(x)^{2}}{2}-n\right)$ outside a compact subset of $M$, the claim follows.
\end{proof}

For a given tensor $T$ on $M$, recall that $\nabla^{g(\tau)}T=\nabla^gT+g(\tau)^{-1}\ast\nabla^g(g(\tau)-g)\ast T$, where all contractions are with respect to $g$. By induction on $i\geq 1$
using Claim \ref{claim-cov-der-g}, one can then prove that for all $i\geq1$, there exists a constant $C_i>0$ such that
\begin{equation}\label{abstract-switch-cov-der}
\begin{split}
r_x^i\cdot\sup_{P_{r_x}(x)}&|\nabla^{g(\tau),\,i}T-\nabla^{g,\,i}T|_g\leq \\
&C_i\sum_{k\,=\,0}^{i-1}r_x^{i-k}\cdot\left[\sup_{P_{r_x}(x)}\sum_{ i_1+...+i_p=i-k}\prod_{j\,=\,1}^{p}|\nabla^{g,\,i_j}g(\tau)|_g\right] r_x^k\cdot\sup_{P_{r_x}(x)}|\nabla^{g,\,k}T|_g.
\end{split}
\end{equation}
If one applies the previous inequality to $T:=\partial_{\tau}^{(j)}\tilde{g}$ for $j\geq 0$, then one can see that \eqref{hyp-basic-ass} implies \eqref{hyp-basic-ass-II} by invoking Claim \ref{claim-cov-der-g}. On the other hand, by reversing the roles of $g(\tau)$ and $g$ in \eqref{abstract-switch-cov-der}, we see that \eqref{hyp-basic-ass-II} implies \eqref{hyp-basic-ass}.
\item Observe that
 \begin{equation}
 \begin{split}\label{id-time-der}
 \partial_{\tau}\tilde{g}(\tau)\big|_{\tau=-1}&=-\tilde{g}+\mathcal{L}_{\frac{X}{2}}\tilde{g}\\
 &=\left(\mathcal{L}_{\frac{X}{2}}g-g\right)+\mathcal{L}_{\frac{X}{2}}(\tilde{g}-g)-\left(\tilde{g}-g\right)\\
 &=\left(\mathcal{L}_{\frac{X}{2}}g-g\right)+\nabla^g_{\frac{X}{2}}(\tilde{g}-g)+\frac{1}{2}(\tilde{g}-g)(\nabla^g_{\cdot}X,\,\cdot)+\frac{1}{2}(\tilde{g}-g)(\cdot\,,\nabla^g_{\cdot}X)-\left(\tilde{g}-g\right)\\
 &=\left(\mathcal{L}_{\frac{X}{2}}g-g\right)+\nabla^g_{\frac{X}{2}}(\tilde{g}-g)+(\tilde{g}-g)\ast(\mathcal{L}_{\frac{X}{2}}g-g),
 \end{split}
 \end{equation}
 where we have used the fact that $\mathcal{L}_{X}T=\nabla^g_XT+T(\nabla^g_{\cdot}X,\,\cdot)+T(\cdot\,,\nabla^g_{\cdot}X)$ for any $2$-tensor $T$ on $M$ in the third line, together with the fact that $X$ is gradient (so that $2g(\nabla^g_{\cdot}X,\,\cdot)=\mathcal{L}_Xg$) in the final line. By \eqref{hyp-basic-ass}, the left-hand side decays at rate $O(r^{-2})$ with respect to $g$, and so by Proposition \ref{mainprop}(i) and \eqref{hyp-basic-ass} again, we deduce that $|\nabla^g_X(\tilde{g}-g)|_g=O(r^{-2})$. Recalling the asymptotic property $\lim_{f\rightarrow +\infty}|\tilde{g}-g|_g=0$, one arrives at the fact that $|\tilde{g}-g|_g=O(r^{-2})$ by integrating to $+\infty$ along radial lines the previous estimate.

To prove the corresponding quadratic decay on the rescaled covariant derivatives\linebreak $\rho^{\frac{i}{2}}\nabla^{g,\,i}(\tilde{g}-g)$, one first invokes \eqref{hyp-basic-ass} with $j=1$ and an arbitrary $i\geq 0$ to deduce that
\begin{equation*}
 \begin{split}
\left |\rho^{\frac{i}{2}}\nabla^{g,\,i}\left( \partial_{\tau}\tilde{g}(\tau)\big|_{\tau=-1}\right)\right|_g=O(r^{-2}).
 \end{split}
 \end{equation*}
From this, one can then see from Proposition \ref{mainprop}(i) and \eqref{id-time-der} and that
 \begin{equation*}
 \begin{split}\label{info-1-i0-j1-bis}
\left |\rho^{\frac{i}{2}}\nabla^{g,\,i}\left(\nabla^{g}_{X}(\tilde{g}-g)\right)\right|_g=O(r^{-2}).
 \end{split}
 \end{equation*}
Next, recall the following commutation formula for an arbitrary tensor $T$ that can be proved by induction on $i\geq 0$:
 \begin{equation*}
 \begin{split}
 \left[\nabla^{g}_{X},\nabla^{g,\,i}\right]T&=-i\nabla^{g,\,i}T+\sum_{k\,=\,0}^i\nabla^{g,\,i-k}\Rm(g)\ast_{g} \nabla^{g,\,k}T+\sum_{k\,=\,0}^i\nabla^{g,\,i-k}\left(\mathcal{L}_Xg-g\right)\ast_{g} \nabla^{g,\,k}T.
\end{split}
\end{equation*}
Plugging in $T:=\tilde{g}-g$ and using Proposition \ref{mainprop}(ii), it follows from \eqref{hyp-basic-ass} that
\begin{equation}
 \begin{split}\label{info-1-i0-j1-bis}
\left |\nabla^{g}_{X}\left(\rho^{\frac{i}{2}}\nabla^{g,\,i}(\tilde{g}-g)\right)\right|_g=O(r^{-2}).
 \end{split}
 \end{equation}
A standard interpolation inequality now implies that $\lim_{f\rightarrow +\infty}|\rho^{\frac{i}{2}}\nabla^{g,\,i}(\tilde{g}-g)|_g=0$ so that $|\rho^{\frac{i}{2}}\nabla^{g,\,i}(\tilde{g}-g)|_g=O(\rho^{-1})$ by integrating \eqref{info-1-i0-j1-bis} along radial lines to $+\infty$.
\end{enumerate}
\end{proof}

We write $\nabla^{\tilde{g}}$ for the Levi-Civita connection of $\tilde{g}$ and
$X=\nabla^{\tilde{g}}\tilde{f}$ for some smooth real-valued function $\tilde{f}:M\rightarrow \R$, a function defined up to an additive constant
that is guaranteed to exist because, as already remarked, $M$ is simply connected. As we are identifying $M\setminus E$ with $C_{0}$ via $\pi$, we can write
 $X=r\partial_{r}$ on $M\setminus E$. Moreover, because $\nabla^{g}f=X=\nabla^{\tilde{g}}\tilde{f}$, it follows from \eqref{hyp-basic-ass} that $|f-\tilde{f}|=O(\log r)$ as $r\to+\infty$. In particular, there exists a positive constant $C$ such that outside a compact subset of $M$, $C^{-1}f\leq \tilde{f}\leq Cf.$
Henceforth, we denote $\Delta_{\tilde{g},\,X}:=\Delta_{\tilde{g}}-X$, and for a tensor $\alpha$ on $M$
we shall write ``$\alpha=O(f^{\lambda})$ with $g$-derivatives'' whenever $|(\nabla^{g})^{k}\alpha|_{g}=O\left(f^{\lambda-\frac{k}{2}}\right)$ for every $k\in\mathbb{N}_{0}$.

We now identify a good barrier function for our geometric situation.
\begin{lemma}\label{lemma-sub-sol-barrier}
For all $\delta\in(0,\,1)$, there exists $R(\delta)>0$ such that the function $e^{\delta f}$ is a sub-solution of the following equation:
\begin{equation*}
\begin{split}
\Delta_{\tilde{g},\,X}e^{\delta f}&\leq 0\qquad\text{on $f\geq R(\delta)$}.
\end{split}
\end{equation*}
Moreover, the polynomial powers of $f:=\frac{r^2}{2}-n$ satisfy for all $\delta \neq0$,
\begin{equation*}
\begin{split}
\Delta_{\tilde{g},\,X}f^{-\delta}-2\delta f^{-\delta}=O(f^{-\delta-1})
\qquad\textrm{ with $g$-derivatives.}
\end{split}
\end{equation*}
\end{lemma}

\begin{proof}
Using \eqref{asy-cond-diff-spatial}, we compute that
\begin{equation*}
\begin{split}
\Delta_{\tilde{g},\,X} e^{\delta f}&=\delta\left(\Delta_{\tilde{g},\,X}f+\delta|\nabla^{\tilde{g}}f|^2_{\tilde{g}}\right) e^{\delta f}\\
&=\delta\left( \Delta_{g,\,X}f+(\Delta_{\tilde{g},\,X}-\Delta_{g,\,X})f+\delta|\nabla^{\tilde{g}}f|^2_{\tilde{g}}\right)e^{\delta f}\\
&=\delta\left(-2f+O(f^{-1})+O(|\tilde{g}-g|_{\tilde{g}})+\delta|\nabla^{\tilde{g}}f|^2_{\tilde{g}}\right)e^{\delta f}\\
&=\delta\left(-2f+O(f^{-1})+\delta \left(|X|^2_g(1+O(f^{-1}))+O(f^{-\frac{1}{2}})\right)\right)e^{\delta f}\\
&=\delta\left(-2f+O(f^{-1})+\delta \left((2f+O(1))(1+O(f^{-1}))+O(f^{-\frac{1}{2}})\right)\right)e^{\delta f}\\
&\leq 0
\end{split}
\end{equation*}
outside a sufficiently large compact subset $K$ of $M$ depending on $\delta$. Here we have also used the fact that $|X|^2_g=2f+n+O(f^{-1})$ from Proposition \ref{mainprop}(ii) and $\delta \in(0,\,1)$ in the last line. We also use the fact that
$$|\nabla^{\tilde{g}}f|_{\tilde{g}}=|X|_{g}(1+O(r^{-2}))+O(r^{-1}).$$
A similar computation based on the asymptotics of $\tilde{g}$ given by \eqref{hyp-basic-ass} shows that
\begin{equation*}
\begin{split}
\Delta_{\tilde{g},\,X}f^{-\delta}&=(\Delta_{\tilde{g}}-X)(f^{-\delta})\\
&=-\delta f^{-\delta-1}(\Delta_{\tilde{g}}f-X\cdot f)+\delta(\delta+1)f^{-\delta-2}|\nabla^{\tilde{g}}f|_{\tilde{g}}^{2}\\
&=-\delta f^{-\delta-1}(\Delta_{g}f-X\cdot f)-\delta f^{-\delta-1}(\Delta_{\tilde{g}}f-\Delta_{g}f)+\delta(\delta+1)f^{-\delta-2}|\nabla^{\tilde{g}}f|_{\tilde{g}}^{2}\\
&=2\delta f^{-\delta}+O(f^{-\delta-2})-\delta f^{-\delta-1}\underbrace{(\Delta_{\tilde{g}}f-\Delta_{g}f)}_{=\,O(|\tilde{g}-g|_{g})}+\delta(\delta+1)f^{-\delta-2}\underbrace{|\nabla^{\tilde{g}}f|_{\tilde{g}}^{2}}_{=\,O(|X|^{2}_{g})\,=\,O(f)}\\
&=2\delta f^{-\delta}+O(f^{-\delta-1}).
\end{split}
\end{equation*}
\end{proof}

\subsection{Function spaces}\label{function-spaces-subsection}

We next define the function spaces within which we will work.
\begin{itemize}
\item
Let $(\varphi^{X}_{\tau})_{\tau<0}$ denote the flow of $\frac{X}{2(-\tau)}$ such that $\varphi^{X}_{\tau}\big|_{\tau=-1}=\Id_M$. Define for a real-valued
function $u:M\rightarrow\R$ the following time-dependent function:
\begin{equation}\label{ronan2}
\widetilde{u}(x,\tau):=(\varphi^X_{\tau})^*u(x),\qquad x\in M,\qquad\tau<0.
\end{equation}
\item For $k\geq 0$ and $\alpha\in\left(0,\frac{1}{2}\right)$, denote the standard $C^{2k,\,2\alpha}$-norm of tensors on $P_{r_x}(x,-1)$ by\linebreak $\|\cdot\|_{C^{2k,\,2\alpha}(P_{r_x}(x))}$:
\begin{equation*}
\|T\|_{C^{2k,\,2\alpha}(P_{r_x}(x))}:=\sum_{i+2j\,\leq\, 2k}r_x^{i+2j}\cdot\sup_{P_{r_x}(x)}|\nabla^{\tilde{g},\,i}(\partial_{\tau}^{(j)})T|_{\tilde{g}}+\sum_{i+2j\,=\,2k}r_x^{2k+2\alpha}\cdot\left[\nabla^{\tilde{g},\,i}(\partial_{\tau}^{(j)})T\right]
_{C^{0,\,2\alpha}(P_{r_x}(x))},
\end{equation*}
where $[\,\cdot\,]_{C^{0,2\alpha}}$ denotes the $2\alpha$-semi-norm on tensors on $M$ induced by $\tilde{g}$.

\begin{remark}\label{trump}
Thanks to \eqref{asy-cond-diff-spatial}, this norm is equivalent to that defined with respect to the background metric $g$, hence we may
use either $\tilde{g}$ or $g$ with our particular choice depending on the situation in question. Moreover, in light of Lemma \ref{lemma-basic-equiv-ass}(ii),
this norm is equivalent to that defined by $g(\tau)$.

Similarly, as $f$ and $\tilde{f}$ are equivalent at infinity, i.e.,
there exists $C>0$ such that $C^{-1}f\leq \tilde{f}\leq Cf$ on $M$, these function
spaces can be defined in terms of either of these two potential functions, i.e., one can either use the weight $r_x^2:=
\iota_0^2\rho$ or $r_x^2:=\iota_{0}^{2}\tilde{\rho}$ by decreasing $\iota_0$ if necessary, where
$\tilde{\rho}$ is a $JX$-invariant function on $M$ equal to $\tilde{f}$ outside a compact subset and bounded below by $1$.
\end{remark}

Notice in the definition of the above norm that the number of spatial derivatives appearing
in each summand is no more than twice the number of time derivatives. This
is because, when solving the Poisson equation for the drift Laplacian,
the drift Laplacian can be treated as a second order parabolic operator with the time derivative
corresponding to the $X$-derivative. These heuristics play out in the proof of Theorem \ref{iso-sch-Laplacian-pol} below.
\item  For $\beta\in\R$, $k$ a non-negative integer, and $\alpha\in\left(0,\,\frac{1}{2}\right)$, define the H\"older space $C_{X,\,\beta}^{2k,\,2\alpha}(M)$ with polynomial weight $f^{\frac{\beta}{2}}$ (or equivalently $\tilde{f}$ or even $\rho$) to be the set of $JX$-invariant functions $u$ such that
\begin{equation*}
\norm{u}_{C^{2k,\,2\alpha}_{X,\,\beta}(M)}:=\sup_{x\in M}\rho(x)^{\frac{\beta}{2}}\|(-\tau)^{-\frac{\beta}{2}}\tilde{u}\|_{C^{2k,\,2\alpha}(P_{r_x}(x))} < \infty.
\end{equation*}
It is straightforward to check that the space $C^{2k,\,2\alpha}_{X,\,\beta}(M)$ is a Banach space.
The intersection
$\bigcap_{k\,\geq\,0}C^{2k}_{X,\,\beta}(M)$ we denote by $C^{\infty}_{X,\,\beta}(M)$.
\item We now consider a smooth cut-off function $\chi:M\rightarrow[0,\,1]$ which equals $1$ outside a compact set.
The source function space $\mathcal{D}^{2k+2,\,2\alpha}_{X}(M)$ is defined as
\begin{equation*}
\mathcal{D}^{2k+2,\,2\alpha}_{X}(M):=\left(\R\chi\cdot\log r\oplus C^{2k+2,\,2\alpha}_{X,\,0}(M)\right),
\end{equation*}
endowed with the norm
\begin{equation*}
\begin{split}
\norm{u}_{\mathcal{D}_{X}^{2k+2,\,2\alpha}(M)}&:=|c|+\|v\|_{C_{X,\,0}^{2k+2,\,2\alpha}(M)},\\
u&:=c\chi \log r+v.
\end{split}
\end{equation*}
We define
\begin{equation*}
C^{\infty}_{X}(M):=\bigcap_{k\,\geq\,0} C^{2k,\,2\alpha}_{X,\,0}(M).
\end{equation*}

\item The target function space is defined as
 \begin{equation*}
\mathcal{C}^{2k,\,2\alpha}_{X}(M):=\left(\R\oplus C^{2k,\,2\alpha}_{X,\,2}(M)\right),
\end{equation*}
endowed with a norm defined in a similar manner as above. We define
\begin{equation*}
\mathcal{C}^{\infty}_{X}(M):=\bigcap_{k\,\geq\,0} \mathcal{C}^{2k,\,2\alpha}_{X}(M).
\end{equation*}
\item Finally, we define the spaces
\begin{equation*}
\begin{split}
\mathcal{M}^{2k+2,\,2\alpha}_{X}(M)&:=\left\{\varphi\in C^2_{\operatorname{\operatorname{loc}}}(M)\,|\,\tilde{\omega}+i\partial\bar{\partial}\varphi>0\right\}\bigcap \mathcal{D}^{2k+2,\,2\alpha}_{X}(M),
\end{split}
\end{equation*}
and we will work with the following convex set of K\"ahler potentials:
\begin{equation*}
\mathcal{M}^{\infty}_{X}(M)=\bigcap_{k\,\geq\,0}\,\mathcal{M}^{2k+2,\,2\alpha}_{X}(M).
\end{equation*}
Notice that for each $k\geq 0$, the spaces $\mathcal{M}^{2k+2,\,2\alpha}_{X}(M)$ depend on the choice of a background metric $\tilde{\omega}$.
However, these spaces are all equivalent as soon as $\tilde{\omega}$ satisfies \eqref{hyp-basic-ass}.
\end{itemize}

In light of Lemma \ref{lemma-basic-equiv-ass}(i), the following stability lemma demonstrates that solving \eqref{ast-0} in the above function spaces is well-posed.
\begin{lemma}
Suppose that $\tilde{g}$ satisfies \eqref{hyp-basic-ass} (or equivalently \eqref{hyp-basic-ass-II}) and let $\psi \in \mathcal{M}^{\infty}_{X}(M)$. Then $\tilde{\omega}+i\partial\bar{\partial}\psi$ satisfies \eqref{hyp-basic-ass} (or equivalently \eqref{hyp-basic-ass-II}). In particular, $\omega+i\partial\bar{\partial}\psi$ satisfies \eqref{hyp-basic-ass}.
\end{lemma}

\begin{proof}
Since $\psi\in \mathcal{M}^{\infty}_{X}(M)$, we know that $|i\partial\bar{\partial}\psi|_g=O(r^{-2})$. From the triangle inequality and the fact that $\tilde{g}$ satisfies \eqref{hyp-basic-ass}, we then see that $\lim_{f\rightarrow+\infty}|\tilde{\omega}+i\partial\bar{\partial}\psi-\omega|_g=0$. Now, if $j=0$ and $k\geq 0$ is arbitrary, then
\begin{equation*}
r_x^k\cdot\sup_{P_{r_x}(x)}|\nabla^{g,\,k}\left[(-\tau)\varphi_{\tau}^*\left(i\partial\bar{\partial}\psi\right)\right]|_g\leq r_x^{k+2}\cdot\sup_{P_{r_x}(x)}|\nabla^{g,\,k+2}\tilde{\psi}|_g\leq  C_k
\end{equation*}
for some positive constant $C_k$ independent of $x\in M$. This leads to \eqref{hyp-basic-ass} for $\tilde{\omega}+i\partial\bar{\partial}\psi$ when $j=0$, thanks to the triangle inequality and the fact that $\tilde{g}$ satisfies \eqref{hyp-basic-ass}.

Now, for $j\geq 1$, observe that
\begin{equation*}
\partial_{\tau}^{(j)}\left((-\tau)\varphi_{\tau}^*
\left(i\partial\bar{\partial}\psi\right)\right)=i\partial\bar{\partial}\left((-\tau)\partial_{\tau}^{(j)}\tilde{\psi}-\partial_{\tau}^{(j-1)}\tilde{\psi}\right).
\end{equation*}
We therefore have that
\begin{equation*}
\begin{split}
r_x^{k+2j}\cdot\sup_{P_{r_x}(x)}\left|\nabla^{g,\,k}\left(\partial_{\tau}^{(j)}\left((-\tau)\varphi_{\tau}^*
\left(i\partial\bar{\partial}\psi\right)\right)\right)\right|_g&\leq r_x^{(k+2)+2j}\cdot\sup_{P_{r_x}(x)}\left|\nabla^{g,\,k+2}\left(\partial_{\tau}^{(j)}\tilde{\psi}\right)\right|_g\\
&\quad+r_x^{(k+2)+2(j-1)}\cdot\sup_{P_{r_x}(x)}\left|\nabla^{g,\,k+2}\left(\partial_{\tau}^{(j-1)}\tilde{\psi}\right)\right|_g\\
&\leq  C_k,
\end{split}
\end{equation*}
which completes the proof of the lemma.
\end{proof}

\subsection{Preliminaries and Fredholm properties of the linearised operator}\label{linear}
We proceed with the same set-up as in Section \ref{setup}, beginning with the following useful observation.

\begin{lemma}\label{lemma-preserved-int}
Let $(\varphi_t)_{t\in[0,\,1]}$ be a $C^1$-path of smooth functions in $\mathcal{M}^{\infty}_{X}(M)$ and write
$\tilde{\omega}_{t}:=\tilde{\omega}+i\partial\bar{\partial}\varphi_{t}>0$ and $\tilde{f}_{t}:=\tilde{f}+\frac{X}{2}\cdot\varphi_{t}$, so that $-d\tilde{\omega}_{t}\lrcorner JX=d\tilde{f}_{t}$.
\begin{enumerate}
  \item
 Let $G:\R\rightarrow \R$ be a $C^1$-function such that for some $-\infty<\alpha<1$, $|G(x)|+|G'(x)|\leq e^{\alpha x}$, $x\geq -C$.
Then
\begin{equation*}
\int_MG(\tilde{f}_{t})\,e^{-\tilde{f}_{t}}\tilde{\omega}_{t}^n=\int_MG(\tilde{f}_{0})\,e^{-\tilde{f}_{0}}\tilde{\omega}_{0}^n,\qquad t\in[0,\,1].
\end{equation*}
  \item $\int_{0}^{1}\int_{M}|\dot{\varphi}_{t}|\,e^{-\tilde{f}_{t}}\tilde{\omega}^{n}_{t}\,dt<+\infty$ and $\int_{0}^{1}\int_{M}|\dot{\varphi}_{t}|\,e^{-\tilde{f}}\tilde{\omega}^{n}\,dt<+\infty$.
\end{enumerate}
\end{lemma}

\begin{proof}
The proof follows verbatim \cite[Lemma 6.2]{ccd2}.
\end{proof}

Next, define the following map as in \cite{siepmann}:
\begin{equation*}
\begin{split}
MA_{\tilde{\omega}}:\psi\in&\left\{\varphi\in C^2_{\operatorname{\operatorname{loc}}}(M)\,|\,\tilde{\omega}_{\varphi}:=\tilde{\omega}+i\partial\bar{\partial}\varphi>0\right\}\mapsto\log\left(\frac{\tilde{\omega}_{\psi}^n}{\tilde{\omega}^n}\right)
-\frac{X}{2}\cdot\psi\in\mathbb{R}.
\end{split}
\end{equation*}
For any $\psi\in C_{\operatorname{\operatorname{loc}}}^{2}(M)$, let
$\tilde{g}_{\psi}$ (respectively $\tilde{g}_{t\psi}$) denote the K\"ahler metric associated to the K\"ahler form $\tilde{\omega}_{\psi}$
(resp.~$\tilde{\omega}_{t\psi}$ for any $t\in[0,\,1]$). Brute force computations show that
\begin{equation*}
\begin{split}
MA_{\tilde{\omega}}(0)&=0,\\
D_{\psi}MA_{\tilde{\omega}}(u)&=\Delta_{\tilde{\omega}_{\psi}}u-\frac{X}{2}\cdot u,\quad u\in C^2_{\operatorname{\operatorname{loc}}}(M),\\
\frac{d^2}{dt^2}\left(MA_{\tilde{\omega}}(t\psi)\right)&=\frac{d}{dt}(\Delta_{\tilde{\omega}_{t\psi}}\psi)=-\arrowvert\partial\bar{\partial}\psi\arrowvert^2_{\tilde{g}_{t\psi}}\quad\textrm{for $t\in[0,\,1]$},
 \end{split}
 \end{equation*}
\begin{equation}\label{equ:taylor-exp}
 \begin{split}
 MA_{\tilde{\omega}}(\psi)&=MA_{\tilde{\omega}}(0)+\left.\frac{d}{dt}\right|_{t\,=\,0}MA_{\tilde{\omega}}(t\psi)+\int_0^1\int_0^{u}\frac{d^2}{dt^2}(MA_{\tilde{\omega}}(t\psi))\,dt\,du\\
 &=\Delta_{\tilde{\omega}}\psi-\frac{X}{2}\cdot\psi-\int_0^1\int_0^{u}\arrowvert \partial\bar{\partial}\psi\arrowvert^2_{\tilde{g}_{t\psi}}\,dt\,du.
 \end{split}
 \end{equation}

The main result of this section is that the drift Laplacian of $\tilde{g}$ is an isomorphism between polynomially weighted function spaces with zero mean value.

\begin{theorem}\label{iso-sch-Laplacian-pol}
Let $\alpha\in\left(0,\,\frac{1}{2}\right)$, $k\in\mathbb{N}$. Then the drift Laplacian
\begin{equation*}
\begin{split}
\Delta_{\tilde{g},\,X}:\mathcal{D}^{2k+2,\,2\alpha}_{X}(M)
\bigcap\left\{\int_{M}u\,e^{-\tilde{f}}\tilde{\omega}^{n}=0\right\}\rightarrow \mathcal{ C}^{2k,\,2\alpha}_{X}(M)\bigcap\left\{\int_{M}v\,e^{-\tilde{f}}\tilde{\omega}^{n}=0\right\}
\end{split}
\end{equation*}
is an isomorphism of Banach spaces. In particular, there exists a positive constant $C=C(k,\alpha,\tilde{g})$ such that for all $u\in \mathcal{D}^{2k+2,\,2\alpha}_{X}(M)
\,\cap\,\left\{\int_{M}u\,e^{-\tilde{f}}\tilde{\omega}^{n}=0\right\}$,
\begin{equation*}
\|u\|_{\mathcal{D}^{2k+2,\,2\alpha}_{X}(M)}\leq C \|\Delta_{\tilde{g},\,X}u\|_{\mathcal{ C}^{2k,\,2\alpha}_{X}(M)}.
\end{equation*}
\end{theorem}

\begin{proof}
The linear map $\Delta_{\tilde{g},\,X}:\mathcal{D}^{2k+2,\,2\alpha}_{X}(M)
\rightarrow \mathcal{ C}^{2k,\,2\alpha}_{X}(M)$ is well-defined and continuous by the very definition of the function spaces involved. Indeed, with notation as in \eqref{ronan2},
it suffices to observe that
\begin{equation*}
\left(\widetilde{\Delta_{\tilde{g},\,X}u}\right)(x,\,\tau)=(-\tau)\left(\Delta_{\tilde{g}(\tau)}-2\partial_{\tau}\right)\tilde{u}(x,\,\tau),\qquad x\in M,\qquad \tau<0,
\end{equation*}
where $u$ is an arbitrary smooth function defined on $M$. One then invokes Lemma \ref{lemma-basic-equiv-ass}(ii) (cf.~Remark \ref{trump}).
Furthermore, the image of the drift Laplacian has zero mean value with respect to the weighted measure $e^{-\tilde{f}}\tilde{\omega}^{n}$.
One can see this by integrating by parts (cf.~Section \ref{metricmeasure}), which in this case is justified because of the decay properties of the source function space $\mathcal{D}^{2k+2,\,2\alpha}_{X}(M)$.

Next, observe that the drift Laplacian $\Delta_{\tilde{g},\,X}$ is symmetric with respect to the
weighted measure $e^{-\tilde{f}}\tilde{\omega}^n$ (cf.~Section \ref{metricmeasure}), a measure with finite volume. Set
\begin{equation*}
\begin{split}
H^1_{\tilde{f}}(M)&:=\left\{u\in H^1_{\operatorname{\operatorname{loc}}}(M)\,\left.\right|\,\textnormal{$u$ is $JX$-invariant, $u\in L^2(e^{-\tilde{f}}\tilde{\omega}^n)$, and $\nabla^{\tilde{g}}u\in L^2(e^{-\tilde{f}}\tilde{\omega}^n)$}\right\},\\
W^2_{\tilde{f}}(M)&:=\left\{u\in H^1_{\tilde{f}}(M)\,\left.\right|\, \Delta_{\tilde{\omega},\,X}u\in L^2(e^{-\tilde{f}}\tilde{\omega}^n)\right\},
\end{split}
\end{equation*}
endowed with the obvious norms induced by that of $L^2(e^{-\tilde{f}}\tilde{\omega}^n)$.
It can be shown that the operator $\Delta_{\tilde{g},\,X}$, when restricted to compactly supported smooth $JX$-invariant functions, admits a unique self-adjoint extension to $W^2_{\tilde{f}}(M)$
with domain contained in $H_{\tilde{f}}^1(M)$, and it can further be shown that it has a discrete $L^2(e^{-\tilde{f}}\tilde{\omega}^n)$-spectrum; see \cite[Proposition $6.13$]{Der-Com-Egs}  in the context of expanding gradient Ricci solitons, but the proof of which can be adapted to the current setup. See also \cite[Theorem $4.6$]{Gri-Boo} for a more general setting.
Observe also that the kernel (and hence the cokernel) of this operator is the constant functions. By considering any function $F$ in the codomain as an element of the weighted $L^2$-space $L^2(e^{-\tilde{f}}\tilde{\omega}^n)$, we can therefore find a unique weak solution $u\in H^1(e^{-\tilde{f}}\tilde{\omega}^n)$ with zero weighted mean value of the equation
\begin{equation}\label{love-drift-lap}
\Delta_{\tilde{g},\,X} u=F.
\end{equation}
In addition, we have the estimate
\begin{equation}
\begin{split}\label{weak-apriori-bd-lin-th}
\|u\|_{L^2(e^{-\tilde{f}}\tilde{\omega}^n)}+\|\nabla^{\tilde{g}} u\|_{L^2(e^{-\tilde{f}}\tilde{\omega}^n)}&\leq C\|F\|_{L^2(e^{-\tilde{f}}\tilde{\omega}^n)}\leq C\|F\|_{C^{0}(M)}
\end{split}
\end{equation}
for some positive constant $C>0$ independent of $u$ and $F$ that may vary from line to line. The first inequality here essentially follows from the weighted $L^2$-Poincar\'e inequality (Proposition \ref{poincare}) with respect to the drift Laplacian $\Delta_{\tilde{g}}-X\cdot$.

We improve on the regularity of $u$ through a series of claims, beginning with:
\begin{claim}\label{claim-first-rough-growth}
There exists a positive constant $C=C(\tilde{\omega},n)$ such that $$|u(x)|\leq Ce^{ \frac{\tilde{f}(x)}{2}}\|F\|_{C^{0}(M)}\qquad\textrm{for all $x\in M$.}$$
\end{claim}

\begin{proof}[Proof of Claim \ref{claim-first-rough-growth}]
By conjugating \eqref{love-drift-lap} with a suitable weight, notice that the function $v:=e^{-\frac{\tilde{f}}{2}}u$ satisfies
\begin{equation*}
\begin{split}
\Delta_{\tilde{g}}v&=e^{-\frac{\tilde{f}}{2}}F+\left(\frac{1}{4}|X|^2_{\tilde{g}}-\frac{1}{2}\Delta_{\tilde{g}}\tilde{f}\right)v.
\end{split}
\end{equation*}
This implies that $|v|$ satisfies the following differential inequality in the weak sense:
\begin{equation}
\Delta_{\tilde{g}}|v|\geq -C|v|-C\|F\|_{C^{0}(M)}.\label{love-drift-lap-conj}
\end{equation}
Here we have made use of the non-negativity of $|X|^2_{\tilde{g}}$ together with the boundedness of $\Delta_{\tilde{g}}\tilde{f}$ given by \eqref{hyp-basic-ass}.

We perform a local Nash-Moser iteration on \eqref{love-drift-lap-conj} in $B_{\tilde{g}}(x,\,r)$.
More precisely, since $(M^{2n},\,\tilde{g})$ is a Riemannian manifold with Ricci curvature bounded from below,
the results of \cite{Sal-Cos-Uni-Ell} yield the following local Sobolev inequality:
\begin{equation}
\begin{split}\label{sob-inequ-loc}
\left(\frac{1}{\operatorname{vol}_{\tilde{g}} (B_{\tilde{g}}(x,\,r))}\int_{B_{\tilde{g}}(x,\,r)}|\varphi|^{\frac{2n}{n-1}}\,\tilde{\omega}^{n}\right)^{\frac{n-1}{n}}\leq \left(\frac{C(r_0)r^2}{\operatorname{vol}_{\tilde{g}}(B_{\tilde{g}}(x,\,r))}\int_{B_{\tilde{g}}(x,\,r)}|\widetilde{\nabla}\varphi|_{\tilde{g}}^2\,\tilde{\omega}^{n}\right)
\end{split}
\end{equation}
for any $\varphi\in H^1_0(B_{\tilde{g}}(x,\,r))$ and for all $x\in M$ and $0<r<r_0$, where $r_0>0$ is some fixed positive radius.

A Nash-Moser iteration proceeds in several steps. First, one multiplies \eqref{love-drift-lap}
across by\linebreak  $\eta_{s,\,s'}^2v|v|^{2(p-1)}$ with $p\geq 1$, where $\eta_{s,\,s'}$, with $0<s+s'<r$ and $s,s'>0$,
is a Lipschitz cut-off function with compact support in $B_{\tilde{g}}(x,\,s+s')$ equal to $1$ on $B_{\tilde{g}}(x,\,s)$ and with
$|\widetilde{\nabla}\eta_{s,\,s'}|_{\tilde{g}}\leq\frac{1}{s'}$ almost everywhere. One then integrates by parts and uses
the Sobolev inequality of \eqref{sob-inequ-loc} to obtain a so-called ``reversed H\" older inequality'' which, after iteration, leads to the bound
\begin{equation*}
\begin{split}
\sup_{B_{\tilde{g}}(x,\,\frac{r}{2})}|v|\leq&\, C\left(\|v\|_{L^2(B_{\tilde{g}}(x,\,r))}+\|F\|_{L^{\infty}(B_{\tilde{g}}(x,\,r))}\right)\\
\leq&\,C\left(\|u\|_{L^2(e^{-\tilde{f}}\tilde{\omega}^{n})}+\|F\|_{C^{0}(M)}\right)\\
\leq&\,C\|F\|_{C^{0}(M)}\\
\end{split}
\end{equation*}
for $r\leq r_0$, where $C=C(r_0,\,\tilde{\omega},n)$. Here we have made use of \eqref{weak-apriori-bd-lin-th} in the last line.
This estimate yields an a priori local $C^0$-estimate which is uniform on the center of the ball $B_{\tilde{g}}(x,\,\frac{r}{2})$. In particular,
unravelling the definition of the function $v$, one obtains the expected a priori uniform exponential growth, namely
\begin{equation*}
|u(x)|\leq Ce^{\frac{\tilde{f}(x)}{2}}\|F\|_{C^{0}(M)}\qquad \textrm{for all $x\in M$}.
\end{equation*}
\end{proof}

Thanks to Claim \ref{claim-first-rough-growth}, by local Schauder elliptic estimates we actually see that $u$ lies in $C^{2k+2,\,2\alpha}_{\operatorname{loc}}(M)$ and that we have the estimates
\begin{equation}\label{first-loc-a-priori-est-lin-th}
\|u\|_{C^{2k+2,\,2\alpha}(\{\tilde{f}\,<\,R\})}\leq C\|F\|_{C^{2k,\,2\alpha}(\{\tilde{f}\,<\,2R\})}\leq C\|F\|_{\mathcal{C}^{2k,\,2\alpha}_{X}(M)}
\end{equation}
for some positive constant $C=C(R,\,\tilde{\omega},n)$. We now proceed to prove the expected a priori weighted estimates on $u$ and on its derivatives, beginning with the
logarithmic growth of $u$.
\begin{claim}\label{claim-sec-rough-growth}
There exists a positive constant $A$ such that
$$|u(x)|\leq A\log\tilde{f}(x)\|F\|_{C^{0}(M)}\qquad\textrm{for all $x\in M$ with $\tilde{f}(x)\geq 2$.}$$
\end{claim}

\begin{proof}[Proof of Claim \ref{claim-sec-rough-growth}]
Let $\varepsilon>0$ and let $\delta \in\left(\frac{1}{2},\,1\right)$ be
such that $\lim_{\tilde{f}\rightarrow+\infty}\left(u-\varepsilon e^{\delta \tilde{f}}\right)=-\infty$, parameters that we can choose by
Claim \ref{claim-first-rough-growth}. For $A>0$ a constant to be determined later, we have outside a compact set $\{\tilde{f}\geq R(\delta)\}$ the inequality
\begin{equation*}
\begin{split}
\Delta_{\tilde{g},\,X}\left(u-A\log (f+n)-\varepsilon e^{\delta f}\right)\geq -\|F\|_{C^{0}(M)}+A\left(2-\frac{C}{f}\right)
\end{split}
\end{equation*}
for some constant $C>0$ that may change from line to line. Here, both \eqref{bds-cov-der-f} and Lemma \ref{lemma-sub-sol-barrier} have been applied. Set $R':=\max\{R(\delta),\,2C\}$. Then for $f\geq R'$, we have
$$\Delta_{\tilde{g},\,X}\left(u-A\log (f+n)-\varepsilon e^{\delta f}\right)\geq -\|F\|_{C^{0}(M)}+\frac{3A}{2}.$$
In particular, for all $A>\frac{3}{2}\|F\|_{C^{0}(M)}$, the right-hand side is strictly positive. The maximum principle therefore yields the bound
$$\max_{\{f\,\geq\,R'\}}\left(u-A\log (f+n)-\varepsilon e^{\delta \tilde{f}}\right)=\max_{\{f\,=\,R'\}}\left(u-A\log (f+n)-\varepsilon e^{\delta \tilde{f}}\right).$$
Letting $\varepsilon\to0$, we then see that for all $x\in\{f\geq R'\}$,
\begin{equation*}
u-A\log (f+n)\leq \max_{\{f\,=\,R'\}}\left(u-A\log(R'+n)\right)\leq 0
\end{equation*}
after choosing $A>\max\left\{\frac{3}{2}\|F\|_{C^{0}(M)},\,\frac{C\|F\|_{C^{0}(M)}}{\log(R'+n)}\right\}\geq\frac{\max_{\{f= R'\}}u}{\log(R'+n)}$ for some $C>0$ that may change from line to line.
Choosing such an $A$ is possible by virtue of \eqref{first-loc-a-priori-est-lin-th}. Applying the same argument to $-u$ concludes the proof of the claim.
\end{proof}

Observe that $u_0:=u+c\chi\cdot\log r$, where $F-c\in C^{2k,\,2\alpha}_{X,\,2}(M)$, satisfies the equation
\begin{equation*}\label{heat-eqn-tilde-u}
\Delta_{\tilde{g},\,X}u_0=F+c\Delta_{\tilde{g},\,X}(\chi\cdot\log r)=(F-c)+\underbrace{c+c\Delta_{\tilde{g},\,X}(\chi\cdot\log r)}_{\text{$=\,O(r^{-2})$ with $\tilde{g}$-derivatives}}:=F_0\in C^{2k,\,2\alpha}_{X,\,2}(M).
\end{equation*}
We next give a $C^{0}$-bound on $u_{0}$.
\begin{claim}\label{claim-sec-rough-growth2}
There exists a positive constant $A>0$ such that
\begin{equation*}
\|u_0\|_{C^{0}(M)}\leq A\|F\|_{C^{0}(M)}.
\end{equation*}
\end{claim}

\begin{proof}[Proof of Claim \ref{claim-sec-rough-growth2}]
Thanks to Claim \ref{claim-first-rough-growth}, it suffices to prove Claim \ref{claim-sec-rough-growth2} on a set $\{f\geq R\}$ with $R$ arbitrarily large where both $\rho$ and $f$ coincide. We will prove one half of the asserted estimate, namely that $u_{0}\leq A \|F_0\|_{C^{0}(M)}$. The other half will follow by applying the same argument to $-u_0$.

Using the differential inequality $\Delta_{\tilde{g},\,X}u_0\geq -\frac{C}{f}\|F_0\|_{C^{0}(M)}$ for some $C>0$ outside a sufficiently large compact subset independent of $u_0$, together with Lemma \ref{lemma-sub-sol-barrier}, we compute that for $A>0$,
\begin{equation}\label{ireland}
\begin{split}
\Delta_{\tilde{g},\,X}\left(u_0+A\|F_0\|_{C^{0}(M)}f^{-1}\right)&\geq \frac{(A-C)\|F_0\|_{C^{0}(M)}}{f}>0,
\end{split}
\end{equation}
after setting $A:=C+1$. Now, if $\varepsilon>0$ is arbitrary, then $\lim_{f\rightarrow+\infty}\left(u_0-\varepsilon e^{\frac{f}{2}}\right)=-\infty$ by virtue of Claim \ref{claim-first-rough-growth}. Moreover, by \eqref{ireland} and Lemma \ref{lemma-sub-sol-barrier}, if $A:=C+1$ and $\varepsilon>0$, then 
on a subset of $M$ of the form $\{f\,\geq\,R\}$ for some $R>0$ independent of $u_{0}$, we have that
\begin{equation*}
\begin{split}
\Delta_{\tilde{g},\,X}\left( u_0+A\|F_0\|_{C^{0}(M)}f^{-1}-\varepsilon e^{\frac{f}{2}}\right)&\geq \|F_0\|_{C^{0}(M)}f^{-1}-\varepsilon\Delta_{\tilde{g},\,X}e^{\frac{f}{2}}>0.
\end{split}
\end{equation*}
Applying the maximum principle then yields the bound
\begin{equation*}
\max_{\{f\,\geq\,R\}}\left(u_0+A\|F_0\|_{C^{0}(M)}f^{-1}-\varepsilon e^{\frac{f}{2}}\right)=\max_{\{f\,=\,R\}}\left( u_0+A\|F_0\|_{C^{0}(M)}f^{-1}-\varepsilon e^{\frac{f}{2}}\right).
\end{equation*}
Letting $\varepsilon\to0$ and unravelling the definitions, we then find, thanks to Claim \ref{claim-sec-rough-growth}, that on $\{f\geq R\}$,
\begin{equation*}
u_0\leq C(R)\|F_0\|_{C^{0}(M)}+\max_{\{f\,=\,R\}}u_0\leq C(R)\|F\|_{C^{0}(M)},
\end{equation*}
as required.
\end{proof}

The following claim establishes the appropriate interior Schauder parabolic estimates that we need to complete the proof of Theorem \ref{iso-sch-Laplacian-pol}.
\begin{claim}\label{claim-a-priori-rough-bd-hih-der}
There exists a positive constant $C>0$ such that for all $x\in M$,
\begin{equation}\label{est-sch-loc-para}
\|\widetilde{u_0}\|_{C^{2k+2,\,2\alpha}\left(P_{\frac{r_x}{2}}(x)\right)}\leq C\left(\|\widetilde{u_0}\|_{C^{0}\left(P_{r_x}(x)\right)}+r_x^2\left\|(-\tau)^{-1}\widetilde{F_0}\right\|_{C^{2k,\,2\alpha}\left(P_{r_x}(x)\right)}\right).
\end{equation}
\end{claim}

\begin{proof}[Proof of Claim \ref{claim-a-priori-rough-bd-hih-der}]
By Lemma \ref{lemma-basic-equiv-ass}, for $x\in M$ and $\tau\in\left(-1-r_x^2,\,-1\right]$, the metrics $\tilde{g}(\tau)$
are uniformly equivalent and their covariant derivatives (with respect to $g$) and time derivatives are bounded on $P_{r_x}(x)$, uniformly with respect to $x\in M$. Since $\widetilde{u_0}(\tau)=(\varphi^X_{\tau})^*u_0$ satisfies
\begin{equation}\label{heat-eqn-disguise}
\partial_{\tau}\widetilde{u_0}=\frac{1}{2}\Delta_{\tilde{g}(\tau)}\widetilde{u_0}+(2\tau)^{-1}\widetilde{F_0},
\end{equation}
standard interior parabolic Schauder estimates \cite[Theorem 8.12.1]{Kry-Boo} applied to \eqref{heat-eqn-disguise} on $P_{r_x}(x)$ therefore ensures the existence of a uniform positive constant $C>0$ such that \eqref{est-sch-loc-para} holds uniformly with respect to $x\in M$.
\end{proof}

Combining Claims \ref{claim-sec-rough-growth} and \ref{claim-a-priori-rough-bd-hih-der} leads to the desired continuity estimate of the inverse of the drift Laplacian between $\mathcal{D}^{2k+2,\,2\alpha}_{X}(M)$ and $\mathcal{C}^{2k,\,2\alpha}_{X}(M)$, that is, the inequality in the statement of the theorem. This concludes the proof of Theorem \ref{iso-sch-Laplacian-pol}.
\end{proof}

\subsection{Small perturbations along the continuity path}\label{invert-poly}

In this section we show, using the implicit function theorem, that the
invertibility of the drift Laplacian given by Theorem \ref{iso-sch-Laplacian-pol}
allows for small perturbations in polynomially weighted function spaces
of solutions to the complex Monge-Amp\`ere equation that we wish to solve.
This forms the openness part of the continuity method, as will be explained in Section \ref{continuitie}.

In notation reminiscent of that of \cite[Chapter $5$]{Tian-Can-Met-Boo}, we consider
the space $\left(\mathcal{C}^{2,\,2\alpha}_{X}(M)\right)_{\tilde{\omega},\,0}$ of functions $F\in\mathcal{C}^{2,\,2\alpha}_{X}(M)$ satisfying
\begin{equation*}
\int_M\left(e^{F}-1\right)\,e^{-\tilde{f}}\tilde{\omega}^n=0.
\end{equation*}
This function space is a hypersurface in the Banach space $\mathcal{C}^{2,\,2\alpha}_{X}(M)$. Notice that the tangent space at a function $F_0\in\left(\mathcal{C}^{2,\,2\alpha}_{X}(M)\right)_{\tilde{\omega},\,0}$ is the set of functions $u\in \mathcal{C}^{2,\,2\alpha}_{X}(M)$ with
\begin{equation*}
\int_Mu\,e^{F_0-\tilde{f}}\tilde{\omega}^n=0.
\end{equation*}
We have:
\begin{theorem}\label{Imp-Def-Kah-Ste}
Let $F_0\in\left(\mathcal{C}^{2,\,2\alpha}_{X}(M)\right)_{\tilde{\omega},\,0}\bigcap\mathcal{C}^{\infty}_{X}(M)$ and let $\psi_0\in\mathcal{M}^{\infty}_{X}(M)$ be a solution of the complex Monge-Amp\`ere equation
\begin{equation*}
\log\left(\frac{\tilde{\omega}^n_{\psi_0}}{\tilde{\omega}^n}\right)-\frac{X}{2}\cdot\psi_0=F_0.
\end{equation*}
Then for any $\alpha\in\left(0,\,\frac{1}{2}\right)$, there exists a neighbourhood $U_{F_0}\subset\left(C^{2,\,2\alpha}_{X}(M)\right)_{\tilde{\omega},\,0}$ of $F_0$ such that for all $F\in U_{F_0}$, there exists a unique function $\psi\in\mathcal{M}^{4,\,2\alpha}_{X}(M)$ such that
\begin{equation}
\log\left(\frac{\tilde{\omega}^n_{\psi}}{\tilde{\omega}^n}\right)-\frac{X}{2}\cdot\psi=F.\label{MA-neigh-small-per-pol}
\end{equation}
Moreover, if $F\in U_{F_0}$ lies in $\mathcal{C}^{\infty}_{X}(M)$, then the unique solution $\psi\in\mathcal{M}^{4,\,2\alpha}_{X}(M)$ to \eqref{MA-neigh-small-per-pol} lies in $\mathcal{M}^{\infty}_{X}(M).$
\end{theorem}

\begin{remark}
{Consideration of only finite regularity of the difference $\omega-\tilde{\omega}$
(which lowers the assumptions on the regularity of the coefficients of the drift Laplacian $\Delta_{\tilde{g},\,X}$)
and of the data $(\psi_0,\,F_0)$ would lead to a more refined version of Theorem \ref{Imp-Def-Kah-Ste}.}
\end{remark}

\begin{proof}[Proof of Theorem \ref{Imp-Def-Kah-Ste}]
In order to apply the implicit function theorem for Banach spaces, we must reformulate
the statement of Theorem \ref{Imp-Def-Kah-Ste} in terms of
the map $MA_{\tilde{\omega}}$ introduced formally at the beginning of Section \ref{linear}. To this end, consider the mapping
\begin{equation*}
\begin{split}
MA_{\tilde{\omega}}:\psi\in\mathcal{M}^{4,\,2\alpha}_{X}(M)\mapsto \log\left(\frac{\tilde{\omega}_{\psi}^n}{\tilde{\omega}^n}\right)-\frac{X}{2}\cdot\psi\in\left(\mathcal{C}^{2,\,2\alpha}_{X}(M)\right)_{\tilde{\omega},\,0},\qquad \alpha\in\left(0,\,\frac{1}{2}\right).
\end{split}
\end{equation*}
Notice that the function spaces above can be defined by either using the metric $\tilde{g}$ or $\tilde{g}_{t\psi_0}$ for any $t\in[0,\,1]$.
To see that $MA_{\tilde{\omega}}$ is well-defined, apply the Taylor expansion \eqref{equ:taylor-exp} to the background metric $\tilde{\omega}$
to obtain
\begin{equation}\label{reform-MA-op}
\begin{split}
MA_{\tilde{\omega}}(\psi)&=\log\left(\frac{\tilde{\omega}_{\psi}^n}{\tilde{\omega}^n}\right)-\frac{X}{2}\cdot\psi\\
&=\Delta_{\tilde{\omega}}\psi-\frac{X}{2}\cdot\psi-\int_0^1\int_0^{u}\arrowvert \partial\bar{\partial}\psi\arrowvert^2_{\tilde{g}_{t\psi}}\,dt\,du.
\end{split}
\end{equation}
Then by the very definition of $\mathcal{D}^{4,\,2\alpha}_{X}(M)$,
the first two terms of the last line of \eqref{reform-MA-op} lie in $\mathcal{C}^{2,\,2\alpha}_{X}(M)$.

Now, if $S$ and $T$ are tensors in $C^{2k,\,2\alpha}_{X,\,2}(M)$ and $C^{2k,\,2\alpha}_{X,\,2}(M)$ respectively,
then observe that $S\ast T$ lies in $C^{2k,\,2\alpha}_{X,\,4}(M)$,
where $\ast$ denotes any linear combination of contractions of tensors with respect to the metric $\tilde{g}$. Moreover,
\begin{equation}\label{mult-inequ-holder}
\|S\ast T\|_{C^{2k,\,2\alpha}_{X,\,4}(M)}\leq C(k,\,\alpha)\|S\|_{C^{2k,\,2\alpha}_{X,\,2}(M)}\cdot \|T\|_{C^{2k,\,2\alpha}_{X,\,2}(M)}.
\end{equation}
Next, notice that $$\arrowvert i\partial \bar{\partial}\psi\arrowvert^2_{\tilde{g}_{t\psi}}=\tilde{g}_{t\psi}^{-1}\ast \tilde{g}_{t\psi}^{-1}
\ast (\nabla^{\tilde{g}})^{\,2}\psi\ast(\nabla^{\tilde{g}})^{\,2}\psi$$ and that $$\tilde{g}_{t\psi}^{-1}
-\tilde{g}^{-1}\in C^{2,\,2\alpha}_{X,\,2}(M).$$
Thus, applying \eqref{mult-inequ-holder} twice to $S=T=(\nabla^{\tilde{g}})^{2}\psi$ and to the inverse $\tilde{g}_{t\psi}^{-1}$ with $k=1$, one finds that $\arrowvert i\partial\bar{\partial}\psi\arrowvert^2_{\tilde{g}_{t\psi}}\in C^{2,\,2\alpha}_{X,\,4}(M)\subseteq C^{2,\,2\alpha}_{X,\,2}(M)$ for each
$t\in[0,\,1]$ and that
\begin{equation*}
\left\|\int_0^1\int_0^{u}\arrowvert i\partial\bar{\partial}\psi\arrowvert^2_{\tilde{g}_{t\psi}}\,dt\,du
\right\|_{C^{2,\,2\alpha}_{X,\,\beta}(M)}\leq C\left(k,\alpha,\tilde{g}\right)\|\psi\|_{\mathcal{D}^{4,\,2\alpha}_{X,\,\beta}(M)},\\
\end{equation*}
so long as $\|\psi\|_{\mathcal{D}^{4,\,2\alpha}_{X,\,\beta}(M)}\leq 1$. Finally, the $JX$-invariance of the right-hand side of \eqref{reform-MA-op} is clear and Lemma \ref{lemma-preserved-int}(i) ensures that the function $$\exp MA_{\tilde{\omega}}(\psi)-1$$ has zero mean value with respect to the weighted measure $e^{-\tilde{f}}\tilde{\omega}^n$. Indeed, Lemma \ref{lemma-preserved-int}(i) applied to the linear path $\tilde{\omega}_{\tau}:=\tilde{\omega}+i\partial\bar{\partial}(\tau\psi)$ for $\tau\in[0,1]$ gives us that
\begin{equation*}
\int_M\left(\exp MA_{\tilde{\omega}}(\psi)-1\right)\,e^{-\tilde{f}}\tilde{\omega}^n=\int_Me^{-\tilde{f}_{\psi}}\tilde{\omega}_{\psi}^n-\int_Me^{-\tilde{f}}\tilde{\omega}^n=0.
\end{equation*}
By \eqref{equ:taylor-exp}, we have that
\begin{equation*}
\begin{split}
D_{\psi_0}MA_{\tilde{\omega}}:\psi\in &\,\mathcal{M}^{4,\,2\alpha}_{X}(M)\bigcap\left\{\int_{M}u\,e^{-\tilde{f}_{\psi_0}}\tilde{\omega}_{\psi_0}^{n}=0\right\}\mapsto \Delta_{\tilde{\omega}_{\psi_0}}\psi-\frac{X}{2}\cdot\psi\in T_{F_0}\left(\mathcal{C}^{2,\,2\alpha}_{X}(M)\right)_{\tilde{\omega},\,0},
\end{split}
\end{equation*}
where the tangent space of $\left(\mathcal{C}^{2,\,2\alpha}_{X}(M)\right)_{\tilde{\omega},\,0}$ at $F_0$ is equal to the set of functions $u\in \mathcal{C}^{2,\,2\alpha}_{X}(M)$ with $0$ mean value with respect to the weighted measure $e^{-\tilde{f}_{\psi_0}}\tilde{\omega}_{\psi_0}^n$.
Therefore, after applying Theorem \ref{iso-sch-Laplacian-pol} to the background metric $\tilde{\omega}_{\psi_0}$ in place of $\tilde{\omega}$, we conclude that
$D_{\psi_0}MA_{\tilde{\omega}}$ is an isomorphism of Banach spaces. The result now follows by applying the implicit function theorem
to the map $MA_{\tilde{\omega}}$ in a neighbourhood of $\psi_0\in\mathcal{M}^{4,\,2\alpha}_{X}(M)\cap\left\{\int_{M}u\,e^{-\tilde{f}_{\psi_0}}\tilde{\omega}_{\psi_0}^{n}=0\right\}$.\\

The proof of the regularity at infinity of the solution $\psi$ when the data $F$ lies in $\mathcal{C}^{\infty}_{X}(M)$ follows from a standard bootstrapping and will therefore be omitted; see Corollary \ref{coro-bds-Ck-weight} for the non-linear setting.
\end{proof}

\section{A priori estimates}\label{sec-a-priori-est}

\subsection{The continuity path}\label{continuitie}

In what follows, we identify $M\setminus E$ with $C_{0}\setminus\{o\}$ via $\pi$,
where $E\subset M$ and $o\in C_{0}$ are the exceptional set of the resolution and
the apex of the cone respectively. We now wish to solve
the complex Monge-Amp\`ere equation \eqref{ast-0} in the function space $\mathcal{M}_{X}^{\infty}(M)$ introduced in
Section \ref{function-spaces-subsection}. Recall that this is the equation
\begin{equation*}
\left\{
\begin{array}{rl}
(\omega+i\partial\bar{\partial}\psi)^{n}=e^{F+\frac{X}{2}\cdot\psi}\omega^{n},&\quad\textrm{$\psi\in\mathcal{M}_{X}^{\infty}(M)$ torus-invariant},\\
\int_{M}e^{F-f}\omega^{n}=\int_{M}e^{-f}\omega^{n}, &
\end{array} \right.
\end{equation*}
where $F:M\to\mathbb{R}$ is a torus-invariant smooth real-valued function equal with $F=c_{0}-\frac{s_{\omega_{0}}}{2}+O(r^{-4})$ with $g_{0}$-derivatives
outside a compact subset of $M$ and $f:M\to\mathbb{R}$ is the Hamiltonian potential of $X$ with respect to $\omega$, i.e., $-\omega\lrcorner JX=df$, normalised so that
\begin{equation*}
\Delta_{\omega}f-f+\frac{X}{2}\cdot f=-\frac{X}{2}\cdot F=O(r^{-2}).
\end{equation*}
From Lemma \ref{lemma-preserved-int}(i), it is clear that the integral condition, which determines $c_{0}$, is in fact a necessary condition for solving \eqref{ast-0}
in $\mathcal{M}_{X}^{\infty}(M)$. Define $F_{s}:=\log(1+s(e^{F}-1))$.

In this section, we prove Theorem \ref{mainthm2} by providing a solution to \eqref{ast-0} by implementing the continuity path
\begin{equation}\label{star-s}
\left\{
\begin{array}{rl}
(\omega+i\partial\bar{\partial}\psi_{s})^{n}=e^{F_{s}+\frac{X}{2}\cdot\psi_{s}}\omega^{n},&\quad\textrm{$\psi_{s}\in\mathcal{M}^{\infty}_{X}(M)$ torus-invariant},\qquad s\in[0,\,1],\\
\int_{M}e^{F-f}\omega^{n}=\int_{M}e^{-f}\omega^{n}, &\quad\int_{M}\psi_{s}\,e^{-f}\omega^{n}=0.
\end{array} \right.\tag{$\star_{s}$}
\end{equation}
When $s=0$, $(\star_{0})$ admits the trivial solution, namely $\psi_{0}\equiv0$. When $s=1$, $(\star_{1})$
corresponds to \eqref{ast-0}, that is, the equation that we wish to solve. Via the a priori estimates to follow, we will
show that the set $s\in[0,\,1]$ for which \eqref{star-s} has a solution is closed. As we have just seen, this set is
non-empty. Openness of this set follows from the isomorphism properties of the drift Laplacian given
by Theorem \ref{Imp-Def-Kah-Ste}. Connectedness of $[0,\,1]$ then implies that \eqref{star-s} has a solution for $s=1$, resulting
in the desired solution of \eqref{ast-0}.

\subsection{The continuity path re-parametrised}\label{sec-path-reparam}

To obtain certain localisation results and in turn, a priori estimates for \eqref{star-s}, we need to consider a reformulation of \eqref{star-s} in the following way.
Identify $M\setminus E$ and $C_0\setminus\{o\}$ via $\pi$, and define $F_{s}:=\log(1+s(e^{F}-1))$. Then there exists a compact subset $K\subset M$ containing $E$ such that
for all $s\in[0,\,1]$, $F_{s}$ is asymptotic to a constant $c_{s}$ on $M\setminus K$. Indeed, $c_{s}=\log(1+s(e^{c_{0}}-1))$, and in light of Proposition \ref{mainprop} we have that
\begin{equation}\label{second}
F_{s}-c_{s}=\log\left(1+\frac{se^{c_{0}}}{1+s(e^{c_{0}}-1)}(e^{F-c_{0}}-1)\right)=O(r^{-2})\quad\textrm{with $g_{0}$-derivatives.}
\end{equation}
Note that $c_{s}$ varies continuously as a function of $s$ and that \eqref{star-s} takes the form
$$(\omega+i\partial\bar{\partial}\psi_{s})^{n}=e^{F_{s}+\frac{X}{2}\cdot\psi_{s}}\omega^{n}.$$

Let $\eta_{s}:=-2c_{s}\log(r)$, a real-valued function defined on $M\setminus K$. Then, with $g$ denoting the K\"ahler metric associated to $\omega$, it is clear that
$$\|(\log(r))^{-1}\cdot\eta_{s}\|_{C^{0}(M\setminus K)}
+\|r\cdot d\eta_{s}\|_{C^{0}(M\setminus K,\,g_{0})}+\|r^{2}\cdot i\partial\bar{\partial}\eta_{s}\|_{C^{0}(M\setminus K,\,g_{0})}\leq C,$$
where $C$ is a positive constant uniform in $s\in[0,1]$ and so Lemma \ref{glue} infers the existence of a bump function $\chi:M\to\mathbb{R}$ supported on $M\setminus K$ and a compact subset $V\supset K$, both independent of $s$, such that $\chi=1$ on $M\setminus V$ and such that for all $s\in[0,\,1]$,
$\omega_{s}:=\omega+i\partial\bar{\partial}\left(\chi\cdot\eta_{s}\right)>0$ on $M$.
Define $\Phi_{s}:=\chi\cdot\eta_{s}$. Then $\omega_{s}=\omega+i\partial\bar{\partial}\Phi_{s}$
and since $\Phi_{s}=-2c_{s}\log r$ on $M\setminus V$, we find that outside a larger compact subset,
\begin{equation}\label{first}
\begin{split}
\omega_{s}&=\omega+i\partial\bar{\partial}\Phi_{s}\\
&=\omega-2c_{s} i\partial\bar{\partial}\log r\\
&=\omega_{0}+\rho_{\omega_{0}}-2c_{s}\omega^{T},
\end{split}
\end{equation}
where recall from Section \ref{conekahler} that $\omega^{T}$ is the transverse K\"ahler form. 
In particular, $\omega_{s}-\omega_{0}=O(r^{-2})$ with $g_{0}$-derivatives. Furthermore, we have that
\begin{equation*}
\begin{split}
\log\left(\frac{(\omega_{s}+i\partial\bar{\partial}(\psi_{s}-\Phi_{s}))^{n}}{\omega_{s}^{n}}\right)&-\frac{X}{2}\cdot(\psi_{s}-\Phi_{s})=
\log\left(\frac{(\omega+i\partial\bar{\partial}\psi_{s})^{n}}{\omega_{s}^{n}}\right)-\frac{X}{2}\cdot(\psi_{s}-\Phi_{s})\\
&=\log\left(\frac{(\omega+i\partial\bar{\partial}\psi_{s})^{n}}{\omega^{n}}\right)-\frac{X}{2}\cdot\psi_{s}
-\log\left(\frac{\omega_{s}^{n}}{\omega^{n}}\right)+\frac{X}{2}\cdot\Phi_{s}\\
&=F_s-\left(\log\left(\frac{\omega_{s}^{n}}{\omega^{n}}\right)-\frac{X}{2}\cdot\Phi_{s}\right)=:G_{s},
\end{split}
\end{equation*}
with
\begin{equation}\label{austin}
G_{s}=(F_{s}-c_{s})-\log\left(\frac{\omega_{s}^{n}}{\omega^{n}}\right)=O(r^{-2})\quad\textrm{with $g_{0}$-derivatives},
\end{equation}
using \eqref{second} and \eqref{first}. Set $\vartheta_{s}:=\psi_{s}-\Phi_{s}$. Then $\vartheta_{s}\in C^{\infty}_{X}(M)$ and
we can rewrite \eqref{star-s} in terms of $\vartheta_{s}$ as the equation
\begin{equation}\label{starstar-s}
\log\left(\frac{(\omega_{s}+i\partial\bar{\partial}\vartheta_{s})^{n}}{\omega_{s}^{n}}\right)-\frac{X}{2}\cdot\vartheta_{s}=G_{s},
\quad\vartheta_{s}\in C^{\infty}_{X}(M)\quad\textrm{torus-invariant},\quad\omega_{s}+i\partial\bar{\partial}\vartheta_{s}>0,\, s\in[0,\,1],
\tag{$\star\star_{s}$}
\end{equation}
with $G_{s}=O(r^{-2})$ and $\omega_{s}-\omega_{0}=O(r^{-2})$, both with $g_{0}$-derivatives.

Define $\sigma_{s}:=\omega_{s}+i\partial\bar{\partial}\vartheta_{s}$. Then in terms of the Ricci forms $\rho_{\sigma_{s}}$ and $\rho_{\omega_{s}}$ of
$\sigma_{s}$ and $\omega_{s}$ respectively, \eqref{starstar-s} yields
\begin{equation*}\label{wtf}
\rho_{\sigma_{s}}+\frac{1}{2}\mathcal{L}_{X}\sigma_{s}=\rho_{\omega_{s}}+\frac{1}{2}\mathcal{L}_{X}\omega_{s}-i\partial\bar{\partial}G_{s}.
\end{equation*}
We write $h_{s}$ for the K\"ahler metric associated to $\sigma_{s}$.

We will need the following lemma regarding the Hamiltonian potential $f_{\omega_{s}}$ of $X$ with respect to $\omega_{s}$.
\begin{lemma}\label{normal-fss}
Let $f_{\omega_{s}}:=f+\frac{X}{2}\cdot\Phi_{s}$. Then $-\omega_{s}\lrcorner JX=df_{\omega_{s}}$ and
there exists a torus-invariant function $H_{s}\in C^{\infty}(M)$ varying smoothly in $s$ and with $H_{s}=-c_{s}+O(r^{-2})$ with $g_{0}$-derivatives such that
\begin{equation}\label{normal-fs}
\Delta_{\omega_{s}}f_{\omega_{s}}-\frac{X}{2}\cdot f_{\omega_{s}}+f_{\omega_{s}}=H_{s}.
\end{equation}
\end{lemma}

\begin{proof}
The first assertion is clear. Regarding the normalisation condition \eqref{normal-fs}, a computation shows that for the Ricci forms $\rho_{\omega}$ and $\rho_{\omega_{s}}$ of $\omega$ and $\omega_{s}$ respectively,
\begin{equation*}
\begin{split}
\rho_{\omega_{s}}+\frac{1}{2}\mathcal{L}_{X}\omega_{s}-\omega_{s}&=\rho_{\omega}+\frac{1}{2}\mathcal{L}_{X}\omega-\omega-i\partial\bar{\partial}
\left(\log\left(\frac{\omega_{s}^{n}}{\omega^{n}}\right)-\frac{X}{2}\cdot\Phi_{s}+ \Phi_{s}\right)\\
&=i\partial\bar{\partial}(F+G_{s}-F_s-\Phi_{s}),
\end{split}
\end{equation*}
where we have used Proposition \ref{mainprop}(i). Write $Q_{s}:=F+G_{s}-F_s-\Phi_{s}$. Then $Q_{s}$ is torus-invariant and
it is easy to see that $$Q_{s}-(2c_{s}\log(r)-c_{s}+c_{0})=O(r^{-2})\quad\textrm{with $g_{0}$-derivatives},$$ where the constant is independent of $s$. Contracting the identity
\begin{equation}\label{macron}
\rho_{\omega_{s}}+\frac{1}{2}\mathcal{L}_{X}\omega_{s}-\omega_{s}=i\partial\bar{\partial}Q_{s}
\end{equation}
with $X^{1,\,0}:=\frac{1}{2}(X-iJX)$ and arguing as in Proposition \ref{mainprop}(ii) using the $JX$-invariance of the functions involved, we find that
$$\Delta_{\omega_{s}}f_{\omega_{s}}-\frac{X}{2}\cdot f_{\omega_{s}}+f_{\omega_{s}}+\frac{X}{2}\cdot Q_{s}$$
is constant on $M$. But since outside a large compact set,
\begin{equation*}
\begin{split}
\Delta_{\omega_{s}}f_{\omega_{s}}-\frac{X}{2}\cdot f_{\omega_{s}}+f_{\omega_{s}}+\frac{X}{2}\cdot Q_{s}&=\Delta_{\omega_{s}}f-\frac{X}{2}\cdot f+f=O(r^{-2})\\
\end{split}
\end{equation*}
as in \eqref{bonjour}, this constant must be zero. Hence the result follows with $H_{s}:=-\frac{X}{2}\cdot Q_{s}$.
\end{proof}

This allows for a normalisation for the Hamiltonian potential $f_{\sigma_{s}}:=f_{\omega_{s}}+\frac{X}{2}\cdot\vartheta_{s}$
of $X$ with respect to $\sigma_{s}$.
\begin{lemma}\label{lemma-tr-star-star}
Let $f_{\sigma_{s}}:=f_{\omega_{s}}+\frac{X}{2}\cdot\vartheta_{s}$. Then $-\sigma_{s}\lrcorner JX=df_{\sigma_{s}}$ and
for all $s\in[0,\,1]$, there exists a torus-invariant function $P_{s}\in C^{\infty}(M)$ varying smoothly in $s$ with
$P_{s}=O(r^{-2})$ with $g_{0}$-derivatives such that
\begin{equation}\label{normal-fsigmas}
\Delta_{\sigma_{s}}f_{\sigma_{s}}-\frac{X}{2}\cdot f_{\sigma_{s}}=-f+P_{s}.
\end{equation}
\end{lemma}

\begin{proof}
Again, the first assertion is clear. As for \eqref{normal-fsigmas}, we have that
\begin{equation*}
\begin{split}
\frac{X}{2}\cdot\log\left(\frac{\sigma_{s}^{n}}{\omega_{s}^{n}}\right)&=\frac{1}{2}\tr_{\sigma_{s}}
\mathcal{L}_{X}\sigma_{s}-\frac{1}{2}\tr_{\omega_{s}}\mathcal{L}_{X}\omega_{s}\\
&=\tr_{\sigma_{s}}(i\partial\bar{\partial}f_{\sigma_{s}})-\tr_{\omega_s}(i\partial\bar{\partial}f_{\omega_{s}})\\
&=\Delta_{\sigma_{s}}f_{\sigma_{s}}-\Delta_{\omega_{s}}f_{\omega_{s}}.
\end{split}
\end{equation*}
Thus, contracting both sides of \eqref{starstar-s} with $\frac{X}{2}$, we obtain
$$\Delta_{\sigma_{s}}f_{\sigma_{s}}-\Delta_{\omega_{s}}f_{\omega_{s}}=\frac{X}{2}\cdot G_{s}+\frac{X}{2}\cdot\left(f_{\omega_{s}}+\frac{X}{2}\cdot\vartheta_{s}\right)-\frac{X}{2}\cdot f_{\omega_{s}},$$
i.e.,
$$\Delta_{\sigma_{s}}f_{\sigma_{s}}-\frac{X}{2}\cdot f_{\sigma_{s}}=\Delta_{\omega_{s}}f_{\omega_{s}}-\frac{X}{2}\cdot f_{\omega_{s}}+\frac{X}{2}\cdot G_{s}.$$
Hence we derive from \eqref{normal-fs} that
$$\Delta_{\sigma_{s}}f_{\sigma_{s}}-\frac{X}{2}\cdot f_{\sigma_{s}}=H_{s}+\frac{X}{2}\cdot G_{s}-f_{\omega_{s}}.$$
With $P_{s}:=H_{s}+\frac{X}{2}\cdot G_{s}-\frac{X}{2}\cdot\Phi_{s}$, the result is now clear.
\end{proof}

\subsection{Summary of notation}

For convenience, we provide in this section a summary of our notation regarding the various K\"ahler forms in play.
\begin{itemize}
\item $(C_{0},\,\omega_{0})$ is a given toric K\"ahler cone with radial function $r$ and K\"ahler form $\omega_{0}$.
\item $\pi:M\to(C_{0},\,\omega_{0})$ is an equivariant resolution of $C_{0}$, and we have that $d\pi(X)=r\partial_{r}$, where $X$ is the hypothetical soliton vector field.
 \item $F$ is the data in \eqref{ast-0} equal to $c_{0}+O(r^{-2})$ with $g_{0}$-derivatives at infinity, with $c_{0}\in\mathbb{R}$ a constant.
  \item $\omega$ is the background K\"ahler form given in \eqref{ast-0} isometric to $\omega_{0}+\rho_{\omega_{0}}$ outside a fixed compact subset of $M$. Here, $\rho_{\omega_{0}}$ is the Ricci form of $\omega_{0}$.
  \item $g$ is the K\"ahler metric associated to $\omega$.
  \item $f$ is the Hamiltonian potential of $JX$ with respect to $\omega$ given in Proposition \ref{mainprop}(ii). It is equal to $\frac{r^{2}}{2}-n$ outside the compact subset and normalised so that
  $$\Delta_{\omega}f+f-\frac{X}{2}\cdot f=O(r^{-2})\quad\textrm{with $g_{0}$-derivatives}.$$
 \item $c_{s}:=\log(1+s(e^{c_{0}}-1))$.
  \item $F_{s}$ is the data in \eqref{star-s} equal to $c_{s}+O(r^{-2})$ outside a fixed compact subset of $M$.
\item $\psi_{s}$ is the solution to the original continuity path \eqref{star-s}.
\item $\Phi_{s}=-2\chi\cdot c_{s}\log r$, where $0\leq\chi\leq1$ is a bump function identically equal to $1$ outside a fixed compact subset of $M$. In particular, notice that
$\Phi_{s}=-c_{s}\log(2(f+n))$ outside a compact subset of $M$.
   \item $\omega_{s}:=\omega+i\partial\bar{\partial}\Phi_{s}$ is the one-parameter family of background metrics equal to $\omega_{0}+O(r^{-2})$ with $g_{0}$-derivatives at infinity
   appearing in \eqref{starstar-s}.
\item $g_{s}$ is the K\"ahler metric associated to $\omega_{s}$.
    \item $f_{\omega_{s}}:=f+\frac{X}{2}\cdot\Phi_{s}$ is the Hamiltonian potential of $JX$ with respect to $\omega_{s}$.
\item $\vartheta_{s}=\psi_s - \Phi_s$ is the solution of the re-parametrised continuity path \eqref{starstar-s}.
    \item $\sigma_{s}:=\omega_{s}+i\partial\bar{\partial}\vartheta_{s}=\omega+i\partial\bar{\partial}\left(\Phi_{s}+\vartheta_{s}\right)$ is the associated K\"ahler form.
    \item $h_{s}$ is the K\"ahler metric associated to $\sigma_{s}$.
        \item $f_{\sigma_{s}}:=f_{\omega_{s}}+\frac{X}{2}\cdot\vartheta_{s}=f+\frac{X}{2}\cdot\left(\Phi_{s}+\vartheta_{s}\right)$ is the Hamiltonian potential of $JX$ with respect to $\sigma_{s}$. It is normalised by the equation
    $$\Delta_{\sigma_{s}}f_{\sigma_{s}}-\frac{X}{2}\cdot f_{\sigma_{s}}=-f+P_{s},$$
where $P_{s}=O(r^{-2})$ with $g_{0}$-derivatives.

\end{itemize}

\subsection{A priori lower bound on the radial derivative}\label{sec-low-bd-rad-der}

The fact that the data $G_{s}$ of \eqref{starstar-s} is bounded allows us to localise the infimum of $X\cdot\vartheta_{s}$ using the minimum principle. This
leads to a uniform lower bound on $X\cdot\vartheta_{s}$, and in particular on $X\cdot\psi_{s}$.

\begin{lemma}[Localising the infimum of the radial derivative]\label{lemma-loc-crit-pts-rad-der-inf}
Let $(\vartheta_s)_{0\,\leq\, s\,\leq\, 1}$  be a path of solutions in $C^{\infty}_{X}(M)$ to \eqref{starstar-s}.
Then there exists a compact subset $K\subset M$ and $C>0$ such that for all $s\in[0,\,1]$,
$\inf_M X\cdot\vartheta_{s}\geq\min\{0\,,\,\min_{K}X\cdot\vartheta_{s}\}-C$.
\end{lemma}

\begin{proof}
We start by differentiating \eqref{starstar-s} along the vector field $X$. This gives
\begin{equation}
\begin{split}\label{dont-forget-me-0}
\Delta_{\sigma_{s},\,X}\left(\frac{X}{2}\cdot\vartheta_{s}\right)&=-\operatorname{tr}_{\sigma_{s}}\left(\mathcal{L}_{\frac{X}{2}}\omega_s\right)+\operatorname{tr}_{\omega_{s}}\mathcal{L}_{\frac{X}{2}}\omega_s+\frac{X}{2}\cdot G_{s}.
\end{split}
\end{equation}
As $\mathcal{L}_{\frac{X}{2}}\omega_s=\mathcal{L}_{\frac{X}{2}}\omega=\omega_0$ outside a fixed compact subset $K\subset M$ thanks to Proposition \ref{mainprop}(i), the arithmetic-geometric inequality applied to the first term on the right-hand side of \eqref{dont-forget-me-0} yields the inequality
\begin{equation}
\begin{split}\label{dont-forget-me}
\Delta_{\sigma_{s},\,X}\left(\frac{X}{2}\cdot\vartheta_{s}\right)&=-\operatorname{tr}_{\sigma_{s}}\omega_{0}+\operatorname{tr}_{\omega_{s}}\omega_{0}+\frac{X}{2}\cdot G_{s}\\
&\leq -n\left(\frac{\omega_0^n}{\sigma_s^n}\right)^{\frac{1}{n}}+\operatorname{tr}_{\omega_{s}}\omega_{0}+\frac{X}{2}\cdot G_{s}\\
&=-n\left(\frac{\omega_s^n}{\sigma_s^n}\right)^{\frac{1}{n}}\cdot\left(\frac{\omega_0^n}{\omega_s^n}\right)^{\frac{1}{n}}+\operatorname{tr}_{\omega_{s}}\omega_{0}+\frac{X}{2}\cdot G_{s}\\
&=-ne^{-\frac{G_s}{n}-\frac{X\cdot \vartheta_{s}}{2n}}\cdot\left(\frac{\omega_0^n}{\omega_s^n}\right)^{\frac{1}{n}}+\operatorname{tr}_{\omega_{s}}\omega_{0}+\frac{X}{2}\cdot G_{s}\\
&\leq -C^{-1}e^{-\frac{X\cdot \vartheta_{s}}{2n}}+C
\end{split}
\end{equation}
for some uniform positive constant $C$ independent of $s\in[0,\,1]$. Here we have used \eqref{starstar-s} in the fourth line. Since $X\cdot \vartheta_{s}\to0$ at infinity, either
$\inf_M X\cdot\vartheta_{s}\geq0$ or $X\cdot\vartheta_{s}$ attains a global minimum at some point $p\in M$. If $p\in M\setminus W$, then from the above inequality we find after applying the minimum principle that
$$0\leq-C^{-1}e^{-\frac{X\cdot \vartheta_{s}(p)}{2n}}+C,$$
i.e., $X\cdot \vartheta_{s}>-C$ for some $C>0$ independent of $s$. The result is now clear.
\end{proof}

Next, we have:
\begin{prop}\label{lowerbound}
There exists a positive constant $C$ such that for all $s\in[0,\,1]$, $X\cdot\vartheta_s\geq-C-2f$. In particular, $X\cdot\psi_{s}>-C-2f$ for all $s\in[0,\,1]$.
\end{prop}

\begin{proof}
In order to prove that $X\cdot\vartheta_{s}$ is uniformly bounded from below, first note that since $X\cdot\Phi_{s}$ is bounded and
$X\cdot\vartheta_{s}$ tends to zero at infinity, $f_{\sigma_{s}}:=f+\frac{X}{2}\cdot\Phi_{s}+\frac{X}{2}\cdot\vartheta_{s}$ is a proper function bounded from below
by virtue of the fact that $f$ is by Proposition \ref{mainprop}(ii). Then since $X=\nabla^{h_{s}}f_{\sigma_{s}}$, $f_{\sigma_{s}}$ must attain its global minimum at a point lying in the
zero set of $X$ and hence must coincide with the global minimum of $f$ on this set; that is to say,
\begin{equation}\label{vladimir}
f_{\sigma_{s}}\geq\min_{\{X\,=\,0\}}f_{\sigma_{s}}=\min_{\{X\,=\,0\}}f\geq -C.
\end{equation}
The lower bound on $X\cdot\vartheta_{s}$ now follows.
\end{proof}

These last two results now allow us to derive a lower bound on $X\cdot\vartheta_{s}$, hence on $X\cdot\psi_{s}$.
\begin{corollary}\label{lowerbound2}
There exists a positive constant $C$ such that for all $s\in[0,\,1]$, $X\cdot\vartheta_s\geq -C$. In particular, $X\cdot\psi_{s}>-C$ for all $s\in[0,\,1]$.
\end{corollary}

\begin{proof}
This follows immediately from the localisation of the infimum of $X\cdot\vartheta_{s}$ given by Lemma \ref{lemma-loc-crit-pts-rad-der-inf}, together with
the lower bound on $X\cdot\vartheta_{s}$ given by Proposition \ref{lowerbound}.
\end{proof}

\subsection{A priori $C^{0}$-estimate}\label{sec-a-priori-energy}

We proceed with the a priori estimate on the $C^{0}$-norm of $(\vartheta_{s})_{0\,\leq \,s\,\leq\, 1}$
which is uniform in $s\in[0,\,1]$. We begin with two crucial observations, both localisation results for the global extrema of $\vartheta_{s}$.
We first localise the supremum of $\vartheta_{s}$.

\begin{lemma}[Localising the supremum of a solution of \eqref{starstar-s}]\label{lemma-loc-crit-pts-max}
Let $(\vartheta_s)_{0\,\leq\, s\,\leq\, 1}$  be a path of solutions in $C^{\infty}_{X}(M)$ to \eqref{starstar-s}. Then
there exists a compact subset $K\subset M$ such that for all $s\in[0,\,1]$, $\sup_M\vartheta_{s}=\max_{K}\vartheta_{s}$.
\end{lemma}

\begin{proof}
First observe from \eqref{starstar-s} {and the basic inequality $\log(1+x)\leq x$ for all $x>-1$} that $\vartheta_{s}$ is a subsolution of the following differential inequality:
\begin{equation*}
\Delta_{\omega_{s}}\vartheta_{s}-\frac{X}{2}\cdot\vartheta_{s}\geq G_{s}.
\end{equation*}
Next, recall from Lemma \ref{normal-fss} that $$\Delta_{\omega_{s}}f_{\omega_{s}}-\frac{X}{2}\cdot f_{\omega_{s}}+f_{\omega_{s}}=-c_{s}+O(r^{-2}).$$
Choose $C>0$ sufficiently large so that $f_{\omega_{s}}+C>0$ on $M$. Then
$$\Delta_{\omega_{s},\,X}((f_{\omega_{s}}+C)^{-1})=(f_{\omega_{s}}+C)^{-1}+O(f_{\omega_{s}}^{-2}).$$
Let $A>0$ be a constant to be fixed later. Then one has on $M$ that
\begin{equation*}
\begin{split}
\Delta_{\omega_{s},\,X}\left(\vartheta_{s}+A(f_{\omega_{s}}+C)^{-1}\right)&\geq
A(f_{\omega_{s}}+C)^{-1}+G_{s}+A\cdot O(f_{\omega_{s}}^{-2})\\
&>0
\end{split}
\end{equation*}
outside a fixed compact subset $K\subset M$, for $A$ chosen sufficient large.
One can check that
$$\Delta_{\omega_{s},\,X}(\log (f_{\omega_{s}}+C))=-1+O(f_{\omega_{s}}^{-1}).$$
Then on $M\setminus K$ one has that
\begin{equation}\label{great}
\begin{split}
\Delta_{\omega_{s},\,X}\left(\vartheta_{s}+A(f_{\omega_{s}}+C)^{-1}-\varepsilon\log (f_{\omega_{s}}+C)\right)&\geq \varepsilon(1+O(f^{-1}_{\omega_{s}}))>\frac{\varepsilon}{2}>0,
\end{split}
\end{equation}
after enlarging $K$ if necessary. Now, $\vartheta_{s}$ being bounded on $M$ and $f_{\sigma_{s}}$ being proper and bounded from below
implies that for each $s\in[0,\,1]$, the function
$$\vartheta_{s}+A(f_{\omega_{s}}+C)^{-1}-\varepsilon\log (f_{\omega_{s}}+C)$$ is bounded from above and tends to $-\infty$ as $r\to\infty$. In particular, the latter function must attain a global maximum on $M$. The maximum principle applied to \eqref{great} then ensures that it must be attained in $K$, i.e.,
$$\max_{M}(\vartheta_{s}+A(f_{\omega_{s}}+C)^{-1}-\varepsilon\log (f_{\omega_{s}}+C))=\max_{K}(\vartheta_{s}+A(f_{\omega_{s}}+C)^{-1}-\varepsilon\log (f_{\omega_{s}}+C)).$$
In conclusion, we have that
$$\vartheta_{s}(x)\leq -A(f_{\omega_{s}}(x)+C)^{-1}+\varepsilon\log (f_{\omega_{s}}(x)+C)+\max_{K}(\vartheta_{s}+A(f_{\omega_{s}}+C)^{-1}-\varepsilon\log (f_{\omega_{s}}+C))\qquad\textrm{for all $x\in M$},$$
which leads to the upper bound
$$\vartheta_{s}(x)\leq\max_{K}(\vartheta_{s}+A(f_{\omega_{s}}+C)^{-1})-A(f_{\omega_{s}}(x)+C)^{-1}\qquad\textrm{for all $x\in M$}$$
by letting $\varepsilon\to0$ and making use of the fact that $\varepsilon\log (f_{\omega_{s}}+C)$, being locally bounded, converges to zero on compact subsets of $M$ as $\varepsilon\to0$. Since this holds true for any $x\in M$, the desired estimate follows.
\end{proof}

We next localise the infimum of $\vartheta_{s}$.
\begin{lemma}[Localising the infimum of a solution of \eqref{starstar-s}]\label{lemma-loc-crit-pts}
Let $(\vartheta_s)_{0\,\leq\, s\,\leq\, 1}$  be a path of solutions in $C^{\infty}_{X}(M)$ to \eqref{starstar-s}. Then
there exists a compact subset $K\subset M$ and a positive constant $C$ such that for all $s\in[0,\,1]$, $\inf_M\vartheta_{s}\geq \min_{K}\vartheta_{s}-C$.
\end{lemma}

\begin{proof}
Let $C>0$ be such that $f_{\sigma_s}+C\geq 1$ on $M$ for all $s\in[0,\,1]$. This is possible thanks to \eqref{vladimir}. Consider
the function $\vartheta_s-A(f_{\sigma_s}+C)^{-1}$ on $M$ for $A>0$ to be chosen later. Thanks to the basic inequality $\log(1+x)\leq x$ for all $x>-1$ applied to \eqref{starstar-s}, we
know that $$\Delta_{\sigma_s,\,X}\vartheta_{s}:=\Delta_{\sigma_s}\vartheta_{s}-\frac{X}{2}\cdot\vartheta_{s}\leq G_{s}.$$ We therefore find that
\begin{equation}
\begin{split}\label{ohmymy}
\Delta_{\sigma_s,\,X}\left(\vartheta_s-A(f_{\sigma_s}+C)^{-1}\right)&\leq G_s -A\Delta_{\sigma_s,\,X}[(f_{\sigma_s}+C)^{-1}]\\
&=G_s-A\Biggl(-\frac{\Delta_{\sigma_s,\,X}f_{\sigma_s}}{(f_{\sigma_s}+C)^2}+\underbrace{\frac{X\cdot f_{\sigma_s}}{(f_{\sigma_s}+C)^{3}}}_{\geq\,0}\Biggr)\\
&\leq \frac{C'}{f+C}+A\frac{\Delta_{\sigma_s,\,X}f_{\sigma_s}}{(f_{\sigma_s}+C)^2}\\
&= \frac{C'}{f+C}-A\frac{(f-P_s)}{(f_{\sigma_s}+C)^2}
\end{split}
\end{equation}
for some $C'>0$, where we have used the asymptotics of $G_s$ from \eqref{starstar-s} in the third line, together with Lemma \ref{lemma-tr-star-star} in the last line.

Next, let $\varepsilon>0$ and consider the function $\vartheta_s-A(f_{\sigma_s}+C)^{-1}+\varepsilon f_{\sigma_s}$. By definition of the relevant function spaces, for each $s\in[0,\,1]$,
the function $\vartheta_s-A(f_{\sigma_s}+C)^{-1}+\varepsilon f_{\sigma_s}$ is proper and bounded from below. In particular, this function attains a global minimum on $M$. At such a point, observe that
\begin{equation}\label{ohmymybis}
0=X\cdot\left(\vartheta_s-A(f_{\sigma_s}+C)^{-1}+\varepsilon f_{\sigma_s}\right)=X\cdot\vartheta_s
+\left(A(f_{\sigma_s}+C)^{-2}+\varepsilon\right)\underbrace{X\cdot f_{\sigma_s}}_{\geq\, 0}\geq X\cdot\vartheta_s.
\end{equation}
Now, thanks to \eqref{ohmymy} and Lemma \ref{lemma-tr-star-star}, we know that
\begin{equation*}
\begin{split}
\Delta_{\sigma_s,\,X}\left(\vartheta_s-A(f_{\sigma_s}+C)^{-1}+\varepsilon f_{\sigma_s}\right)&\leq \frac{C'}{f_{\omega_s}+C}-A\frac{(f-P_s)}{(f_{\sigma_s}+C)^2}+\varepsilon (P_s-f)\\
&\leq \frac{C'}{f_{\omega_s}+C}-A\frac{(f-\sup_{M\times[0,\,1]}P_s)}{(f_{\sigma_s}+C)^2}+\varepsilon\left(\sup_{M\times[0,\,1]} P_s-f\right)\\
&\leq \frac{C'}{f_{\omega_s}+C}-A\frac{(f-\sup_{M\times[0,\,1]}P_s)}{(f_{\sigma_s}+C)^2},
\end{split}
\end{equation*}
provided that $f\geq \sup_{M\times[0,1]}P_s$. If a global minimum of the aforementioned function lies in the set $\{f\geq \sup_{M\times[0,1]}P_s\}$, then
the minimum principle applied to \eqref{ohmymy} tells us that at this point,
\begin{equation*}
\begin{split}
A\left(f-\sup_{M\times[0,1]}P_s\right)(f_{\omega_s}+C)\leq  C'(f_{\sigma_s}+C)^2\leq C'(f_{\omega_s}+C)^2
\end{split}
\end{equation*}
by virtue of \eqref{ohmymybis}. In particular, after setting $A:=2C'$, we find that $f\leq 2\sup_{M\times[0,1]}P_s+C$ at such a point for a possibly larger uniform positive constant $C$. We subsequently deduce that
if a global minimum of the function in question lies in the set $\{f\geq \sup_{M\times[0,1]}P_s\}$, then for all $\varepsilon>0$,
\begin{equation*}
\min_M\left(\vartheta_s-2C'(f_{\sigma_s}+C)^{-1}+\varepsilon f_{\sigma_s}\right)=\min_{\{f\,\leq\, 2\sup_{M\times[0,1]}P_s+C\}}\left(\vartheta_s-2C'(f_{\sigma_s}+C)^{-1}+\varepsilon f_{\sigma_s}\right).
\end{equation*}
Since $\varepsilon>0$ is arbitrary, letting $\varepsilon\to0$ gives us that
\begin{equation*}
\begin{split}
\vartheta_s&\geq \min_{\{f\,\leq\, 2\sup_{M\times[0,1]}P_s+C\}}\left(\vartheta_s-2C'(f_{\sigma_s}+C)^{-1}\right)+\underbrace{2C'(f_{\sigma_s}+C)^{-1}}_{\geq\,0}\\
&\geq \min_{\{f\,\leq\, 2\sup_{M\times[0,1]}P_s+C\}}\vartheta_s+\min_{\{f\,\leq\, 2\sup_{M\times[0,1]}P_s+C\}}\left(-2C'(f_{\sigma_s}+C)^{-1}\right)\\
&\geq \min_{\{f\,\leq\, 2\sup_{M\times[0,1]}P_s+C\}}\vartheta_s-2C',
\end{split}
\end{equation*}
because $f_{\sigma_s}+C\geq 1$ on $M$ for all $s\in[0,1]$. From this observation, the desired localisation follows.
\end{proof}

\subsubsection{Aubin-Tian-Zhu's functionals}

We now introduce two functionals that have been defined and used by
Aubin \cite{Aub-red-Cas-Pos}, Bando and Mabuchi \cite{Ban-Mab-Uni}, and Tian \cite[Chapter $6$]{Tian-Can-Met-Boo}
in the study of Fano manifolds, and by Tian and Zhu \cite{Tian-Zhu-I} in the study of shrinking gradient K\"ahler-Ricci solitons
on compact K\"ahler manifolds.
\begin{definition}\label{IJ}
Let $(\varphi_t)_{0\,\leq\, t\,\leq\, 1}$ be a $C^1$-path in $\mathcal{M}^{\infty}_{X}(M)$ from $\varphi_{0}=0$ to $\varphi_{1}=\varphi$.
We define the following two generalised weighted energies:
\begin{equation*}
\begin{split}
I_{\omega,\,X}(\varphi)&:=\int_M\varphi\left(e^{-f}\omega^n-e^{-f-\frac{X}{2}\cdot\varphi}\omega_{\varphi}^n\right),\\
J_{\omega,\,X}(\varphi)&:=\int_0^1\int_M\dot{\varphi_s}\left(e^{-f}\omega^n-e^{-f-\frac{X}{2}\cdot\varphi_s}\omega_{\varphi_s}^n\right)\wedge ds.
\end{split}
\end{equation*}
\end{definition}

At first sight, these two functionals resemble relative weighted mean values of a potential $\varphi$ in
$\mathcal{M}^{\infty}_{X}(M)$ or of
a path $(\varphi_t)_{0\,\leq\, t\,\leq\, 1}$ in $\mathcal{M}^{\infty}_{X}(M)$ respectively. When $X\equiv 0$ and $(M,\,\omega)$ is a compact K\"ahler manifold,
an integration by parts together with some algebraic manipulations (see Aubin's seminal paper \cite{Aub-red-Cas-Pos} or Tian's book \cite[Chapter $6$]{Tian-Can-Met-Boo}) show that
\begin{equation}
\begin{split}\label{formulae--fct-I-J-ein}
I_{\omega,\,0}(\varphi)&=\sum_{k\,=\,0}^{n-1}\int_{M}i\partial\varphi\wedge\bar{\partial}\varphi\wedge\omega^k\wedge\omega_{\varphi}^{n-1-k},\\
J_{\omega,\,0}(\varphi)&=\sum_{k\,=\,0}^{n-1}\frac{k+1}{n+1}\int_{M}i\partial\varphi\wedge\bar{\partial}\varphi\wedge\omega^{k}\wedge\omega_{\varphi}^{n-1-k}.
\end{split}
\end{equation}
This justifies the description of $I_{\omega,\,0}(\varphi)$ and $J_{\omega,\,0}(\varphi)$ as modified energies.
Moreover, it demonstrates that on a compact K\"ahler manifold $J_{\omega,\,0}$ is a true functional, that is to say, it does not depend on the choice of path.

Such formulae \eqref{formulae--fct-I-J-ein} for $I_{\omega,\,X}$ and $J_{\omega,\,X}$
for a non-vanishing vector field $X$ and a non-compact K\"ahler manifold $(M,\,\omega)$
do not seem to be readily available for a good reason: the exponential function is not algebraic.
However, following Tian and Zhu's work \cite{Tian-Zhu-I}, one can
prove that the essential properties shared by both
$I_{\omega,\,0}$ and $J_{\omega,\,0}$ hold true for a non-vanishing vector field $X$ in a non-compact setting. The proof follows exactly as in \cite[Theorem 7.5]{conlon33}.

\begin{theorem}\label{main-thm-I-J}
$I_{\omega,\,X}(\varphi)$
and $J_{\omega,\,X}(\varphi)$ are well-defined for $\varphi\in\mathcal{M}^{\infty}_{X}(M)$.
Moreover, $J_{\omega,\,X}$ does not depend on the choice of a $C^{1}$-path $(\varphi_t)_{0\,\leq\, t\,\leq\, 1}$ in
$\mathcal{M}^{\infty}_{X}(M)$ from $\varphi_{0}=0$ to $\varphi_{1}=\varphi$, hence defines
a functional on $\mathcal{M}^{\infty}_{X}(M)$. Finally, the first variation of the difference $(I_{\omega,\,X}-J_{\omega,\,X})$ is
given by
\begin{equation*}
\frac{d}{dt}\left(I_{\omega,\,X}-J_{\omega,\,X}\right)(\varphi_t)=-\int_M\varphi_t\left(\Delta_{\omega_{\varphi_t}}\dot{\varphi_t}-
\frac{X}{2}\cdot\dot{\varphi_t}\right)\,e^{-f_{\varphi_t}}\omega_{\varphi_t}^n,
\end{equation*}
where $f_{\varphi_t}:=f+\frac{X}{2}\cdot\varphi_t$ satisfies
$X=\nabla^{\omega_{\varphi_t}}f_{\varphi_t}$
and where $(\varphi_t)_{0\,\leq\, t\,\leq\, 1}$ is any $C^1$-path in $\mathcal{M}^{\infty}_{X}(M)$ from $\varphi_{0}=0$ to $\varphi_{1}=\varphi$.
\end{theorem}

Recall that the equation we wish to solve is \eqref{star-s}, that is,
\begin{equation*}
e^{-f_{\psi_{s}}}\omega^{n}_{\psi_{s}}=e^{F_s-f}\omega^{n},
\end{equation*}
where $\omega_{\psi_{s}}:=\omega+i\partial\bar{\partial}\psi_{s}>0$
and $f_{\psi_{s}}:=f+\frac{X}{2}\cdot\psi_{s}$. This pair satisfies $-\omega_{\psi_{s}}\lrcorner X=df_{\psi_{s}}$.
We work under the assumption that $\int_{M}\psi_{s}\,e^{-f}\omega^{n}=0$. We have the following uniform energy bound.
\begin{prop}[A priori energy estimates]\label{prop-a-priori-ene-est}
Let $(\psi_s)_{0\,\leq\, s\,\leq\, 1}$ be a path of solutions in $\mathcal{M}^{\infty}_{X}(M)$ to \eqref{star-s}. Then for all $p\in(1,2)$, there exists a positive constant $C=C\left(n,p,\omega,\sup_{s\in[0,1]}\|F_s\|_{C^{0}(M)}\right)$ such that
\begin{equation*}
\sup_{0\,\leq\, s\,\leq\, 1}\int_M|\psi_s-\overline{\psi}_{s}|^p\,e^{-f}\omega^n\leq C,
\end{equation*}
where $$\overline{\psi}_{s}:=\frac{1}{\int_{M}e^{-f}\omega^{n}}\int_{M}\psi_{s}\,e^{-f}\omega^{n}=\fint_{M}\psi_{s}\,e^{-f}\omega^{n}.$$
In particular, if $\overline{\psi}_{s}=0$, then
\begin{equation*}
\sup_{0\,\leq\, s\,\leq\, 1}\int_M|\vartheta_{s}|^p\,e^{-f}\omega^n\leq C.
\end{equation*}
\end{prop}

\begin{proof}
The proof is verbatim that of \cite[Proposition 7.9]{ccd2}, using the Poincar\'e inequality of Proposition \ref{poincare},
Lemma \ref{lemma-preserved-int}(i), the lower bound given by Corollary \ref{lowerbound2}, and Theorem \ref{main-thm-I-J}.
\end{proof}

\subsubsection{A priori estimate on $\sup_M\vartheta_{s}$}
Let $\vartheta_s$ be a solution to \eqref{starstar-s} for some fixed value of the parameter $s\in[0,\,1]$.
We next give an upper bound for $\sup_M\vartheta_s$ that is uniform in $s\in[0,\,1]$.

\begin{prop}[A priori upper bound on $\sup_M\vartheta$]\label{prop-bd-abo-uni-psi}
Let $(\vartheta_s)_{0\,\leq\, s\,\leq\, 1}$ be a path of solutions in $C^{\infty}_{X}(M)$ to \eqref{starstar-s}. Then there exists a positive constant $C=C\left(n,\omega,\sup_{s\in[0,1]}\|G_s\|_{C^{0}(M)}\right)$ such that
\begin{equation*}
\sup_{0\,\leq\, s\,\leq\, 1}\sup_{M}\vartheta_s\leq C.
\end{equation*}
\end{prop}

\begin{proof}
This is verbatim the proof of \cite[Proposition 7.11]{ccd2}, using the localisation Lemma \ref{lemma-loc-crit-pts-max}
and the energy bound Proposition \ref{prop-a-priori-ene-est}.
\end{proof}

\subsubsection{A priori estimate on $\inf_M\vartheta_{s}$}\label{infmyballs}
We now derive a lower bound for $\inf_M\vartheta_s$ that is uniform in $s\in[0,\,1]$. This exploits the toric geometry of the situation.

\subsubsection*{An upper bound on the $I_{\omega,\,X}$-functional}
We first show that the $I_{\omega,\,X}$-functional is bounded along the continuity path.
\begin{lemma}\label{I-bounded}
$\sup_{s\,\in\,[0,\,1]}I_{\omega,\,X}(\psi_{s})\leq C(\sup_M(\vartheta_{s})_+).$
\end{lemma}

\begin{proof}
This is verbatim the proof of \cite[Lemma 7.12]{ccd2}.
\end{proof}

\subsubsection*{An upper bound on the weighted $L^{p}$-norm of the gradient of the Legendre transform}

Recall the continuity path \eqref{star-s}:
\begin{equation*}
(\omega+i\partial\bar{\partial}\psi_{s})^{n}=e^{F_{s}+\frac{X}{2}\cdot\psi_{s}}\omega^{n},\qquad s\in[0,\,1],
\end{equation*}
where $$F_{s}:=\log\left(se^{F}+(1-s)\right)\qquad\textrm{and}\qquad i\partial\bar{\partial}F=\rho_{\omega}+\frac{1}{2}\mathcal{L}_{X}\omega-\omega.$$
Here, $\rho_{\omega}$ denotes the Ricci form of $\omega$ and $F\in C^{\infty}(M)$ is bounded.
On $\mathfrak{t}\simeq\mathbb{R}^{n}$ we have coordinates $\xi:=(\xi_{1},\ldots,\xi_{n})$, induced coordinates $x=(x_{1},\ldots,x_{n})$ on $\mathfrak{t}^{*}$ which contains the image
of the moment map, and we can write $\omega=2i\partial\bar{\partial}\phi_{0}$ for a convex function $\phi_{0}$ on $\mathbb{R}^{n}\simeq\mathfrak{t}$
up to the addition of a linear function (cf.~Section \ref{toric-geom}). Let $b_{X}\in\mathbb{R}^{n}$ denote the vector field $JX\in\mathfrak{t}$ as in \eqref{eqnY4},
write $\nabla$ for the Levi-Civita connection of the flat metric on $\mathbb{R}^{n}$, and $\langle\cdot\,,\,\cdot\rangle$ for the corresponding inner product.

As in \cite[equation (2.12)]{ccd2}, we normalise $\phi_{0}$ so that
\begin{equation}\label{bbq}
F=-\log\det(\phi_{0,\,ij})+\langle\nabla\phi_0,\,b_X \rangle-2\phi_0.
\end{equation}
Set $\phi_s:=\phi_{0}+\frac{1}{2}\psi_{s}$. Then in the coordinates $\xi$ on $\mathbb{R}^{n}$, equation \eqref{star-s} reads as
$$\det(\phi_{s,\,ij})=\left(se^{F} + (1-s)\right)e^{\langle \nabla \phi_s,\,b_X \rangle - \langle \nabla \phi_0,\,b_X \rangle}\det(\phi_{0,\,ij}),\qquad s\in[0,\,1].$$
Plugging in the definition of $F$, this becomes
\begin{equation*}
\begin{split}
\det(\phi_{s,\,ij})&=\left(se^{-2\phi_0 - \log\det(\phi_{0,\,ij})} + (1-s)e^{-\langle\nabla\phi_0, b_X\rangle}\right)e^{\langle \nabla \phi_s,\,b_X \rangle}\det(\phi_{0,\,ij}) \\
				&= \left(se^{-2\phi_0} + (1-s)e^{-\langle\nabla \phi_0,\,b_X \rangle}\det(\phi_{0,\,ij}) \right)e^{\langle\nabla\phi_s,\,b_X \rangle},\qquad s\in[0,\,1],
\end{split}
\end{equation*}
or equivalently,
\begin{equation}\label{caddo}
e^{-\langle \nabla \phi_s, \,b_X \rangle}\det(\phi_{s,\,ij}) = se^{-2\phi_0}+(1-s)e^{-\langle\nabla \phi_0,\,b_X \rangle}\det(\phi_{0,\,ij}),\qquad s\in[0,\,1].
\end{equation}
Let $u_{s}=L(\phi_{s})$. Then we have the following uniform integral bound on $|\nabla u_{s}|^{p},\,p\geq1$.

\begin{lemma}\label{Lpnorm}
For all $p\geq1$,
$$\sup_{s\,\in\,[0,\,1]}\int_{P_{-K_{M}}}|\nabla u_{s}|^{p}e^{-\langle b_{X},\,x\rangle}\,dx\leq C.$$
\end{lemma}

\begin{proof}
First note that
$$\int_{\mathbb{R}^{n}}|\xi|^{p}e^{-\langle \nabla\phi_{0},\,b_{X}\rangle}\,\det(\phi_{0,\,ij})d\xi
=\int_{\mathbb{R}^{n}}|\xi|^{p}e^{-2\phi_{0}(\xi)-F}\,d\xi\leq C\int_{\mathbb{R}^{n}}|\xi|^{p}e^{-C|\xi|}\,d\xi<\infty,$$
where we have used \eqref{bbq}, together with Lemma \ref{growth}. Then using Lemma \ref{growth} once again and \eqref{caddo}, we find that
\begin{equation*}
\begin{split}
\int_{P_{-K_{M}}}|\nabla u_{s}|^{p}e^{-\langle b_{X},\,x\rangle}\,dx&=\int_{\mathbb{R}^{n}}|\xi|^{p}e^{-\langle b_{X},\,\nabla\phi_{s}\rangle}\det(\phi_{s,ij})\,d\xi \\
 &= s\int_{\mathbb{R}^{n}}|\xi|^{p}e^{-2\phi_{0}}\,d\xi   + (1-s) \int_{\mathbb{R}^{n}}|\xi|^{p}e^{-\langle b_{X},\,\nabla\phi_{0}\rangle}\det(\phi_{0,ij})\,d\xi  \\
 &\leq C,
\end{split}
\end{equation*}
as desired.
\end{proof}

\subsubsection*{An upper bound on the $\hat{F}$-functional}
Now, our background metric $\omega$ satisfies the two bullet points above Lemma \ref{expression}, as follows from the already proven Theorem \ref{mainthm}(i), (iv), and (v).
As a consequence, it is clear from Lemma \ref{expression}(i) that Definition \ref{fhat}(a) holds true.
The hypothesis of Lemma \ref{converge2} as well as condition (b) of Definition \ref{fhat} via Lemma \ref{converge1} also
hold true thanks to Lemma \ref{lemma-preserved-int}(ii). Thus, the $\hat{F}$-functional from Definition \ref{fhat} is finite and is therefore well-defined along the continuity path \eqref{star-s}
and moreover, by Lemma \ref{converge2}, may be expressed along in terms of the $J_{\omega,\,X}$-functional as $$\hat{F}(\psi_{s})=J_{\omega,\,X}(\psi_{s})-\int_{M}\psi_{s}\,e^{-f}\omega^{n}.$$

We next show that $\hat{F}$ is bounded above along the continuity path \eqref{star-s} using Lemma \ref{I-bounded}. This will in turn provide
an a priori estimate on the weighted integral of the Legendre transform $u_{s}:=L(\phi_{s})$ of $\phi_{s}$. From this,
we derive an a priori estimate on the weighted $L^{1}$-norm of $u_{s}$. Via the Sobolev inequality, we then obtain local
control on $u_{s}$, and as a result, on $\psi_{s}$. This eventually leads to the desired uniform lower bound on $\inf_{M}\vartheta_{s}$.

\begin{lemma}\label{boundedd}
$\hat{F}(\psi_{s})\leq C(\sup_M(\vartheta_{s})_+).$
\end{lemma}

\begin{proof}
This follows from Lemma \ref{I-bounded} as in the proof of \cite[Lemma 7.14]{ccd2}, using the fact that $(I_{\omega,\,X}-J_{\omega,\,X})(\psi_{s})\geq0$ which one derives in
the proof of Proposition \ref{prop-a-priori-ene-est}; cf.~\cite[equation (7.10)]{ccd2} for the precise equation.
\end{proof}

\subsubsection*{An upper bound on the weighted integral of the Legendre transform}
We know that
\begin{equation*}
\begin{split}
\int_{P_{-K_{M}}}|u_{s}|\,e^{-\langle b_{X},\,x\rangle}dx&\leq
\int_{P_{-K_{M}}}|u_{s}-u_{0}|\,e^{-\langle b_{X},\,x\rangle}dx+\int_{P_{-K_{M}}}|u_{0}|\,e^{-\langle b_{X},\,x\rangle}dx\\
&\leq\int_{P_{-K_{M}}}\left(\int_{0}^{1}|\dot{u}_{st}|\,dt\right)\,e^{-\langle b_{X},\,x\rangle}dx+\int_{P_{-K_{M}}}|u_{0}|\,e^{-\langle b_{X},\,x\rangle}dx,
\end{split}
\end{equation*}
and these last two integrals are finite by Lemma \ref{lemma-preserved-int}(ii) via Lemma \ref{converge1}, and
Lemma \ref{expression}(ii), respectively.
By definition, the $\hat{F}$-functional along \eqref{star-s} is given by
\begin{equation}\label{fhat2}
\hat{F}(\psi_s)=2\int_{P_{-K_{M}}}(u_{s}-u_{0})\,e^{-\langle b_{X},\,x\rangle}dx.
\end{equation}
Therefore with $\int_{P_{-K_{M}}}|u_{0}|\,e^{-\langle b_{X},\,x\rangle}dx$
and $\int_{P_{-K_{M}}}|u_{1}|\,e^{-\langle b_{X},\,x\rangle}dx$ convergent, we can split
the integral in \eqref{fhat2}. Together with the integral bound given in
Lemma \ref{expression}(ii), this leads to the following consequence of Lemma \ref{boundedd}.

\begin{corollary}\label{boundd}
$$\sup_{s\,\in\,[0,\,1]}\int_{P_{-K_{M}}}u_{s}\,e^{-\langle b_{X},\,x\rangle}\,dx\leq C.$$
\end{corollary}

\subsubsection*{An upper bound on the weighted $L^{1}$-norm of the Legendre transform}

We now use Corollary \ref{boundd} to derive a uniform weighted $L^{1}$-norm on $u_{s}$.
Notice that this uses the already obtained uniform upper bound on $\vartheta_{s}$.
\begin{lemma}\label{L1norm}
$$\sup_{s\,\in\,[0,\,1]}\int_{P_{-K_{M}}}|u_s|\,e^{-\langle b_X,x\rangle}\,dx\leq C.$$
\end{lemma}

\begin{proof}
This follows verbatim as in the proof of \cite[Lemma 7.16]{ccd2} using the a priori upper bound on $\vartheta_s$ given by Proposition \ref{prop-bd-abo-uni-psi}, together with
Corollary \ref{boundd}, Lemma \ref{expression}(ii), and the fact that
$$\int_{P_{-K_{M}}}e^{-\langle b_X,\,x\rangle}dx=(2\pi)^{n}\int_{M}e^{-f}\omega^{n}<\infty.$$
\end{proof}

\subsubsection*{Local control on $u_{s}$}

Lemmas \ref{Lpnorm} and \ref{L1norm}, combined with an application of the Sobolev inequality, now give us local control on $u_{s}$.
\begin{prop}\label{u-bound}
There exists $C>0$ such that for all $x\in P_{-K_{M}}$ and $s\in[0,\,1]$, $$|u_{s}(x)-u_{0}(x)|\leq Ce^{\langle b_{X},\,x\rangle}.$$
\end{prop}

\begin{proof}
Write $P$ for $P_{-K_{M}}$. Recall from Lemma \ref{normpolytope-tof} that there exists $C>0$ such that for all $x$ outside a compact subset $K\subset P$,
\begin{equation}\label{steak}
C^{-1}|x|\leq\langle b_{X},\,x \rangle\leq C|x|.
\end{equation}
In particular, $\langle b_{X},\,x \rangle$ is strictly positive and is equivalent to $|x|$ on $P\setminus K$. Consider the set
$\Omega_{R}:=\{x\in P\,|\,R<\langle b_{X},\,x \rangle<R+1\}$. Then there exists $R_{0}\gg0$ such that $\Omega_{R}\neq\emptyset$
and $\Omega_{R}\cap K=\emptyset$ for all $R>R_{0}$. As an intersection of half-planes, $\Omega_{R}$ is clearly a bounded convex subset of $\mathbb{R}^{n}$.
We require the following claim.

\begin{claim}\label{knife}
There exists $c>0$ such that for all $R>R_{0}$, $|\Omega_{R}|>c$.
\end{claim}

\begin{proof}[Proof of Claim \ref{knife}]
Since $P$ is non-compact, there exists a sequence of points $x_{i}\in\operatorname{int}(P)$ with $x_{i}\to\infty$.
By Lemma \ref{note}(i), we know that $0\in\operatorname{int}(P)$. Let $L_{i}$ be the line segment connecting $0$ to $x_{i}$.
By convexity, $L_{i}$ lies in $\operatorname{int}(P)$. By compactness, we can take a limit of these lines to obtain a half-ray $L$
contained in $P$ with end point based at the origin. Being open, there exists $\varepsilon>0$ such that $B_{\varepsilon}(0)\subset\operatorname{int}(P)$.
Let $D_{\varepsilon}$ denote the intersection of this ball with the plane perpendicular to $L$. Translating the centre of $D_{\varepsilon}$ along $L$, we
obtain a cylinder $D_{\varepsilon}\times L$ which, by convexity, is also contained in $P$. By \eqref{steak}, 
$\langle b_{X},\,x \rangle$ defines a proper function on the closure $\overline{D_{\varepsilon}\times L}$, hence the planes
$\langle b_{X},\,x \rangle=\operatorname{const.}$ are transverse to $\overline{D_{\varepsilon}\times L}$. In particular, 
$|\Omega_{R}\cap\overline{D_{\varepsilon}\times L}|>0$ is independent of $R>R_{0}$ and serves as a lower bound $c>0$ for $|\Omega_{R}|$.
\end{proof}

Let $x\in P$ be such that
$\langle b_{X},\,x\rangle>R_{0}$. Then there exists $R>R_{0}$ such that
$x\in\Omega_{R}$. For brevity, set $U_{t}:=u_{t}-u_{0}$. Then since $U_{t}$ is smooth up to $\partial P$ by Lemma
\ref{boundaryy}, we can apply
\cite[Theorem 3.4]{polysob} and the first paragraph of the proof thereof
to show that for any $q>n$, there exists $\varepsilon(q):=\frac{q-n}{q-1}>0$ such that
$$\norm{U_{t}}_{C^{0}(\Omega_{R})}\leq \frac{C}{|\Omega_{R}|}\left(\int_{\Omega_{R}}|U_{t}|\,dx+\left(\sup_{y\,\in\,\Omega_{R}}
|y|^{\varepsilon+n}\right)\norm{\nabla U_{t}}_{L^{q}(\Omega_{R})}\right).$$
Fix such a $q$. Then Lemmas \ref{Lpnorm} and \ref{L1norm}, Claim \ref{knife}, and the equivalence of $\langle b_{X},\,x\rangle$ and $|x|$
give us that
\begin{equation*}
\begin{split}
|U_{t}(x)|&\leq\norm{U_{t}}_{C^{0}(\Omega_{R})}\\
&\leq\frac{C}{|\Omega_{R}|}\left(\int_{\Omega_{R}}|U_{t}|\,dx+\left(\sup_{y\,\in\,\Omega_{R}}
|y|^{\varepsilon+n}\right)\norm{\nabla U_{t}}_{L^{q}(\Omega_{R})}\right)\\
&\leq C\left(\sup_{y\,\in\,\Omega_{R}} e^{\langle b_{X},\,y\rangle}+
\left(\sup_{y\,\in\,\Omega_{R}}|y|^{\varepsilon+n}\right)\left(\sup_{y\,\in\,\Omega_{R}} e^{\langle b_{X},\,y\rangle}\right)^{\frac{1}{q}}\right)\\
&\leq C\left(e^{R+1}+(R+1)^{\varepsilon+n}e^{\frac{R+1}{q}}\right)\\
&\leq C\left(e^{\langle b_{X},\,x \rangle}+(\langle b_{X},\,x \rangle+1)^{\varepsilon+n}e^{\frac{\langle b_{X},\,x \rangle}{q}}\right).\\
\end{split}
\end{equation*}
Since $\langle b_{X},\,x \rangle>0$ and $\frac{1}{q}<1$, we find that
$$|U_{t}(x)|\leq C(q)e^{\langle b_{X},\,x\rangle}.$$
A modification of this argument also gives that
$|U_{t}(x)|\leq Ce^{\langle b_{X},\,x\rangle}$ for all $x$ in the remaining compact part of $P$, and so altogether we arrive at the bound
$$|U_{t}(x)|\leq Ce^{\langle b_{X},\,x\rangle}$$
for all $x\in P$, as required.
\end{proof}

\subsubsection*{Local control on $\psi_{s}$}

The previous proposition can be reformulated in the following way to give local control on $\psi_{s}$.

\begin{prop}\label{lowerr}
There exists $C>0$ such that for all $x\in M$ and $s\in[0,\,1]$,
$$\psi_{s}(x)\geq-Ce^{f(x)}.$$
\end{prop}

\begin{proof}
This follows from Proposition \ref{u-bound} as in the proof of \cite[Proposition 7.18]{ccd2}.
\end{proof}

\subsubsection*{A priori lower bound on $\inf_M\vartheta_{s}$}
This brings us to the concluding bound of this section. Proposition \ref{lowerr} yields a uniform lower bound on $\min_{K}\psi_{s}$. By Lemma \ref{lemma-loc-crit-pts}, this results in a uniform lower bound on $\inf_{M}\vartheta_{s}$. This is demonstrated as follows.

\begin{prop}[A priori lower bound on $\inf_M\vartheta_{s}$]\label{prop-bd-bel-uni-psi}
Let $(\vartheta_s)_{0\,\leq\, s\,\leq\, 1}$ be a path of solutions in $C^{\infty}_{X}(M)$ to \eqref{starstar-s}. Then there exists a uniform constant $C>0$ such that
\begin{equation*}
\inf_{0\,\leq\, s\,\leq\, 1}\inf_{M}\vartheta_s\geq-C.
\end{equation*}
\end{prop}

\begin{proof}
This follows from Lemma \ref{lemma-loc-crit-pts} and Proposition \ref{lowerr} as in the proof of \cite[Proposition 7.19]{ccd2}.
\end{proof}

\subsection{A priori upper bound on the radial derivative}\label{sec-upp-bd-rad-der}

The fact that the data $G_{s}$ of \eqref{starstar-s} decays at infinity
allows us to localise the supremum of $X\cdot\vartheta_{s}$ using the maximum principle.
Using the $C^{0}$-bound on $\vartheta_{s}$, this then allows us to derive an a priori upper bound on $X\cdot\vartheta_{s}$.

\begin{lemma}[Localising the supremum of the radial derivative]\label{lemma-loc-crit-pts-rad-der-sup}
Let $(\vartheta_s)_{0\,\leq\, s\,\leq\, 1}$  be a path of solutions in $C^{\infty}_{X}(M)$ to \eqref{starstar-s}.
Then there exists a compact subset $W\subset M$ and a (uniform) positive constant $C$ such that for all $s\in[0,\,1]$,
$\sup_M X\cdot\vartheta_{s}\leq\max_{W}X\cdot\vartheta_{s}+C$.
\end{lemma}

\begin{proof}
From \eqref{starstar-s}, we have the following equation outside a fixed compact subset $W\subset M$:
\begin{equation*}
\begin{split}
\frac{X}{2}\cdot\left(G_{s}+\frac{X}{2}\cdot\vartheta_{s}\right)=\frac{X}{2}\cdot\left(\log\left(\frac{\sigma_{s}^{n}}{\omega_{s}^{n}}\right)\right)&=\operatorname{tr}_{\sigma_{s}}\mathcal{L}_{\frac{X}{2}}\sigma_{s}-\operatorname{tr}_{\omega_{s}}\mathcal{L}_{\frac{X}{2}}\omega_{s}\\
&=\operatorname{tr}_{\sigma_{s}}\mathcal{L}_{\frac{X}{2}}(\omega_{s}+i\partial\bar{\partial}\vartheta_{s})-\operatorname{tr}_{\omega_{s}}\mathcal{L}_{\frac{X}{2}}\omega_{s}\\
&=\operatorname{tr}_{\sigma_{s}}\omega_{0}+\frac{1}{2}\Delta_{\sigma_{s}}(X\cdot\vartheta_{s})-\operatorname{tr}_{\omega_{s}}\omega_{0}.\\
\end{split}
\end{equation*}
Thus, upon differentiating \eqref{starstar-s} along $X$, we obtain on $M\setminus W$ the equation
\begin{equation}\label{banana1}
\begin{split}
\Delta_{\sigma_{s},\,X}\left(\frac{X}{2}\cdot\vartheta_{s}\right):=
\Delta_{\sigma_{s}}\left(\frac{X}{2}\cdot\vartheta_{s}\right)-\frac{X}{2}\cdot\left(\frac{X}{2}\cdot\vartheta_{s}\right)
&=\operatorname{tr}_{\omega_{s}}\omega_{0}-\operatorname{tr}_{\sigma_{s}}\omega_{0}+\frac{X}{2}\cdot G_{s}.
\end{split}
\end{equation}
Now, recall that outside a fixed compact subset, we have that
\begin{equation*}
\begin{split}
\omega_{s}&=\omega+i\partial\bar{\partial}\Phi_{s}=\omega-2c_{s} i\partial\bar{\partial}\log r=\omega_{0}+\rho_{\omega_{0}}-2c_{s}\omega^{T},
\end{split}
\end{equation*}
so that there exists a positive constant $C$ such that for all $s\in[0,\,1]$, $C^{-1}\omega_0\leq \omega_s\leq C\omega_0$ on $M$. In particular, since the norm of $\rho_{\omega_0}+i\partial\bar{\partial}\Phi_s$ with respect to the cone metric $\omega_0$ is bounded outside a fixed compact subset contained the exceptional set of the resolution, one sees that there is a positive constant $C$ (that may differ from line to another) such that for all $s\in[0,1]$,
\begin{equation}\label{hello-savior}
|\tr_{\sigma_s}\left(\rho_{\omega_0}+i\partial\bar{\partial}\Phi_s\right)|\leq C\tr_{\sigma_s}\omega_s.
\end{equation}
Combining \eqref{banana1} and \eqref{hello-savior}, one now sees that outside a sufficiently large compact subset of $M$,
\begin{equation*}
\begin{split}
\Delta_{\sigma_{s},\,X}\left(\frac{X}{2}\cdot\vartheta_{s}-\vartheta_s\right)
&=\operatorname{tr}_{\omega_{s}}\omega_{0}-\operatorname{tr}_{\sigma_{s}}\omega_{0}+\frac{X}{2}\cdot G_{s}-\Delta_{\sigma_{s}}\vartheta_{s}+\frac{X}{2}\cdot\vartheta_{s}\\
&=\frac{X}{2}\cdot\vartheta_{s}+\operatorname{tr}_{\omega_{s}}\omega_{0}-\operatorname{tr}_{\sigma_{s}}\bigl(
\underbrace{\omega_{0}+\rho_{\omega_0}+i\partial\bar{\partial}\left(\Phi_s+\vartheta_s\right)}_{=\,\sigma_{s}}\bigr)\\
&\quad+\operatorname{tr}_{\sigma_{s}}\left(\rho_{\omega_0}+i\partial\bar{\partial}\Phi_s\right)+\frac{X}{2}\cdot G_{s}\\
&\geq \frac{X}{2}\cdot\vartheta_{s}+\operatorname{tr}_{\omega_{s}}\omega_{0}-n+\frac{X}{2}\cdot G_{s}-C\tr_{\sigma_s}\omega_s,
\end{split}
\end{equation*}
so that by using the relation $\tr_{\sigma_s}\omega_s=n-\Delta_{\sigma_s}\vartheta_s$ again, we arrive at the bound
\begin{equation}
\begin{split}\label{getting-better}
\Delta_{\sigma_{s},\,X}\left(\frac{X}{2}\cdot\vartheta_{s}-(C+1)\vartheta_s\right)
&\geq (C+1)\frac{X}{2}\cdot\vartheta_{s}+\operatorname{tr}_{\omega_{s}}\omega_{0}-(C+1)n+\frac{X}{2}\cdot G_{s}
\end{split}
\end{equation}
outside a sufficiently large compact subset of $M$. Observe from Lemma \ref{lemma-tr-star-star} and the previous estimate \eqref{getting-better}
that for $\varepsilon>0$, 
\begin{equation}
\begin{split}\label{olala}
\Delta_{\sigma_{s},\,X}\left(\frac{X}{2}\cdot\vartheta_{s}-(C+1)\vartheta_s-\varepsilon f_{\sigma_s}\right)
&\geq (C+1)\frac{X}{2}\cdot\vartheta_{s}+\operatorname{tr}_{\omega_{s}}\omega_{0}-(C+1)n+\frac{X}{2}\cdot G_{s}+\varepsilon(f-P_s)
\\
&\geq (C+1)\frac{X}{2}\cdot\vartheta_{s}+\operatorname{tr}_{\omega_{s}}\omega_{0}-(C+1)n+\frac{X}{2}\cdot G_{s}\\
&\geq  (C+1)\frac{X}{2}\cdot\vartheta_{s}-C_{1}
\end{split}
\end{equation}
outside a large compact subset $W$ of $M$ (independent of $\varepsilon$) chosen such that $f\geq\sup_{s\in[0,\,1]}P_s$ on $M\setminus W$. 
Here, $C_{1}$ denotes a positive constant that is independent of $s\in[0,1]$.

Next, for $\varepsilon>0$ consider the function $\frac{X}{2}\cdot\vartheta_{s}-(C+1)\vartheta_s-\varepsilon f_{\sigma_s}$. By definition of the relevant function spaces, for each $s\in[0,\,1]$, the function $\frac{X}{2}\cdot\vartheta_{s}-(C+1)\vartheta_s-\varepsilon f_{\sigma_s}$ is proper and bounded from above. In particular, it attains a global maximum on $M$. If such a maximum lies in $M\setminus W$, then \eqref{olala} implies that $(C+1)X\cdot \vartheta_s\leq 2C_{1}$. In particular, equation \eqref{vladimir} and Proposition \ref{prop-bd-bel-uni-psi} allow us to estimate pointwise on $M$ that 
\begin{equation*}
\frac{X}{2}\cdot\vartheta_{s}-(C+1)\vartheta_s-\varepsilon f_{\sigma_s}\leq \max\left\{\max_{W}\left(\frac{X}{2}\cdot\vartheta_{s}-(C+1)\vartheta_s-\varepsilon f_{\sigma_s}\right),\,C_{2}\right\}
\end{equation*}
for another uniform positive constant $C_{2}$. Since this estimate holds for all $\varepsilon>0$, we may let $\varepsilon\to0$ to obtain the upper bound 
\begin{equation*}
\begin{split}
\frac{X}{2}\cdot\vartheta_{s}&\leq \max\left\{\max_{W}\left(\frac{X}{2}\cdot\vartheta_{s}-(C+1)\vartheta_s\right),\,C_{2}\right\}+(C+1)\vartheta_s\\
&\leq \max_{W}\left(\frac{X}{2}\cdot\vartheta_{s}\right)+C_{3},
\end{split}
\end{equation*}
for yet another uniform positive constant $C_{3}$, thanks to Propositions \ref{prop-bd-abo-uni-psi} and \ref{prop-bd-bel-uni-psi}.
\end{proof}

We now conclude with:
\begin{prop}\label{prop-bd-uni-X-psi}
Let $(\vartheta_s)_{0\,\leq\, s\,\leq\, 1}$ be a path of solutions in $C^{\infty}_{X}(M)$ to \eqref{starstar-s}. Then there exists a positive constant $C=C\left(n,\omega,\sup_{s\in[0,1]}\|G_s\|_{C^{0}(M)}\right)$ such that
\begin{equation*}
\sup_{0\,\leq\, s\,\leq\, 1}\sup_M X\cdot\vartheta_s\leq C.
\end{equation*}
In particular, $X\cdot\vartheta_{s}<C$ for all $s\in[0,\,1]$.
\end{prop}

\begin{proof}
The proof is verbatim the proof of \cite[Proposition 7.20]{ccd2}, using the localisation given by Lemma \ref{lemma-loc-crit-pts-rad-der-sup}
and the uniform upper bound on $\|\vartheta_{s}\|_{C^{0}(M)}$.
\end{proof}

\subsection{A priori estimates on higher derivatives}\label{sec-high-der}
We next derive a priori global bounds on higher derivatives of solutions to the complex Monge-Amp\`ere equation \eqref{starstar-s}, beginning with
the $C^{2}$-estimate. The a priori bounds we derive hold everywhere on the manifold $M$, not just on a given fixed compact subset.
We get around this problem by working with the drift Laplacian. More precisely, by
following Yau's original $C^{2}$-estimate, we show that the Laplacian of $\vartheta_{s}$
is a subsolution of a certain elliptic PDE with respect to the drift Laplacian of the unknown metric.

\subsubsection{$C^2$ a priori estimate}

\begin{prop}[A priori $C^2$-estimate]\label{prop-C^2-est}
Let $(\vartheta_s)_{0\,\leq\, s\,\leq\, 1}$ be a path of solutions in $C^{\infty}_{X}(M)$ to \eqref{starstar-s}. Then there exists a positive constant $C=C\left(n,\omega,\sup_{s\in[0,1]}\|G_s\|_{C^2}\right)$ such that the following $C^2$ a priori estimate holds true:
\begin{equation*}
\sup_{0\,\leq\, s\,\leq\, 1}\|i\partial\bar{\partial}\vartheta_s\|_{C^{0}(M)}\leq C.
\end{equation*}
In particular,
\begin{equation*}
\sup_{0\,\leq\, s\,\leq\, 1}\|i\partial\bar{\partial}\psi_s\|_{C^{0}(M)}\leq C.
\end{equation*}
\end{prop}

\begin{proof}
Following closely the proof of \cite[Proposition 6.6]{con-der} where the
approach taken is based on standard computations performed in Yau's seminal paper \cite[pp.347--351]{Calabiconj} (see also \cite[Lemma 5.4.16]{siepmann}
and \cite[pp.52--55]{Tian-Can-Met-Boo}), let $\Delta_{\sigma_{s}}$ denote the Laplacian with respect to $\sigma_{s}$ and let $u_{s}:=e^{-\lambda\vartheta_{s}}(n+\Delta_{\sigma_{s}}\vartheta_{s})$, where $\lambda>0$ will be specified later. One first estimates the drift Laplacian $\Delta_{\sigma_{s}}-\frac{X}{2}\cdot$ of $u_{s}$ with respect to $\sigma_{s}$ in the following way, using the fact that $\vartheta_{s}$ satisfies \eqref{starstar-s}:
 \begin{equation}\label{C2-est-yau}
 \begin{split}
\left(\Delta_{\sigma_{s}}-\frac{X}{2}\cdot\right) u_{s}
&\geq e^{-\lambda\vartheta_{s}}\Delta_{\sigma_{s}}G_{s}+e^{-\lambda\vartheta_{s}}g_{s}\left(\mathcal{L}_{\frac{X}{2}}\omega_s,\,i\partial\bar\partial\vartheta_{s}\right)
-C_{s}n^{2}e^{-\lambda\vartheta_{s}}+\lambda\left(\frac{X}{2}\cdot\vartheta_{s}\right)u_{s}
-\lambda n u_{s}\\
&\quad+(\lambda+C_{s})e^{\frac{\lambda\vartheta_{s}-G_{s}-\frac{X}{2}\cdot\vartheta_{s}}{n-1}}u_{s}^{\frac{n}{n-1}},
\end{split}
\end{equation}
where $C_{s}:=\inf_{i\,\neq\,k}\operatorname{Rm}^{s}_{i\bar{\imath}k\bar{k}}$, $\operatorname{Rm}^{s}$ here denoting the complex linear extension of the curvature operator of the metric $g_{s}$. Since $C_{s}$ is uniformly bounded below in $s$ by a constant $A$, we may choose $\lambda>0$ so that $\lambda+A=1$. Moreover, as
$$\left|g_{s}\left(\mathcal{L}_{\frac{X}{2}}\omega_s,\,i\partial\bar\partial\vartheta_{s}\right)\right|\leq C\|\nabla^{s}X\|_{C^{0}(M)}(1+u_{s})$$
for some generic constant $C>0$, we deduce that $u_{s}$ satisfies the following differential inequality:
\begin{equation*}
\left(\Delta_{\sigma_{s}}-\frac{X}{2}\cdot\right) u_{s}\geq-C_{1}(1+u_{s})+C_{2}u_{s}^{\frac{n}{n-1}},
\end{equation*}
where $C_{1}$ and $C_{2}$ depend only on $n$, $A$, $\sup_{s\,\in\,[0,\,1]}\|\vartheta_{s}\|_{C^{0}(M)}$, $\sup_{s\,\in\,[0,\,1]}\|X\cdot\vartheta_{s}\|_{C^{0}(M)}$,\linebreak
$\sup_{s\,\in\,[0,\,1]}\|G_{s}\|_{C^{2}}$, and $\sup_{s\,\in\,[0,\,1]}\|\nabla^{s}X\|_{C^{0}(M)}$. The combination of Corollary \ref{lowerbound2} and Propositions \ref{prop-bd-abo-uni-psi}, \ref{prop-bd-bel-uni-psi}, and \ref{prop-bd-uni-X-psi} shows that $C_{1}$ and $C_{2}$ actually depend only on $n$, $A$, and $\sup_{s\,\in\,[0,\,1]}\|G_{s}\|_{C^{2}}$.

Since $u_{s}$ is non-negative and converges to $n$ at infinity because $\vartheta_{s}\in C^{\infty}_{X}(M)$,
an application of the maximum principle to an exhausting sequence of domains of $M$ yields an upper bound for $n+\Delta_{\sigma_{s}}\vartheta_{s}$
and in turn, the desired bound on $i\partial\bar{\partial}\vartheta_{s}$.
\end{proof}

A useful consequence of Proposition \ref{prop-C^2-est} is that the K\"ahler metrics
induced by $\sigma_{s}$ and $\omega_{s}$ are uniformly equivalent.
\begin{corollary}\label{coro-equiv-metrics-0}
Let $(\vartheta_s)_{0\,\leq\, s\,\leq\, 1}$ be a path of solutions in $C^{\infty}_{X}(M)$ to \eqref{starstar-s} and
for $s\in[0,\,1]$, let $g_{s},\,h_{s}$ denote the K\"ahler metrics induced by $\omega_{s},\,\sigma_{s}$ respectively. Then the tensors
$g_{s}^{-1}h_{s}$ and $h_{s}^{-1}g_{s}$ satisfy the following uniform estimate:
\begin{equation*}
\sup_{0\,\leq\, t\,\leq\, 1}\|g_{s}^{-1}h_{s}\|_{C^{0}(M)}+\sup_{0\,\leq\, t\,\leq\, 1}\|h_{s}^{-1}g_{s}\|_{C^{0}(M)}\leq C
\end{equation*}
for some positive constant $C=C\left(n,\omega,\sup_{s\in[0,1]}\|G_s\|_{C^2}\right)$.
In particular, the metrics $g_{s}$ and $(h_{s})_{0\,\leq\, s\,\leq\, 1}$ are uniformly equivalent.
\end{corollary}

\begin{proof}
The estimate follows as in the proof of \cite[Corollary 7.15]{conlon33} using Corollary \ref{lowerbound2} and Propositions \ref{prop-bd-uni-X-psi} and \ref{prop-C^2-est}. The fact that $\omega_{s}$ and $\sigma_{s}$ differ by a $(1,\,1)$-form whose norm is controlled uniformly in $s$ yields the last claim of the corollary.
\end{proof}

\subsubsection{$C^3$ a priori estimate}
We now present the $C^{3}$-estimate.
\begin{prop}[A priori $C^3$-estimate]\label{prop-C^3-est}
Let $(\vartheta_s)_{0\,\leq\, s\,\leq\, 1}$ be a path of solutions in $C^{\infty}_{X}(M)$ to \eqref{starstar-s} and let
$g_{s}$ be the K\"ahler metric induced by $\omega_{s}$ with Levi-Civita connection $\nabla^{g_{s}}$. Then
\begin{equation*}
\sup_{0\,\leq\, s\,\leq\, 1}\|\nabla^{g_{s}}\partial\bar{\partial}\vartheta_s\|_{C^{0}(M)}\leq C\left(n,\omega,\sup_{s\in[0,1]}\| G_s\|_{C^3}\right).
\end{equation*}
In particular,
\begin{equation}\label{a-priori-nabla-rad-der}
\sup_{0\,\leq\, s\,\leq\, 1}\|\nabla^{g_{s}}\left(X\cdot\vartheta_s\right)\|_{C^{0}(M)}\leq C\left(n,\omega,\sup_{s\in[0,1]}\| G_s\|_{C^3}\right).
\end{equation}
\end{prop}

Before proving Proposition \ref{prop-C^3-est}, we need a lemma dictating the evolution equation of the norm squared of the covariant derivative of the difference $h_s-g_s$.
To this end, set
\begin{equation*}\label{def-S}
S:=S(h_s,\,g_s):=\arrowvert\nabla^{g_s}h_s\arrowvert^2_{h_s}.
\end{equation*}
Then from the definition of $S$, we see that
\begin{equation*}
\begin{split}
S=&h_s^{i\bar{\jmath}}h_s^{k\bar{\ell}}h_s^{p\bar{q}}\nabla^{g_s}_i(h_s)_{k\bar{q}}\overline{\nabla^{g_s}_{j}(h_s)_{l\bar{p}}}=|\Psi|_{h_s}^2,
\end{split}
\end{equation*}
where
\begin{equation}
\begin{split}\label{def-Psi}
\Psi_{ij}^k=\Psi_{ij}^k(h_s,\,g_s)&:=\Gamma(h_s)_{ij}^k-\Gamma(g_s)_{ij}^k=h_s^{k\bar{\ell}}\nabla^{g_s}_i(h_s)_{j\bar{\ell}}.
\end{split}
\end{equation}
Define the ``half'' Laplacian with respect to the K\"ahler metric $\sigma_s$ in the following way:
\begin{equation*}
\Delta_{\sigma_{s},\,1/2}:=h_s^{i\bar{\jmath}}\nabla^{h_s}_i\nabla^{h_s}_{\bar{\jmath}}.\label{def-lap-half}
\end{equation*}
Then we have:
\begin{lemma}\label{lemma-S}
Recall $Q_s:=F+G_{s}-F_s-\Phi_{s}$ from the proof of Lemma \ref{normal-fss} so that
\begin{equation*}
\begin{split}
\nabla^{g_s}\overline{\nabla}^{g_s}Q_s&=\Ric(g_s)+\mathcal{L}_{\frac{X}{2}}g_s-g_s,\\
Q_s&=2c_{s}\log(r)-c_{s}+c_{0}+O(r^{-2})\quad\text{with $g_0$-derivatives},\\
\Lambda_s&:=\Ric(g_s)+\mathcal{L}_{\frac{X}{2}}g_s-g_s-\nabla^{g_s}\overline{\nabla}^{g_s}G_s=\nabla^{g_s}\overline{\nabla}^{g_s}(Q_s-G_s).
\end{split}
\end{equation*}
Then tensor $\Psi$ schematically satisfies
\begin{equation*}
\begin{split}
\left(\Delta_{\sigma_{s},\,1/2}-\frac{1}{2}\nabla^{h_s}_X\right)\Psi_{ip}^k&=\frac{3}{2}\Psi_{ip}^k+T_{ip}^k,\\
T&:=\Psi\ast \left(\Ric(h_s)+(h_s-g_s)+\Rm(g_s)\right)+\nabla^{g_s}\left(\Rm(g_s)+\Lambda_s+\nabla^{g_s}\overline{\nabla}^{g_s}Q_s\right).
\end{split}
\end{equation*}
Here the dependence of contractions with respect to $h_s$ is omitted. Moreover, the function $S$ satisfies
\begin{equation*}
\begin{split}
\left(\Delta_{\sigma_{s}}-\frac{X}{2}\cdot\right)S&\geq |\nabla^{h_s} \Psi|^2_{h_s}+|\overline{\nabla}^{h_s}\Psi|_{h_s}^2+3S\\
&\quad-C(n,\tr_{g_s}h_s,\tr_{h_s}g_s))\left(|\Rm(g_s)|_{g_s}+|\Lambda_s|_{g_s}+|h_s-g_s|_{g_s}\right)S\\
&\quad-C(n,\tr_{g_s}h_s,\tr_{h_s}g_s))\left(|\nabla^{g_s}\Rm(g_s)|_{g_s}+|\nabla^{g_s}\Lambda_s|_{g_s}+|\nabla^{g_s}\nabla^{g_s}\overline{\nabla}^{g_s}Q_s|_{g_s}\right)\sqrt{S}.
\end{split}
\end{equation*}
Here, $C(\cdot,\,\cdot\,,\,\cdot)$ denotes a function that is increasing in each of its arguments
and $\overline{\nabla}^g=\nabla^{0,\,1}$. In particular, $\sqrt{S}$ weakly satisfies the inequality
\begin{equation}
\begin{split}\label{weak-subsol-sqrt-S}
\left(\Delta_{\sigma_{s}}-\frac{X}{2}\cdot\right)\sqrt{S}&\geq \frac{3}{2}\sqrt{S}-C(n,\tr_{g_s}h_s,\tr_{h_s}g_s))\left(|\Rm(g_s)|_{g_s}+|\Lambda_s|_{g_s}+|h_s-g_s|_{g_s}\right)\sqrt{S}\\
&\quad-C(n,\tr_{g_s}h_s,\tr_{h_s}g_s))\left(|\nabla^{g_s}\Rm(g_s)|_{g_s}+|\nabla^{g_s}\Lambda_s|_{g_s}+|\nabla^{g_s}\nabla^{g_s}\overline{\nabla}^{g_s}Q_s|_{g_s}\right).
\end{split}
\end{equation}
\end{lemma}

\begin{remark}
Notice the coefficient $\frac{3}{2}$ appearing in front of the leading term $\sqrt{S}$ on the right-hand side of \eqref{weak-subsol-sqrt-S}.
This reflects the expected optimal decay rate $f^{-\frac{3}{2}}$ to be established in Corollary \ref{coro-bds-Ck-weight}.
\end{remark}

\begin{proof}[Proof of Lemma \ref{lemma-S}]
The first assertions concerning $Q_s$ and $\Lambda_s$ are obvious the proof of Lemma \ref{normal-fss}. For the remainder of the proof,
we follow closely the proof of \cite[Proposition 6.9]{con-der}, which itself is based on \cite{Pho-Ses-Stu}.

Since $\vartheta_s$ solves \eqref{starstar-s}, $(M,\,h_s,\,X)$
is an ``approximate'' shrinking gradient K\"ahler-Ricci soliton in the following precise sense:
if $h_s(\tau):=(-\tau)(\varphi^{X}_{\tau})^*h_s$ and $g_s(\tau):=(-\tau)(\varphi^{X}_{\tau})^*g$,
where $(\varphi^{X}_{\tau})_{\tau\,<\,0}$ is the
one-parameter family of diffeomorphisms generated by $\frac{X}{2(-\tau)}$ such that $\varphi^{X}_{\tau}\big|_{\tau\,=\,-1}=\Id_M$,
then $(h_s(\tau))_{\tau\,<\,0}$
is a solution of the following perturbed K\"ahler-Ricci flow with ``initial condition'' $h_s$:
\begin{equation}
\begin{split}\label{hello-ugly-eqn}
\partial_{\tau}h_s(\tau)\big|_{\tau\,=\,-1}&=-h_s+\mathcal{L}_{\frac{X}{2}}h_s\\
&=-\Ric(h_s)+\underbrace{\mathcal{L}_{\frac{X}{2}}g_s+\Ric(g_s)-g_s-\nabla^{g_s}\overline{\nabla}^{g_s}G_s}_{=:\,\Lambda_s}+(g_s-h_s),\qquad h_s(\tau)\big|_{\tau\,=\,-1}=h_s.
 \end{split}
\end{equation}
Similarly, by definition of the datum $Q_s$, we have that
\begin{equation*}
\begin{split}
\partial_{\tau}g_s(\tau)\big|_{\tau\,=\,-1}&=-g_s+\mathcal{L}_{\frac{X}{2}}g_s\\
&=-\Ric(g_s)+\nabla^{g_s}\overline{\nabla}^{g_s}Q_s,\qquad g_s(\tau)\big|_{\tau\,=\,-1}=g_s.
 \end{split}
\end{equation*}

Define $S(\tau):=S(h_s(\tau),\,g_s(\tau))$, and correspondingly set $\Psi(\tau):=\Psi(h_s(\tau),\,g_s(\tau))$. We adapt the proof of \cite[Proposition 3.2.8]{Bou-Eys-Gue} to our setting. By a
brute force computation, we have that
\begin{equation}
\begin{split}\label{brute-force-S}
\Delta_{\sigma_{s}}S&=2\Re\left(h_s^{i\bar{\jmath}}h_s^{p\bar{q}}(h_s)_{k\bar{\ell}}\left(\Delta_{\sigma_{s},\,1/2}\Psi_{ip}^k\right)
\overline{\Psi_{jq}^\ell}\right)+|\nabla^{h_s} \Psi|^2_{h_s}+|\overline{\nabla}^{h_s}\Psi|_{h_s}^2\\
&\quad+\Ric(h_s)^{i\bar{\jmath}}h_s^{p\bar{q}}(h_s)_{k\bar{\ell}}\Psi_{ip}^k\overline{\Psi_{jq}^\ell}
+h_s^{i\bar{\jmath}}\Ric(h_s)^{p\bar{q}}(h_s)_{k\bar{\ell}}\Psi_{ip}^k\overline{\Psi_{jq}^\ell}-h_s^{i\bar{\jmath}}h_s^{p\bar{q}}\Ric(h_s)_{k\bar{\ell}}\Psi_{ip}^k\overline{\Psi_{jq}^\ell},
\end{split}
\end{equation}
where
\begin{equation*}
\begin{split}
&T^{i\bar{\jmath}}:=h_s^{i\bar{k}}h_s^{l\bar{\jmath}}T_{k\bar{\ell}}
\end{split}
\end{equation*}
for $T_{k\bar{\ell}}\in\Lambda^{1,\,0}M\otimes\Lambda^{0,\,1}M$. We also have that
\begin{equation}
\begin{split}\label{time-der-Psi-I}
\partial_{\tau}\Psi(\tau)_{ip}^k|_{\tau\,=\,-1}&=\partial_{\tau}|_{\tau\,=\,-1}(\Gamma(h_s(\tau))-\Gamma(g_s(\tau)))_{ip}^k\\
&=\nabla^{h_s}_i(-\Ric(h_s)_p^k+(\Lambda_s)_p^k+(g_s-h_s)_p^k)-\nabla^{g_s}_i(-\Ric(g_s)_p^k+\nabla^{g_s}\overline{\nabla}^{g_s}(Q_s)_p^k)\\
&=-\nabla^{h_s}_i\Ric(h_s)_p^k+\nabla^{h_s}_i (g_s)_p^k+\nabla^{h_s}_i(\Lambda_s)_p^k+\nabla^{g_s}_i(\Ric(g_s)_p^k-\nabla^{g_s}\overline{\nabla}^{g_s}(Q_s)_p^k)\\
&=-\nabla^{h_s}_i\Ric(h_s)_p^k-\Psi_{ip}^k+\left(\nabla^{h_s}_i (g_s)_p^k+\nabla^{g_s}_i (h_s)_p^k\right)+\nabla^{h_s}_i(\Lambda_s)_p^k\\
&\quad\qquad+\nabla^{g_s}_i(\Ric(g_s)_p^k-\nabla^{g_s}\overline{\nabla}^{g_s}(Q_s)_p^k).
\end{split}
\end{equation}
Here we have used the definition of $\Psi$ from \eqref{def-Psi} in the last line. Since $h_s$ (respectively $g_s$) is parallel with respect to the Levi-Civita connection $\nabla^{h_s}$ (resp.~$\nabla^{g_s}$), observe that
\begin{equation*}
h_s^{k\bar{\ell}}\nabla^{h_s}_i (g_s)_{p\bar{\ell}}+h_s^{k\bar{\ell}}\nabla^{g_s}_i (h_s)_{p\bar{\ell}}=h_s^{k\bar{\ell}}\left(\nabla^{g_s}_i-\nabla^{h_s}_i\right)(h_s-g_s)_{p\bar{\ell}},
\end{equation*}
so that schematically,
\begin{equation*}\label{scheme}
\begin{split}
|\nabla^{h_s}g_s+\nabla^{g_s} (h_s)|_{h_s}&\leq C(n,\tr_{g_s}h_s,\tr_{h_s}g_s)|h_s-g_s|_{g_s}| \Psi|_{h_s}\\
&\leq C(n,\tr_{g_s}h_s,\tr_{h_s}g_s)|h_s-g_s|_{g_s}\sqrt{S}.
\end{split}
\end{equation*}
Moreover, by the very definition of $\Psi$ and its scaling properties, we see that
\begin{equation}
\begin{split}\label{time-der-Psi-II}
\partial_{\tau}\Psi(\tau)_{ip}^k|_{\tau\,=\,-1}&=\partial_{\tau}(\varphi^{X}_{\tau})^*\Psi_{ip}^k|_{\tau\,=\,-1}=\frac{1}{2}\mathcal{L}_{X}\Psi_{ip}^k\\
&=\frac{1}{2}\nabla^{h_s}_X\Psi_{ip}^k+\frac{1}{2}\Psi_{ip}^k+\Psi\ast \left(\frac{1}{2}\mathcal{L}_{X}h_s-h_s\right)_{ip}^k\\
&=\frac{1}{2}\nabla^{h_s}_X\Psi_{ip}^k+\frac{1}{2}\Psi_{ip}^k+\Psi\ast \left(-\Ric(h_s)-(h_s-g_s)+\Lambda_s\right)_{ip}^k,
\end{split}
\end{equation}
where \eqref{hello-ugly-eqn} has been used in the last line.

Finally, using the second Bianchi identity (the traced version to be precise), we compute that
\begin{equation}\label{ronan}
\Delta_{\sigma_{s},\,1/2}\Psi_{ip}^k=h_s^{a\bar{b}}\nabla_a^{h_s}\Rm(g_s)_{i\bar{b}p}^k-\nabla^{h_s}_i\Ric(h_s)_p^k.
\end{equation}
This in turn implies that $\Psi$ satisfies the following schematic evolution equation obtained by combining equations \eqref{time-der-Psi-I}--\eqref{ronan}:
\begin{equation*}
\begin{split}
\left(\Delta_{\sigma_{s},\,1/2}-\frac{1}{2}\nabla^{h_s}_X\right)\Psi_{ip}^k&=\frac{3}{2}\Psi_{ip}^k+T_{ip}^k,\\
T&:=\Psi\ast \left(\Ric(h_s)+(h_s-g_s)+\Lambda_s+\Rm(g_s)\right)\\
&\quad+\nabla^{g_s}\left(\Rm(g_s)+\Lambda_s+\nabla^{g_s}\overline{\nabla}^{g_s}Q_s\right).
\end{split}
\end{equation*}
Here we have used that for a tensor $R$, $\nabla^{h_s}R=\nabla^{g_s}R+R\ast \Psi$ in the last line, all contractions being with respect to $h_s$.

Notice the simplification here regarding the ``bad'' term $-\nabla^{h_s}\Ric(h_s)$. Since this flow is evolving only by pullback by diffeomorphisms and scalings, we know that
\begin{equation}
\begin{split}\label{time-der-S-I}
S(\tau)&=(-\tau)^{-1}(\varphi^{X}_{\tau})^*S(h_s,\,g_s),\\
\partial_{\tau}S|_{\tau\,=\,-1}&=S+\frac{X}{2}\cdot S.
\end{split}
\end{equation}
By definition of $S$ in terms of the tensor $\Psi$, we furthermore know that
\begin{equation}
\begin{split}\label{time-der-S-II}
\partial_{\tau}S|_{\tau\,=\,-1}&=-\left(\frac{1}{2}\mathcal{L}_{X}h_s-h_s\right)^{i\bar{\jmath}}h_s^{p\bar{q}}(h_s)_{k\bar{\ell}}\Psi_{ip}^k\overline{\Psi_{jq}^\ell}
-h_s^{i\bar{\jmath}}\left(\frac{1}{2}\mathcal{L}_{X}h_s-h_s\right)^{p\bar{q}}(h_s)_{k\bar{\ell}}\Psi_{ip}^k\overline{\Psi_{jq}^\ell}\\
&\quad+h_s^{i\bar{\jmath}}h_s^{p\bar{q}}\left(\frac{1}{2}\mathcal{L}_{X}h_s-h_s\right)_{k\bar{\ell}}\Psi_{ip}^k\overline{\Psi_{jq}^\ell}+h_s^{i\bar{\jmath}}h_s^{p\bar{q}}(h_s)_{k\bar{\ell}}\partial_{\tau}\Psi(\tau)_{ip}^k|_{\tau\,=\,-1}\overline{\Psi(\tau)_{jq}^\ell}\\
&\quad+h_s^{i\bar{\jmath}}h_s^{p\bar{q}}(h_s)_{k\bar{\ell}}\Psi(\tau)_{ip}^k\overline{\partial_{\tau}\Psi(\tau)_{jq}^\ell}|_{\tau\,=\,-1}.
\end{split}
\end{equation}
By combining \eqref{hello-ugly-eqn}, \eqref{brute-force-S}, \eqref{ronan}--\eqref{time-der-S-II}, we arrive at the second estimate of the lemma, namely that
\begin{equation}\label{dream-come-true}
\begin{split}
\Delta_{\sigma_{s}}S-\frac{X}{2}\cdot S&\geq |\nabla^{h_s} \Psi|^2_{h_s}+|\overline{\nabla}^{h_s}\Psi|_{h_s}^2+3S\\
&\quad-C(n,\tr_{g_s}h_s,\tr_{h_s}g_s))\left(|\Rm(g_s)|_{g_s}+|h_s-g_s|_{g_s}+|\Lambda_s|_{g_s}\right)S\\
&\quad-C(n,\tr_{g_s}h_s,\tr_{h_s}g_s))\left(|\nabla^{g_s}\Rm(g_s)|_{g_s}+|\nabla^{g_s}\Lambda_s|_{g_s}+|\nabla^{g_s}\nabla^{g_s}\overline{\nabla}^{g_s}Q_s|_{g_s}\right)\sqrt{S}\\
\end{split}
\end{equation}
for some positive uniform constant $C$. Notice that, thanks to \eqref{hello-ugly-eqn}, the last three terms of the right-hand side of \eqref{brute-force-S} cancel with the first three terms of the right-hand side of \eqref{time-der-S-II}, up to terms of the form $\left(\Lambda_s+(g_s-h_s)\right)\ast \Psi\ast\Psi$, where again all contractions are with respect to $h_{s}$.

The third estimate is a straightforward application of Kato's inequality to \eqref{dream-come-true}.
\end{proof}

We are now in a position to prove Proposition \ref{prop-C^3-est}.
\begin{proof}[Proof of Proposition \ref{prop-C^3-est}]
Lemma \ref{lemma-S} together with
the boundedness of $\|h_{s}^{-1}g_{s}\|_{C^{0}(M)}$ and\linebreak $\|h_{s}g_{s}^{-1}\|_{C^{0}(M)}$ given by Corollary \ref{coro-equiv-metrics-0} give us that
\begin{equation*}
\left(\Delta_{\sigma_{s}}-\frac{X}{2}\cdot\right) S\geq-C(S+1)
\end{equation*}
for some uniform positive constant $C$. Here we have used the boundedness of (the covariant derivatives of) the tensors $\Rm(g_{s})$, $\Lambda_s$, and $\nabla^{g_s}\overline{\nabla}^{g_s}Q_s$.
We use as a barrier function $\tr_{\omega_{s}}\sigma_{s}$ which, by an adaptation of the proof of \eqref{C2-est-yau} and the uniform equivalence of the metrics $g_{s}$ and $h_{s}$ given by Corollary \ref{coro-equiv-metrics-0}, satisfies
\begin{equation*}
\left(\Delta_{\sigma_{s}}-\frac{X}{2}\cdot\right)\tr_{\omega_{s}}\sigma_{s}\geq C^{-1}S-C,
\end{equation*}
where $C$ is a uniform positive constant that may vary from line to line. By applying the maximum principle to $\varepsilon S+\tr_{\omega_{s}}\sigma_{s}$ for some sufficiently small $\varepsilon>0$, one arrives at the desired a priori estimate. The proof of \eqref{a-priori-nabla-rad-der} follows from the previously proved a priori bound on $\nabla^{g_s}\partial\bar{\partial}\vartheta_s$ after differentiating \eqref{starstar-s}.
\end{proof}

We next establish H\"older regularity of $g_{s}^{-1}h_{s}$ and $h^{-1}_{s}g_{s}$, which is an improvement on Corollary \ref{coro-equiv-metrics-0}.
\begin{corollary}\label{coro-equiv-metrics}
Let $(\vartheta_{s})_{0\,\leq\, s\,\leq\, 1}$ be a path of solutions in $C^{\infty}_{X}(M)$ to \eqref{starstar-s}
and for $s\in[0,\,1]$, let $h_{s}$ be the K\"ahler metric induced by $\sigma_{s}$.
Then for any $\alpha\in\left(0,\,\frac{1}{2}\right)$, the tensors $g_{s}^{-1}h_{s}$ and $h_{s}^{-1}g_{s}$ satisfy the following uniform estimate:
 \begin{equation*}
\sup_{0\,\leq\, s\,\leq\, 1}\left(\|g_{s}^{-1}h_{s}\|_{C_{\operatorname{loc}}^{0,\,2\alpha}(M)}+\|h_{s}^{-1}g_{s}
\|_{C_{\operatorname{loc}}^{0,\,2\alpha}(M)}\right)\leq C\left(n,\alpha,\omega,\sup_{s\in[0,1]}\| G_s\|_{C^{3}(M)}\right).
\end{equation*}
\end{corollary}

 \begin{proof}
By standard local interpolation inequalities applied to Propositions \ref{prop-C^2-est} and \ref{prop-C^3-est}, we see that
\begin{equation*}
\|g_{s}^{-1}h_{s}\|_{C_{\operatorname{loc}}^{0,\,2\alpha}(M)}\leq C\left(n,\alpha,\omega,\sup_{s\,\in\,[0,1]}\| G_s\|_{C^{3}(M)}\right).
\end{equation*}
Combining this estimate with Corollary \ref{coro-equiv-metrics-0}, it suffices to prove a uniform bound on the local $2\alpha$-H\"older norm of $h_{s}^{-1}g_{s}$.

To this end, observe the following: if $u$ is a positive function on $M$ lying in $C_{\operatorname{loc}}^{0,\,2\alpha}(M)$ which is uniformly bounded from below by a positive constant, then $[u^{-1}]_{2\alpha}\leq [u]_{2\alpha}(\inf_Mu)^{-2}$. By invoking Corollary \ref{coro-equiv-metrics-0} once again, this remark applied to $h_{s}^{-1}g_{s}$ gives us that
\begin{equation*}
\|h_{s}^{-1}g_{s}\|_{C_{\operatorname{loc}}^{0,\,2\alpha}(M)}\leq C\left(n,\alpha,\omega,\sup_{s\in[0,1]}\| G_s\|_{C^{3}(M)}\right),
\end{equation*}
as required.
\end{proof}

\subsubsection{Local bootstrapping}
We now improve the local regularity of our continuity path of solutions to \eqref{starstar-s}. The next estimate will be used
in deriving the subsequent weighted a priori estimates.

\begin{prop}\label{prop-loc-holder-C-3}
Let $(\vartheta_s)_{0\,\leq\, s\,\leq\, 1}$ be a path of solutions in $C^{\infty}_{X}(M)$
to \eqref{starstar-s}. Then for any $\alpha\in\left(0,\,\frac{1}{2}\right)$ and for any compact subset $K\subset M$,
\begin{equation*}
\sup_{0\,\leq\, s\,\leq\, 1}\|\vartheta_{s}\|_{C^{3,\,2\alpha}(K)}\leq C\left(n,\alpha,\omega, \sup_{s\in[0,\,1]}\| G_s\|_{C^{3}(M)},K\right).
\end{equation*}
\end{prop}

\begin{proof}
From the standard computations involved in the proof of the a priori $C^2$-estimate, we derive that
\begin{equation}\label{weloveele}
\begin{split}
\Delta_{\sigma_{s}}\left(\Delta_{\omega_{s}}\vartheta_{s}-\frac{X}{2}\cdot\vartheta_{s}\right) =& \Delta_{\omega_{s}}G_{s}+h_{s}^{-1}\ast g_{s}^{-1}\ast\Rm(g_{s})+\Rm(g_{s})\ast\nabla^{h_{s}}\bar{\nabla}^{h_{s}}\vartheta_{s}\ast h_{s}^{-1}\\
&\quad+g_{s}^{-1}\ast g_{s}^{-1}\ast\Rm(g_{s})+g_{s}^{-1}\ast h_{s}^{-1}\ast h_{s}^{-1}\ast \bar{\nabla}^{h_{s}}\nabla^{h_{s}}\bar{\nabla}^{h_{s}}\vartheta_{s}\ast\nabla^{h_{s}}\bar{\nabla}^{h_{s}}\nabla^{h_{s}} \vartheta_{s}\\
&\quad-\left(\Delta_{\sigma_{s}}-\Delta_{\omega_{s}}\right)\left(\frac{X\cdot \vartheta_{s}}{2}\right),
\end{split}
\end{equation}
where $\ast$ denotes the ordinary contraction of two tensors. Now, since $X$ is real holomorphic and $\vartheta_s$ is $JX$-invariant,
we see that
\begin{equation*}\label{lie-der-cov-der-X}
i\partial\bar{\partial}(X\cdot\vartheta_s)=\mathcal{L}_X(i\partial\bar{\partial}\vartheta_s)=\nabla^{g_s}_X(i\partial\bar{\partial}\vartheta_s)+i\partial\bar{\partial}\vartheta_s\ast\nabla^{g_s}X.
\end{equation*}
This leads to the following pointwise estimate:
\begin{equation}
\begin{split}\label{easy-obs-diff-lap}
\left|\left(\Delta_{\sigma_{s}}-\Delta_{\omega_{s}}\right)(X\cdot\vartheta_{s})\right|&
=\left|g_{s}^{-1}\ast h_{s}^{-1}\ast i\partial\bar{\partial}\vartheta_{s}\ast i\partial\bar{\partial}(X\cdot\vartheta_{s})\right|_{g_{s}}\\
&\leq C|h_{s}^{-1}|_{g_{s}}\cdot|i\partial\bar{\partial}\vartheta_{s}|_{g_{s}}\cdot \left(
|i\partial\bar{\partial}\vartheta_{s}|_{g_{s}}|\nabla^{g_{s}}X|_{g_{s}}+|\nabla^{g_s}i\partial\bar{\partial}\vartheta_{s}|_{g_{s}}|X|_{g_{s}}\right).
\end{split}
\end{equation}
By Propositions \ref{prop-C^2-est} and \ref{prop-C^3-est} together with \eqref{easy-obs-diff-lap}, we now see that
the $C^0$-norm of the right-hand side of \eqref{weloveele} is uniformly bounded on compact subsets and, thanks to Corollary \ref{coro-equiv-metrics}, so too are the coefficients of $\Delta_{\sigma_{s}}$ in the $C^{0,\,2\alpha}_{\operatorname{loc}}$-sense. As a result, by applying the Morrey-Schauder $C^{1,\,2\alpha}$-estimates, we see that for any $x\in M$ and for $\delta<\inj_{g_{s}}(M)$, we have that
\begin{equation*}
\left\|\Delta_{\omega_{s}}\vartheta_{s}-\frac{X}{2}\cdot\vartheta_{s}\right\|_{C^{1,\,2\alpha}(B_{g_{s}}(x,\,\delta))}\leq C(x,\,\delta,\,\alpha).
\end{equation*}
Finally, applying standard interior Schauder estimates for elliptic equations once again with respect to $\Delta_{\omega_{s},\,X}$, we deduce that
\begin{equation*}
\begin{split}
\|\vartheta_{s}\|_{C^{3,\,2\alpha}(B_{g_{s}}(x,\,\frac{\delta}{2}))}&\leq C(x,\,\delta,\,\alpha)\left(\left\|\Delta_{\omega_{s}}\vartheta_{s}-\frac{X}{2}\cdot\vartheta_{s}\right\|_{C^{1,\,2\alpha}(B_{g_{s}}(x,\,\delta))}
+\|\vartheta_{s}\|_{C^{1,\,2\alpha}(B_{g_{s}}(x,\,\delta))}\right)\\
&\leq C(x,\,\delta,\,\alpha).
\end{split}
\end{equation*}
\end{proof}

We now establish the following well-known local regularity result for solutions to \eqref{starstar-s}.
\begin{prop}\label{prop-loc-reg}
Let $G_{s}\in C^{k,\,\alpha}_{\operatorname{loc}}(M)$ for some $k\geq1$ and $\alpha\in(0,\,1)$, and let $\vartheta_{s}\in C^{3,\,\alpha}_{\operatorname{loc}}(M)$ be a solution to \eqref{starstar-s} with data $G_{s}$. Then $\vartheta_{s}\in C^{k+2,\alpha}_{\operatorname{loc}}(M)$. Moreover, for all $k\geq 1$, $\alpha\in(0,\,1)$, and compact subsets $K\subset M$,
\begin{equation*}
\begin{split}
\|\vartheta_{s}\|_{C^{k+2,\alpha}(K)}\leq C\left(n,\alpha,\omega, \sup_{s\in[0,\,1]}\| G_s\|_{C^{\max\{k,3\},\alpha}},K\right).
\end{split}
\end{equation*}
\end{prop}

\begin{proof}
We prove this proposition by induction on $k\geq 1$. The case $k=1$ is true by Proposition \ref{prop-loc-holder-C-3}, so let $G_{s}\in C^{k+1,\,\alpha}_{\operatorname{loc}}(M)$ and let $\vartheta_{s}\in C^{3,\,\alpha}_{\operatorname{loc}}(M)$ be a solution of \eqref{starstar-s}. Then by our induction hypothesis, $\vartheta_{s}\in C^{k+2,\alpha}_{\operatorname{loc}}(M)$.
Let $x\in M$ and choose local holomorphic coordinates defined on $B_{g_{s}}(x,\,\delta)$ for some $0<\delta<\inj_{g_{s}}(M)$. Then since $\vartheta_{s}$ satisfies
\begin{equation*}
G_{s}=\log\left(\frac{\sigma_{s}^n}{\omega_{s}^n}\right)-\frac{X}{2}\cdot\vartheta_{s},
\end{equation*}
we know that for $j=1,...,2n$, the derivative $\partial_j\vartheta_{s}$ satisfies
\begin{eqnarray*}
\Delta_{\sigma_{s}}\left(\partial_j\vartheta_{s}\right)=\partial_j\left(G_{s}+\frac{X}{2}\cdot\vartheta_{s}\right)\in C^{k,\,\alpha}_{\operatorname{loc}}(M).
\end{eqnarray*}
Since the coefficients of $\Delta_{\sigma_{s}}$ are in $C^{k,\,\alpha}_{\operatorname{loc}}$,
an application of the standard interior Schauder estimates for elliptic equations now gives us the
desired local regularity result, namely that $\partial_j\vartheta_{s}\in C^{k+2,\alpha}_{\operatorname{loc}}(M)$ for all $j=1,...,2n$, or equivalently, that $\vartheta_{s}\in C^{k+3,\alpha}_{\operatorname{loc}}(M)$ together with the expected estimate.
\end{proof}

\subsection{Weighted a priori estimates}\label{sec-wei-bd}

Thanks to the bounds on the $X$-derivative of solutions to \eqref{starstar-s} obtained in Corollary \ref{lowerbound2} and Proposition \ref{prop-bd-uni-X-psi}, together with the $C^2$-bound from Corollary \ref{coro-equiv-metrics-0}, the background potential function $f$ serves as a good barrier in the following precise sense.
\begin{lemma}\label{lemma-barrier-fct}
Let $(\vartheta_s)_{0\,\leq\, s\,\leq\, 1}$ be a path of solutions in $ C^{\infty}_{X}(M)$ to \eqref{starstar-s}. Then there exists a positive constant $C$ such that  for all $s\in[0,\,1]$,
\begin{equation*}
\begin{split}
\left|f_{\sigma_s}-f\right|\leq C\qquad\textrm{and}\qquad\left|\Delta_{\sigma_s,\,X}f+f\right|\leq C.
\end{split}
\end{equation*}
In particular, for all $\delta\neq 0$, there exists $C_{\delta}>0$ and a compact subset $K\subset M$ such that on $M\setminus K$,
\begin{equation*}
\begin{split}
\left|\Delta_{\sigma_s,\,X}f^{-\delta}-\delta f^{-\delta}\right|\leq C_{\delta}f^{-\delta-1}.
\end{split}
\end{equation*}
\end{lemma}
\noindent
This lemma is reminiscent of Lemma \ref{lemma-sub-sol-barrier} and so we omit its proof.

The next lemma establishes a general decay estimate for supersolutions to the drift Laplacian $\Delta_{\sigma_s,\,X}$.
\begin{lemma}\label{lemma-dec-sub-fct}
Let $\alpha,\,\lambda>0$, let $(\vartheta_s)_{0\,\leq\, s\,\leq\, 1}$ be a path of solutions in $C^{\infty}_{X}(M)$ to \eqref{starstar-s}, and let $u:M\rightarrow\R$ be a $C^2_{\operatorname{loc}}(M)$ function such that outside a sufficiently large compact set $K$ of $M$, for some constants $C_0,\,C_1\geq0$,
\begin{equation*}
\begin{split}
\Delta_{\sigma_s,\,X}u\leq \left(\lambda-\frac{C_0}{f^{\alpha}}\right)u+\frac{C_1}{f^{\lambda}}\qquad\textrm{whenever $u<0$}.
\end{split}
\end{equation*}
If $u$ is bounded on $M$, then there exists $C>0$ depending only on $\sup_{K}|u|$ such that
\begin{equation*}
\begin{split}
u\geq -C\left(\frac{\log f}{f^{\lambda}}\right).
\end{split}
\end{equation*}
\end{lemma}

\begin{proof}
Consider the Riemannian drift Laplacian $\Delta_{h_s,\,X}$ associated to $\sigma_s$. Using Lemma \ref{lemma-barrier-fct} applied to $\delta=-\lambda$,
we compute, on the complement of a compact subset of $M$ chosen so that $f>0$, that at a point where $u<0$, 
\begin{equation*}
\begin{split}
\Delta_{h_s,\,X}\left(f^{\lambda}\cdot u\right)&\leq \left(\Delta_{h_s,\,X}f^{\lambda}+2\left(\lambda-\frac{C_0}{f^{\alpha}}\right) f^{\lambda}\right)u+2C_1+2\lambda h_s(\nabla^{h_s}\log f,\nabla^{h_s}(f^{\lambda}\cdot u))\\
&\quad-2\lambda^2|\nabla^{h_s}\log f|^2_{h_s}(f^{\lambda}\cdot u)\\
&\leq -2\left(\frac{C_{-\lambda}}{f}+\frac{C_{0}}{f^{\alpha}}\right)(f^{\lambda}\cdot u)
-2\lambda^2|\nabla^{h_s}\log f|^2_{h_s}(f^{\lambda}\cdot u)\\
&\quad+2\lambda h_s(\nabla^{h_s}\log f,\nabla^{h_s}(f^{\lambda}\cdot u))+2C_1\\
&\leq -2\left(C_{0}f^{-\alpha_0}+\lambda^2|\nabla^{h_s}\log f|^2_{h_s}\right)\left(f^{\lambda}\cdot u\right)+2\lambda h_s(\nabla^{h_s}\log f,\nabla^{h_s}(f^{\lambda}\cdot u))+2C_1\\
&\leq -2C_{0}f^{-\alpha_0} \left(f^{\lambda}\cdot u\right)+2\lambda h_s(\nabla^{h_s}\log f,\nabla^{h_s}(f^{\lambda}\cdot u))+2C_1,
\end{split}
\end{equation*}
where $\alpha_0:=\min\{\alpha,1\}$ and where $C_{0}$ is a uniform positive constant (that may depend on $\lambda$) that may vary from line to line. Here we have used the fact that $u<0$ in the penultimate line and invoked Corollary \ref{coro-equiv-metrics-0} in the last line.

In order to absorb the term $-2C_0f^{-\alpha_0}$, we consider a function of the form $e^{Bf^{-\alpha_0}}$ for some $B>0$ to be chosen later.
To this end, we define $v:=f^{\lambda}\cdot u$ and compute using Lemma \ref{lemma-barrier-fct} as above
at a point where $u<0$ the following:
\begin{equation*}\begin{split}
\Delta_{h_s,\,X}\left(e^{Bf^{-\alpha_0}}\cdot v\right)&\leq \left(-2C_0f^{-\alpha_0}e^{Bf^{-\alpha_0}}+\Delta_{h_s,\,X}\left(e^{Bf^{-\alpha_0}} \right) \right)v+2\lambda e^{Bf^{-\alpha_0}}h_s(\nabla^{h_s}\log f,\nabla^{h_s}v)\\
&\quad+2C_1e^{Bf^{-\alpha_0}}+2h_s\left(\nabla^{h_s}e^{Bf^{-\alpha_0}},\nabla^{h_s}v\right)\\
&=\left(-2C_0f^{-\alpha_0}+e^{-Bf^{-\alpha_0}}\Delta_{h_s,\,X}\left(e^{Bf^{-\alpha_0}}\right)\right)\left(e^{Bf^{-\alpha_0}}\cdot v\right)\\
&\quad-2v\cdot h_s\left(\nabla^{h_s}\left(\lambda\log f+Bf^{-\alpha_0}\right),\nabla^{h_s}\left(e^{Bf^{-\alpha_0}}\right)\right)\\
&\quad+2 h_s\left(\nabla^{h_s}\left(\lambda\log f+Bf^{-\alpha_0}\right),\nabla^{h_s}\left(e^{Bf^{-\alpha_0}}\cdot v\right)\right)+2C_1e^{Bf^{-\alpha_0}}\\
&\leq \left(-2C_0f^{-\alpha_0}+e^{-Bf^{-\alpha_0}}\Delta_{h_s,\,X}\left(e^{Bf^{-\alpha_0}} \right) +2\alpha_0 B\lambda\frac{|\nabla^{h_s}f|^2_{h_s}}{f^{\alpha_0+2}}-2\alpha_{0}^{2}B^{2}\frac{|\nabla^{h_s}f|^2_{h_s}}{f^{2\alpha_0+2}}\right)\cdot\\
&\quad\cdot\left(e^{Bf^{-\alpha_0}}\cdot v\right)+2 h_s\left(\nabla^{h_s}\left(\lambda\log f+Bf^{-\alpha_0}\right),\nabla^{h_s}\left(e^{Bf^{-\alpha_0}}\cdot v\right)\right)+2C_1e^{Bf^{-\alpha_0}}.
\end{split}
\end{equation*}
Now, thanks to Lemma \ref{lemma-barrier-fct} yet again, we see, provided that $B:=2C_0\alpha_0^{-1}$, that
\begin{equation}\label{lindsay}
\begin{split}
V&:=-2C_0f^{-\alpha_0}+e^{-Bf^{-\alpha_0}}\Delta_{h_s,\,X}\left(e^{Bf^{-\alpha_0}} \right) +2\alpha_0 B\lambda\left(\frac{|\nabla^{h_s}f|^2_{h_s}}{f^{\alpha_0+2}}\right)-2\alpha_{0}^{2}B^2\left(\frac{|\nabla^{h_s}f|^2_{h_s}}{f^{2\alpha_0+2}}\right)\\
&=\underbrace{-2C_0f^{-\alpha_0} +B\Delta_{h_s,\,X}f^{-\alpha_0}}_{\geq\,2C_{0}f^{-\alpha_{0}}-2BC_{\alpha_{0}}f^{-1-\alpha_{0}}}+2\alpha_{0}B\lambda\left(\frac{|\nabla^{h_s}f|^2_{h_s}}{f^{\alpha_0+2}}\right)-\alpha_{0}^{2}B^2\left(\frac{|\nabla^{h_s}f|^2_{h_s}}{f^{2\alpha_0+2}}\right)\\
&\geq \left(2C_{0}+B\cdot\left(\left(2\alpha_{0}\lambda-\frac{\alpha_{0}^{2}B}{f^{\alpha_{0}}}\right)\left(\frac{|\nabla^{h_s}f|^2_{h_s}}{f^2}\right)-
\frac{2C_{\alpha_0}}{f}\right)\right)f^{-\alpha_0}\\
&\geq C_0f^{-\alpha_0}
\end{split}
\end{equation}
for $f\geq R_0$ is large enough.

Now, consider the function $w:=e^{Bf^{-\alpha_0}}\cdot v+A\log f=e^{Bf^{-\alpha_0}}\cdot f^{\lambda}\cdot u+A\log f$ for some $A>0$, and observe that
at a point where $u<0$, if $f\geq R_0$ is large enough, then thanks to \eqref{lindsay}, we have that
\begin{equation*}
\begin{split}
\left(\Delta_{h_s,\,X}-2 h_s(\nabla^{h_s}\left(\lambda\log f+Bf^{-\alpha_0}\right),\nabla^{h_s}(\,\cdot\,))\right) w&\leq V\cdot\left(v\cdot e^{Bf^{-\alpha_{0}}}\right)+2C_{1}e^{Bf^{-\alpha_{0}}}
+A\Delta_{h_s,\,X}\log f\\
&\quad-2 h_s(\nabla^{h_s}\left(\lambda\log f+Bf^{-\alpha_0}\right),\nabla^{h_s}(A\log f))\\
&\leq V\cdot w+2C_1e^{Bf^{-\alpha_0}}+A\Delta_{h_s,\,X}\log f\\
&\leq V\cdot w+2C_1e^{Bf^{-\alpha_0}}+A\left(\frac{\Delta_{h_s,\,X}f}{f}\right)\\
&\leq V\cdot w,
\end{split}
\end{equation*}
provided that $A>0$ is chosen sufficiently large for $B$ fixed.

In order to conclude the argument, we invoke the minimum principle. Indeed, consider the function $$w_{\varepsilon}:=w+\varepsilon f^{\lambda+\delta}=e^{Bf^{-\alpha_0}}\cdot f^{\lambda}\cdot u+A\log f+\varepsilon f^{\lambda+\delta}$$ with $\delta>0$ fixed. For each $\varepsilon>0$, this function is proper and bounded from below because $u$ is bounded by assumption. For each $\varepsilon>0$ therefore, it attains a minimum on $\{f\,\geq\,R_0\}$. First choose $R_{0}>0$, and then
$A$ and $B$, so that
\begin{equation}\label{diff-inequ-w}
\begin{split}
\left(\Delta_{h_s,\,X}-2 h_s(\nabla^{h_s}\left(\lambda\log f+Bf^{-\alpha_0}\right),\nabla^{h_s}(\,\cdot\,))\right) w_{\varepsilon}&\leq V\cdot w_{\varepsilon}\quad\textrm{holds on $\{f\,\geq\,R_0\}\cap\{u<0\}$.}
\end{split}
\end{equation}
These choices can be made independent of the choice of $\varepsilon>0$. If $w_{\varepsilon}$ attains its minimum on $\{f\,>\,R_0\}\cap\{u<0\}$, then the minimum principle
applied to \eqref{diff-inequ-w} tells us that $w_{\varepsilon}\geq 0$, i.e., $$\min_{\{f\,\geq\,R_0\}}w_{\varepsilon}\geq \min\left\{\min_{\{f\,= \,R_0\}}w_{\varepsilon},\,0\right\}\geq -C(R_0)\left(1+\max_{\{f\,= \,R_0\}}|u|\right).$$
Letting $\varepsilon\to0$ leads to the desired lower bound on $u$.
\end{proof}

We now use this lemma to establish an initial a priori decay estimate on the $X$-derivative of solutions to \eqref{starstar-s}.
\begin{prop}\label{decay-first-der-rad-der}
Let $(\vartheta_s)_{0\,\leq\, s\,\leq\, 1}$ be a path of solutions in $ C^{\infty}_{X}(M)$ to \eqref{starstar-s}.
Then there exist positive constants $C$ and $R_0$ such that for all $s\in[0,\,1]$,
\begin{equation*}
X\cdot \vartheta_s\geq -C\left(\frac{\log f}{f}\right)\qquad\textrm{on $\{f\geq R_0\}$}.
\end{equation*}
\end{prop}

\begin{remark}
The a priori lower bound on the $X$-derivative of solutions to \eqref{starstar-s} obtained in this proposition is not optimal. An optimal a priori lower bound will be given in Corollary \ref{coro-bds-Ck-weight}.
\end{remark}

\begin{proof}[Proof of Proposition \ref{decay-first-der-rad-der}]
From \eqref{dont-forget-me} we read that outside a sufficiently large uniform compact set of $M$,
\begin{equation*}
\begin{split}\label{I-am-back-0}
\Delta_{\sigma_{s},\,X}\left(\frac{X}{2}\cdot\vartheta_{s}\right)
&\leq -ne^{-\frac{G_s}{n}-\frac{X\cdot \vartheta_{s}}{2n}}\cdot\left(\frac{\omega_0^n}{\omega_s^n}\right)^{\frac{1}{n}}+\operatorname{tr}_{\omega_{s}}\omega_{0}+\frac{X}{2}\cdot G_{s}\\
&\leq n\left(\frac{G_s}{n}+\frac{X\cdot \vartheta_{s}}{2n}-1\right)\left(\frac{\omega_0^n}{\omega_s^n}\right)^{\frac{1}{n}}+\operatorname{tr}_{\omega_{s}}\omega_{0}+\frac{X}{2}\cdot G_{s}\\
&=\left(\frac{\omega_0^n}{\omega_s^n}\right)^{\frac{1}{n}}\left(\frac{X\cdot \vartheta_{s}}{2}\right)+G_{s}\cdot\left(\frac{\omega_0^n}{\omega_s^n}\right)^{\frac{1}{n}}+\frac{X}{2}\cdot G_{s}
-\left(n-\operatorname{tr}_{\omega_{s}}\omega_{0}\cdot \left(\frac{\omega_s^n}{\omega_0^n}\right)^{\frac{1}{n}}\right)\left(\frac{\omega_0^n}{\omega_s^n}\right)^{\frac{1}{n}}\\
&\leq \frac{X\cdot \vartheta_{s}}{2}+\frac{C}{f},
\end{split}
\end{equation*}
where $C$ is a positive constant that may vary from line to line. Here we have used the elementary inequality $1+x\leq e^x$ for all $x\in \R$ in the second line, and
the fact that $G_s=O(r^{-2})$ with $g_{0}$-derivatives by \eqref{austin}, $\omega_s=\omega_0+O(r^{-2})$, and Proposition \ref{prop-bd-uni-X-psi} in the last line.
Appealing to Lemma \ref{lemma-dec-sub-fct} (with $C_{0}=0$ and $\lambda=1$) now yields the result.
\end{proof}

We can now improve Corollary \ref{coro-equiv-metrics-0} to the following.
\begin{prop}\label{decay-sec-der}
Let $(\vartheta_s)_{0\,\leq\, s\,\leq\, 1}$ be a path of solutions in $ C^{\infty}_{X}(M)$ to \eqref{starstar-s}.
Then there exist positive constants $C$ and $R_0$ such that for all $s\in[0,\,1]$,
\begin{equation*}
\tr_{\omega_s}\sigma_s\leq n+C\left(\frac{\log f}{f}\right)\qquad\textrm{on $\{f\geq R_0\}$.}
\end{equation*}
\end{prop}

\begin{proof}
Applying Yau's computation \cite[equations (3.67) and (3.75)]{Bou-Eys-Gue} to the equation $\log\left(\frac{\sigma_{s}^n}{\omega_s^n}\right)=\frac{X}{2}\cdot\vartheta_s+G_s=:\widehat{G_s}$, we see that outside a sufficiently large compact set of $M$,
\begin{equation*}
\Delta_{\sigma_s}\tr_{\omega_{s}}\sigma_s\geq \frac{|\nabla^{h_{s}}\tr_{\omega_{s}}\sigma_s|^2_{h_{s}}}{\tr_{\omega_{s}}\sigma_s}+\Delta_{\omega_{s}}\widehat{G_s}-\frac{C}{f}\left(\tr_{\omega_{s}}\sigma_s\right)\left(\tr_{\sigma_s}\omega_{s}\right),
\end{equation*}
where we have used the quadratic curvature decay of $\omega_{s}$. Recall \eqref{macron} and the definition of $Q_{s}$ just before. Then we have that
\begin{equation*}
\begin{split}
\Delta_{\omega_{s}}\widehat{G_s}&=\Delta_{\omega_{s}}G_s+\tr_{\omega_{s}}i\partial\bar{\partial}\left(\frac{X}{2}\cdot \vartheta_s\right)\\
&=\Delta_{\omega_s}G_s+\tr_{\omega_s}\left(\mathcal{L}_{\frac{X}{2}}\sigma_s\right)-\tr_{\omega_s}\left(\mathcal{L}_{\frac{X}{2}}\omega_s\right)\\
&=\Delta_{\omega_s}G_s+\frac{X}{2}\cdot\tr_{\omega_s}\sigma_s+\left\langle\mathcal{L}_{\frac{X}{2}}\omega_s,\sigma_s\right\rangle_{\omega_s}-n
-\Delta_{\omega_{s}}Q_s+s_{\omega_s}\\
&=\underbrace{\Delta_{\omega_s}(F_s-F+\Phi_{s})}_{=\,O(f^{-1})}+\frac{X}{2}\cdot\tr_{\omega_s}\sigma_s+\langle\underbrace{\mathcal{L}_{\frac{X}{2}}\omega_{s}-
\omega_{s}}_{\substack{=\,i\partial\bar{\partial}Q_s-\rho_{\omega_{s}} \\ =\,O_{\omega_{s}}(f^{-1})}},\sigma_s\rangle_{\omega_{s}}
+\left(\left\langle\omega_{s},\sigma_s\right\rangle_{\omega_{s}}-n\right)+s_{\omega_s}\\
&\geq \frac{X}{2}\cdot\tr_{\omega_{s}}\sigma_s+\left(\tr_{\omega_{s}}\sigma_s-n\right)-\frac{C}{f}\cdot\tr_{\omega_{s}}\sigma_s-\frac{C}{f}.
\end{split}
\end{equation*}
Here, recall that $s_{\omega_s}$ denotes the scalar curvature of $\omega_s$ and in the last line
we have used the (pointwise) Cauchy-Schwarz inequality together with the fact that $\rho_{\omega_{s}}=O(f^{-1})$ uniformly in $s$.
Combining the previous two inequalities therefore leads to the lower bound
\begin{equation*}
\begin{split}\label{diff-ineq-c2}
\Delta_{\sigma_{s},\,X}\left(\tr_{\omega_s}\sigma_{s}-n\right)&=\Delta_{\sigma_{s}}\tr_{\omega_s}\sigma_{s}-\frac{X}{2}\cdot\tr_{\omega_s}\sigma_{s}\\
&\geq \frac{|\nabla^{h_s}\tr_{\omega_s}\sigma_{s}|^2_{h_{s}}}{\tr_{\omega_s}\sigma_{s}}-\frac{C}{f}+\left(\tr_{\omega_s}\sigma_{s}-n\right)-\frac{C}{f}\cdot\tr_{\omega_s}\sigma_{s}
-\frac{C}{f}\left(\tr_{\omega_s}\sigma_{s}\right)\left(\tr_{\sigma_{s}}\omega_s\right)\\
&\geq \left(\tr_{\omega_s}\sigma_{s}-n\right)-\frac{C}{f},
\end{split}
\end{equation*}
where we have used the upper bounds on $\tr_{\sigma_{s}}\omega_s$ and  $\tr_{\omega_{s}}\sigma_s$ given by Corollary \ref{coro-equiv-metrics-0} in the last line. Applying Lemma \ref{lemma-dec-sub-fct} now gives us the desired result.
\end{proof}

We can now derive the weighted $C^{2}$-estimate.
\begin{corollary}[Weighted $C^{2}$-bounds]\label{coro-weak-wei-C2-est}
Let $(\vartheta_s)_{0\,\leq\, s\,\leq\, 1}$ be a path of solutions in $ C^{\infty}_{X}(M)$ to \eqref{starstar-s}.
Then there exist positive constants $C$ and $R_0$ such that for all $s\in[0,\,1]$,
\begin{equation*}
|\Delta_{\omega_s}\vartheta_s|+|X\cdot\vartheta_s|\leq C\left(\frac{\log f}{f}\right)\qquad\textrm{on $\{f\geq R_0\}$.}
\end{equation*}
Moreover, there exist positive constants $C$ and $R_0$ such that for all $s\in[0,\,1]$,
\begin{equation*}
|\partial\bar{\partial}\vartheta_s|_{g_s}\leq C\left(\frac{\log f}{f}\right)^{\frac{1}{2}}\qquad\textrm{on $\{f\geq R_0\}$.}
\end{equation*}
\end{corollary}

\begin{proof}
The upper bound on $\Delta_{\omega_s}\vartheta_s$ follows from Proposition \ref{decay-sec-der} after taking the trace with respect to $\omega_{s}$ of the inequality
$\sigma_{s}=\omega_{s}+i\partial\bar{\partial}\vartheta_{s}>0$ . The lower bound follows from Corollary \ref{decay-first-der-rad-der} after noting that
from \eqref{starstar-s} and the basic inequality $\log(1+x)\leq x$ for all $x>-1$ that
\begin{equation*}
\Delta_{\omega_s}\vartheta_s\geq\frac{X}{2}\cdot\vartheta_s+G_{s}.
\end{equation*}

As for the bound on $|X\cdot\vartheta_s|$, thanks to Corollary \ref{decay-first-der-rad-der}, all we are left to do is to prove the corresponding upper bound
on $X\cdot\vartheta_s$. To this end, from the previous inequality we derive that
\begin{equation*}
\frac{X}{2}\cdot\vartheta_s\leq \Delta_{\omega_s}\vartheta_s-G_s\leq \frac{C\log f}{f}+\frac{C}{f},
\end{equation*}
which leads to the desired decay on $X\cdot\vartheta_s$.

Regarding the resulting decay on $\partial\bar{\partial}\vartheta_s$, setting $\tilde{\omega}:=\omega_s$ and $\psi:=\vartheta_s$ in
\eqref{equ:taylor-exp}, we see from \eqref{starstar-s} that
\begin{equation*}
 \begin{split}
 \Delta_{\omega_s}\vartheta_s-\frac{X}{2}\cdot\vartheta_s-\int_0^1\int_0^{u}\arrowvert \partial\bar{\partial}\vartheta_s\arrowvert^2_{(g_s)_{\tau\vartheta_s}}\,d\tau\,du=G_s,
 \end{split}
 \end{equation*}
where $(g_s)_{\tau\vartheta_s}$ is the Riemannian metric induced by the K\"ahler form $\omega_s+i\tau\partial\bar{\partial}\vartheta_s$ for $\tau\in [0,1]$. Since there exists $C>0$ such that $C^{-1}g_s\leq (g_s)_{\tau\vartheta_s}\leq Cg_s$ uniformly for all $s,\,\tau\in [0,1]$ by Corollary \ref{coro-equiv-metrics-0}, the desired decay on $\partial\bar{\partial}\vartheta_s$ follows immediately from the above.
\end{proof}

Next, we give the weighted $C^{3}$-estimate.
\begin{prop}[Weighted $C^{3}$-bounds]\label{prop-C3-eswei}
Let $(\vartheta_s)_{0\,\leq\, s\,\leq\, 1}$ be a path of solutions in $ C^{\infty}_{X}(M)$ to \eqref{starstar-s}.
Then there exist positive constants $C$ and $R_0$ such that for all $s\in[0,\,1]$,
\begin{equation*}
|\nabla^{g_s}\partial\bar{\partial}\vartheta_s|\leq C\left(\frac{\log f}{f^{\frac{3}{2}}}\right)\qquad\textrm{on $\{f\geq R_0\}$.}
\end{equation*}
Moreover, there exist positive constants $C$ and $R_0$ such that for all $s\in[0,\,1]$,
\begin{equation*}
|\partial\bar{\partial}\vartheta_s|_{h_s}\leq C\left(\frac{\log f}{f}\right)\qquad\textrm{on $\{f\geq R_0\}$.}
\end{equation*}
\end{prop}

\begin{proof}
From equation \eqref{weak-subsol-sqrt-S} of Lemma \ref{lemma-S}, we see, after using the bounds just derived in Corollary \ref{coro-weak-wei-C2-est}, that for all $\varepsilon\in\left(0,\frac{1}{2}\right)$, there exists positive constants $C_{\varepsilon}$, $C$, and $R_0$ such that
\begin{equation*}
\begin{split}
\left(\Delta_{\sigma_{s}}-\frac{X}{2}\cdot\right)\sqrt{S}&\geq \left(\frac{3}{2}-\frac{C_{\varepsilon}}{f^{\frac{1}{2}-\varepsilon}}\right)\sqrt{S}-\frac{C}{f^{\frac{3}{2}}}\quad\textrm{on $\{f\geq R_{0}\}$}.
\end{split}
\end{equation*}
Applying Lemma \ref{lemma-dec-sub-fct} with $\lambda=\frac{3}{2}$, $\varepsilon=\frac{1}{4}$, and $\alpha=\frac{1}{4}$
to $u:=-\sqrt{S}$ gives first inequality of the proposition. The second inequality
follows by integrating the previously achieved bound on $\sqrt{S}$ along radial lines.
\end{proof}

This leads on to the weighted $C^{k}$-estimate.
\begin{prop}[Weighted $C^k$-bounds]\label{prop-bds-Ck-weight}
Let $(\vartheta_s)_{0\,\leq\, s\,\leq\, 1}$ be a path of solutions in $ C^{\infty}_{X}(M)$ to \eqref{starstar-s}.
Then there exists a compact subset $K\subset M$ such that for all $k\geq 0$ and $\varepsilon>0$, there exists a positive constant $C_{k,\,\varepsilon}$ such that for all $s\in[0,\,1]$,
\begin{equation}\label{first-prop-bds-Ck-weight}
\begin{split}
f^{\frac{k}{2}+1-\varepsilon}\left(|\nabla^{h_s,\,k}\Rm(h_s)|_{h_{s}}+|\nabla^{g_{s},\,k}(h_s-g_s)|_{h_s}\right)\leq C_{k,\,\varepsilon}\qquad\textrm{on $M\setminus K$.}
\end{split}
\end{equation}
\end{prop}

\begin{remark}
Notice that Proposition \ref{prop-bds-Ck-weight} does not establish sharp bounds on the covariant derivatives of the curvature of $h_s$ and of the difference $h_s-g_s$ because of the exponent $\varepsilon$ appearing on the left-hand side of \eqref{first-prop-bds-Ck-weight}. One can obtain sharp bounds once \eqref{first-prop-bds-Ck-weight} has been established by bootstrapping the proof of this proposition and adapting the arguments wherever necessary. However, we have postponed this to Corollary \ref{coro-bds-Ck-weight}.
\end{remark}

\begin{proof}[Proof of Proposition \ref{prop-bds-Ck-weight}]
Observe that the estimates on $|\nabla^{g_{s},\,k}(h_s-g_s)|_{h_{s}}$ for $k=0$ and $k=1$ have already been established in Proposition \ref{prop-C3-eswei}.

We begin with the following claim that identifies the elliptic equation satisfied by the curvature $\Rm(h_s)$.
\begin{claim}\label{claim-Rm-bd}
Write $\frac{1}{2}\Delta_{\R}$ for $\Delta_{\sigma_s}+\overline{\Delta_{\sigma_s}}$ acting on tensors. Then
\begin{equation*}
\begin{split}
\frac{1}{2}\left(\Delta_{\R}-\nabla^{h_s}_{X}\right)\Rm(h_s)&= \Rm(h_s)+\left(\Rm(h_s)-\Rm(g_s)\right)\\
&\quad+\left(\Rm(h_s)+(h_s-g_s)+\Lambda_s\right)\ast \Rm(h_s)+Q(h_s,g_s),\\
Q(h_s,g_s)&:=\overline{\nabla}^{h_s}\Psi\ast \Lambda_s+\Psi\ast\Psi+\Psi\ast\Psi\ast \Lambda_s+\nabla^{g_s}\Lambda_s\ast \Psi+\nabla^{g_s,\,2}\Lambda_s.
\end{split}
\end{equation*}
Here we omit the dependence of the contractions on $g_s$ and $h_s$.
\end{claim}

\begin{proof}[Proof of Claim \ref{claim-Rm-bd}]
We follow closely and adapt the proof of \cite[Lemma $3.2.10$]{Bou-Eys-Gue} to our setting. As in the proof of Lemma \ref{lemma-S}, we consider the one-parameter families of metrics defined by $h_s(\tau):=(-\tau)(\varphi^{X}_{\tau})^*h_s$ and $g_s(\tau):=(-\tau)(\varphi^{X}_{\tau})^*g_s$,
where $(\varphi^{X}_{\tau})_{\tau\,<\,0}$ is the one-parameter family of diffeomorphisms generated by $\frac{X}{2(-\tau)}$ with $\varphi^{X}_{\tau}\big|_{\tau\,=\,-1}=\Id_M$.

On one hand, we see that
\begin{equation}
\begin{split}\label{one-hand-Rm}
\frac{\partial}{\partial \tau}\bigg\vert_{\tau=-1} \Rm(h_s(\tau))&=\frac{\partial}{\partial \tau}\bigg\vert_{\tau=-1} (-\tau)(\varphi^{X}_{\tau})^{\ast}\Rm(h_s)=-\Rm(h_s)+\mathcal{L}_{\frac{X}{2}}\Rm(h_s)\\
&=\nabla^{h_s}_{\frac{X}{2}}\Rm(h_s)+\Rm(h_s)+\Rm(h_s)\ast \Ric(h_s)+\Rm(h_s)\ast (\Lambda_s+h_s-g_s).
\end{split}
\end{equation}
Here, we have used the fact that for a $(4,\,0)$-tensor $T$,
\begin{equation*}
\begin{split}
\mathcal{L}_XT(\cdot\,,\cdot\,,\cdot\,,\cdot\,)=\nabla^{h_s}_XT(\cdot\,,\cdot\,,\cdot\,,\cdot\,)+T(\nabla^{h_s}_{\cdot}X,\cdot\,,\cdot\,,\cdot\,)+T(\cdot\,,\nabla^{h_s}_{\cdot}X,\cdot\,,\cdot\,)+T(\cdot\,,\cdot\,,\nabla^{h_s}_{\cdot}X,\cdot\,)+T(\cdot\,,\cdot\,,\cdot\,,\nabla^{h_s}_{\cdot}X)
\end{split}
\end{equation*}
and the fact that $\nabla^{h_s}X=\frac{1}{2}\mathcal{L}_{X}h_s=h_s-\Ric(h_s)+\Lambda_s-(h_s-g_s)$ by virtue of \eqref{starstar-s}.
On the other hand, by the properties of the curvature tensor of a K\"ahler metric (namely $h_s$ here), one has that
\begin{equation}\label{snd-hand-Rm}
\begin{split}
\frac{\partial}{\partial \tau}\bigg\vert_{\tau=-1} \Rm(h_s(\tau))_{i\bar{\jmath}k\bar{\ell}}&=-\frac{\partial}{\partial \tau}\bigg\vert_{\tau=-1}h_s(\tau)_{p\bar{\jmath}}\partial_{\bar{\ell}}\Gamma(h_s)_{ik}^p-(h_s)_{p\bar{\jmath}}\partial_{\bar{\ell}}\left(\frac{\partial}{\partial \tau}\bigg\vert_{\tau=-1}\Gamma(h_s(\tau))_{ik}^p\right)\\
&=-(\Ric(h_s)-\Lambda_s+(h_s-g_s))_{p\bar{\jmath}} \Rm(h_s)_{i\bar{\ell}k}^p\\
&\quad+\nabla^{h_s}_{\bar{\ell}}\nabla^{h_s}_i(\Ric(h_s)-\Lambda_s+(h_s-g_s))_{k\bar{\jmath}}.
\end{split}
\end{equation}
Now, by choosing normal holomorphic coordinates with respect to $g_s$ at a fixed point, we see that
\begin{equation*}
\begin{split}\label{thd-hand-Rm}
\nabla^{h_s}_{\bar{\ell}}\nabla^{h_s}_i(g_s)_{k\bar{\jmath}}&=\partial_{\bar{\ell}}\partial_{i}(g_s)_{k\bar{\jmath}}-\partial_{\bar{\ell}}\Gamma(h_s)_{ik}^p(g_s)_{p\bar{\jmath}}+\Gamma(h_s)_{\bar{\ell}\bar{\jmath}}^{\bar{q}}\Gamma(h_s)_{ik}^{p}(g_s)_{p\bar{q}}\\
&=-\Rm(g_s)_{i\bar{\ell}k\bar{\jmath}}-\partial_{\bar{\ell}}\Gamma(h_s)_{ik}^p(g_s-h_s)_{p\bar{\jmath}}+\Rm(h_s)_{i\bar{\ell}k\bar{\jmath}}+\Psi_{\bar{\ell}\bar{\jmath}}^{\bar{q}}\Psi_{ik}^{p}(g_s)_{p\bar{q}}\\
&=\Rm(h_s)_{i\bar{\jmath}k\bar{\ell}}-\Rm(g_s)_{i\bar{\jmath}k\bar{\ell}}+\Rm(h_s)_{i\bar{\ell}k}^p(g_s-h_s)_{p\bar{\jmath}}+\Psi_{\bar{\ell}\bar{\jmath}}^{\bar{q}}\Psi_{ik}^{p}(g_s)_{p\bar{q}},
\end{split}
\end{equation*}
where we have used the Bianchi identity for the curvature tensor of a K\"ahler manifold \cite[Proposition $3.1.2$]{Bou-Eys-Gue} in the last line.
In addition, we have that
\begin{equation*}
\begin{split}\label{fth-hand-Rm}
\nabla^{h_s}_{\bar{\ell}}\nabla^{h_s}_i(\Lambda_s)_{k\bar{\jmath}}&=\nabla^{g_s}_{\bar{\ell}}\nabla^{g_s}_i(\Lambda_s)_{k\bar{\jmath}}+\left(\nabla^{g_s}\Lambda_s\ast \Psi+\Lambda_s\ast(\underbrace{\Rm(h_s)-\Rm(g_s)}_{=-\overline{\nabla}^{h_s}\Psi})+\Lambda_s\ast \Psi\ast \Psi\right)_{\bar{\ell} ik\bar{\jmath}}.
\end{split}
\end{equation*}

In order to complete the proof, we recall a computation made in the proof of \cite[Lemma 3.2.10]{Bou-Eys-Gue}, namely that
\begin{equation}\label{fith-hand-Rm}
\frac{1}{2}\Delta_{\R}\Rm(h_s)_{i\bar{\jmath}k\bar{\ell}}=\nabla^{h_s}_{\bar{\ell}}\nabla^{h_s}_k\Ric(h_s)_{i\bar{\jmath}}+\Rm(h_s)\ast \Rm(h_s)_{i\bar{\jmath}k\bar{\ell}}=\nabla^{h_s}_{\bar{\ell}}\nabla^{h_s}_i\Ric(h_s)_{k\bar{\jmath}}+\Rm(h_s)\ast \Rm(h_s)_{i\bar{\jmath}k\bar{\ell}},
\end{equation}
where we have used the Bianchi identity in the last equality. Combining \eqref{one-hand-Rm}--\eqref{fith-hand-Rm}, we arrive at the desired assertion.
\end{proof}

From this, we can now derive a $C^{0}$-bound for the curvature of $h_{s}$, as presented in the next claim.
\begin{claim}\label{claim-curv}
There exists a positive constant $C$ such that for all $s\in[0,\,1]$,
\begin{equation*}
\sup_M|\Rm(h_s)|_{h_s}\leq C.
\end{equation*}
\end{claim}

\begin{proof}[Proof of Claim \ref{claim-curv}]
By Claim \ref{claim-Rm-bd}, the function $|\Rm(h_s)|_{h_s}$ weakly satisfies on the domain\linebreak $\{f\geq R_0\}$ with $R_0$ chosen large enough the differential inequality
\begin{equation}
\begin{split}\label{oh-my-curv-wei}
\frac{1}{2}\Delta_{h_s,\,X}|\Rm(h_s)|_{h_s}&\geq \left(1-C\left(|\Lambda_s|_{h_s}+|h_s-g_s|_{h_s}\right)\right)|\Rm(h_s)|_{h_s}-|\overline{\nabla}^{h_s}\Psi|_{h_s}\\
&\quad-C|\Rm(h_s)|_{h_s}^2-|Q(h_s,g_s)|_{h_s}\\
&\geq \left(1-\frac{C\log f}{f}\right)|\Rm(h_s)|_{h_s}-|\overline{\nabla}^{h_s}\Psi|_{h_s}-C|\overline{\nabla}^{h_s}\Psi|_{h_s}^2-\frac{C}{f^2}\\
&\quad-C\left(S+|\overline{\nabla}^{h_s}\Psi|_{h_s}|\Lambda_s|_{h_s}\right)-C|\nabla^{g_s}\Lambda_s|_{h_s}\sqrt{S}-C|\nabla^{g_s,\,2}\Lambda_s|_{h_s}\\
&\geq \left(1-\frac{C\log f}{f}\right)|\Rm(h_s)|_{h_s}-\left(1+\frac{C}{f}\right)|\overline{\nabla}^{h_s}\Psi|_{h_s}-C|\overline{\nabla}^{h_s}\Psi|_{h_s}^2-CS-\frac{C}{f^2},
\end{split}
\end{equation}
where we have used Proposition \ref{prop-C3-eswei} in the second inequality together with the asymptotic properties of the covariant derivatives of $g_s$ and $\Lambda_s$ (the latter being defined in Lemma \ref{lemma-S}). As always, $C$ denotes a positive constant that may vary from line to line. Also observe that Lemma \ref{lemma-S} asserts that $S$ satisfies the inequality
\begin{equation}
\begin{split}\label{beauty-S-wei}
\frac{1}{2}\Delta_{h_{s},\,X}S&\geq |\nabla^{h_s} \Psi|^2_{h_s}+|\overline{\nabla}^{h_s}\Psi|_{h_s}^2+3S-C\left(\frac{\log f}{f}\right)S-\left(\frac{C}{f^{\frac{3}{2}}}\right)\sqrt{S}\\
&\geq |\overline{\nabla}^{h_s}\Psi|_{h_s}^2+\left(3-C\left(\frac{\log f}{f}\right)\right)S-\frac{C}{f^{2}}
\end{split}
\end{equation}
on $\{f\geq R_0\}$ with $R_0$ large enough (but uniform in $s\in[0,\,1]$). Here, we have used the asymptotic properties of the covariant derivatives of $g_s$ and $\Lambda_s$ and Lemma \ref{lemma-S} in the first inequality, together with Young's inequality (after writing $f^{\frac{3}{2}}=f^{1+\frac{1}{2}}$) in the second.

Consider the sum $AS+|\Rm(h_s)|_{h_s}$ on $\{f>R_0\}$ for $R_0$ and $A>0$ to be determined later. Combining \eqref{oh-my-curv-wei} and \eqref{beauty-S-wei} gives for a positive constant $C$ that may vary from line to line, the inequality
\begin{equation*}
\begin{split}\label{never-stop-subsol-curv}
\frac{1}{2}\Delta_{h_{s},\,X}\left(AS+|\Rm(h_s)|_{h_s}\right)&\geq (A-C) |\overline{\nabla}^{h_s} \Psi|^2_{h_s}-\left(1+\frac{C}{f}\right)|\overline{\nabla}^{h_s} \Psi|_{h_s}-\frac{C}{f^{2}}\\
&\quad+\left(3A-C-AC\left(\frac{\log f}{f}\right)\right)S-A\left(\frac{C}{f^{2}}\right)+\left(1-\frac{C\log f}{f}\right)|\Rm(h_s)|_{h_s}\\
&\geq \frac{1}{2}\left(AS+|\Rm(h_s)|_{h_s}\right)-(1+A)\frac{C}{f^2}-C,
\end{split}
\end{equation*}
provided that $A:=C+1$ and $R_0$ is chosen uniformly large enough. Here we have used Young's inequality in the second inequality. Henceforth fix such constants $A$ and $R_0$.
Then the maximum principle implies that
\begin{equation*}
\sup_{\{f\,\geq\, R_0\}}\left(AS+|\Rm(h_s)|_{h_s}\right)=\max\left\{2(1+A)\frac{C}{R_0^2}+2C,\,\max_{\{f\,=\,R_0\}}\left(AS+|\Rm(h_s)|_{h_s}\right)\right\}.
\end{equation*}
The right-hand side is uniformly bounded from above thanks to Proposition \ref{prop-loc-reg}, which proves the boundedness of the curvature.
\end{proof}

Next, multiplying \eqref{oh-my-curv-wei} across by $|\Rm(h_s)|_{h_s}$, we derive that on $\{f\geq R_0\}$ with $R_0$ large enough, the lower bound
\begin{equation}
 \begin{split}\label{oh-my-curv-wei-bis}
\frac{1}{2}\Delta_{h_s,\,X}|\Rm(h_s)|^2_{h_s}&= |\nabla^{h_s}\Rm(h_s)|^2_{h_s}+|\Rm(h_s)|_{h_s}\Delta_{h_s,\,X}|\Rm(h_s)|_{h_s}\\
&\geq
 |\nabla^{h_s}\Rm(h_s)|^2_{h_s}+\left(2-\frac{C\log f}{f}\right)|\Rm(h_s)|^2_{h_s}\\
 &\quad-\left(2+\frac{C}{f}\right)|\overline{\nabla}^{h_s}\Psi|_{h_s}|\Rm(h_s)|_{h_s}-C|\overline{\nabla}^{h_s}\Psi|_{h_s}^2|\Rm(h_s)|_{h_s}\\
 &\quad-CS|\Rm(h_s)|_{h_s}-\frac{C}{f^2}|\Rm(h_s)|_{h_s}\\
 &\geq \left(2-\frac{C\log f}{f}\right)|\Rm(h_s)|^2_{h_s}-C|\overline{\nabla}^{h_s}\Psi|^2_{h_s}-\frac{C}{f^2},
\end{split}
\end{equation}
where we have used the fact that $-\overline{\nabla}^{h_s}\Psi=\Rm(h_s)-\Rm(g_s)$ and Young's inequality together with Claim \ref{claim-curv} in the last inequality. We have also invoked Proposition \ref{prop-C3-eswei} in the last line. Again, $C$ denotes a uniform positive constant that may vary from line to line. Combining \eqref{beauty-S-wei} and \eqref{oh-my-curv-wei-bis} now gives that
\begin{equation*}
\begin{split}\label{never-stop-subsol-curv}
\frac{1}{2}\Delta_{h_{s},\,X}\left(AS+|\Rm(h_s)|^2_{h_s}\right)&\geq (A-C) |\overline{\nabla}^{h_s} \Psi|^2_{h_s}+\left(2-\frac{C\log f}{f}\right)|\Rm(h_s)|^2_{h_s}-(1+A)\frac{C}{f^2}\\
&\quad +A\left(3-C\left(\frac{\log f}{f}\right)\right)S\\
&\geq \left(2-\frac{C\log f}{f}\right)\left(AS+|\Rm(h_s)|^2_{h_s}\right)-(1+A)\frac{C}{f^2},
\end{split}
\end{equation*}
provided that $A:=C$ and $R_0$ is chosen uniformly large enough. Applying Lemma \ref{lemma-dec-sub-fct} to $u:=-AS-|\Rm(h_s)|^2_{h_s}$ with $\lambda:=2$ and $\alpha:=\frac{1}{2}$ now gives us
the desired decay on $|\Rm(h_s)|_{h_s}$.

Regarding the higher covariant derivative bounds, one can couple the equations satisfied by $\Psi$ and $\Rm(h_s)$ from Lemma \ref{lemma-S} and Claim \ref{claim-Rm-bd} respectively to prove by induction that \eqref{first-prop-bds-Ck-weight} holds. Observe that Claim \ref{claim-Rm-bd} only contains zeroth order terms depending on $\Psi$, since $\overline{\nabla}^{h_s}\Psi=\linebreak-\Rm(h_s)+\Rm(g_s)$. In particular, upon differentiating the equation satisfied by $\Rm(h_s)$, the non-homogeneous terms arising from $\Psi$ (and its covariant derivatives) can be controlled
using the function $S$ from Lemma \ref{lemma-S}. This allows one to apply a maximum principle-type argument to obtain a bound on $\nabla^{h_s}\Rm(h_s)$. Differentiating the equation satisfied by $\Psi$ from Lemma \ref{lemma-S} then yields an equation for $\nabla^{h_s}\Psi$, whose linear coefficients in front of the zeroth order terms in $\Psi$ and $\nabla^{h_s}\Psi$ now decay appropriately, and so on.
\end{proof}

To close this series of a priori weighted estimate, we first need to show that the data $G_{s}$ lies in the correct function space, namely $\mathcal{C}^{\infty}_{X}(M)$.
\begin{lemma}\label{lemma-check-data-F-s-fct-space}
For each $k\geq 0$ and $\alpha\in\left(0,\frac{1}{2}\right)$, there exists $C(k,\,\alpha)>0$ such that for all $s\in[0,\,1]$,
\begin{equation*}
\|G_s\|_{C^{2k,\,2\alpha}_{X,\,2}(M)}\leq C(k,\alpha).
\end{equation*}
In particular, $\sup_{s\in[0,\,1]}\|F_s\|_{\mathcal{C}^{2k,\,2\alpha}_{X}(M)}<\infty.$
\end{lemma}

\begin{proof}
We work on the complement of a fixed compact subset $K$ of $M$ containing the exceptional set $E$ of the resolution for which $\omega=\omega_{0}+\rho_{\omega_{0}}$
and $f=\frac{r^{2}}{2}-n$ on $M\setminus K$ (cf.~Proposition \ref{mainprop}).

First observe that on $M\setminus K$,
\eqref{normal12} can be rewritten as
\begin{equation*}
\begin{split}
-\frac{X}{2}\cdot F&=\Delta_{\omega}f-\Delta_{\omega_0}f=\tr_{\omega}\omega_0-\tr_{\omega_0}\omega_0,
\end{split}
\end{equation*}
because $i\partial \bar{\partial}f=\omega_0$ on this set. In particular, if $\omega(\tau):=(-\tau)\varphi_{\tau}^*\omega$, then at all points
$x\in M\setminus K$ and $\tau\in(-\iota_0r(x)^2,\,0)$ with $\iota_0>0$ small enough, we have that
\begin{equation}\label{persistence}
\begin{split}
-(-\tau)\partial_{\tau}\tilde{F}(\cdot\,,\tau)=-\varphi_{\tau}^*\left(\frac{X}{2}\cdot F\right)=\tr_{\omega(\tau)}\omega_0-\tr_{\omega_0}\omega_0=\tr_{\omega_0+(-\tau)\rho_{\omega_0}}\omega_0-\tr_{\omega_0}\omega_0,
\end{split}
\end{equation}
because $(-\tau)\varphi_{\tau}^*\omega_0=\omega_0$ and $\varphi_{\tau}^*\rho_{\omega_0}=\rho_{\varphi_{\tau}^*\omega_0}=\rho_{(-\tau)^{-1}\omega_0}=\rho_{\omega_0}$
at such points $x$ and times $\tau$.

We next show that $X\cdot F$, and correspondingly $F$, lie in the correct function space.
\begin{claim}\label{claim-partial-der-F}
\begin{equation*}
\|X\cdot F\|_{C^{2k,\,2\alpha}_{X,\,2}(M)}\leq C(k,\alpha).
\end{equation*}
Moreover, $\|F\|_{C^{2k,\,2\alpha}_{X,\,0}(M)}\leq C(k,\alpha)$.
\end{claim}

\begin{proof}[Proof of Claim \ref{claim-partial-der-F}]
By construction of the data $\omega$ and $F$ and the definition of the function spaces, it suffices to prove that there exists a compact subset $K\subset M$ such that for all $i,\,j\geq 0$, $$\sup_{s\,\in\,[0,\,1]}\sup_{x\,\in\,M\setminus K}\sup_{P_{r_x}(x)}r_x^{i+2j+2}\cdot|\nabla^{g,\,i}(\partial_{\tau}^{(j)})\partial_{\tau}\tilde{F}|_{g}<\infty.$$
To this end, for a fixed point $x\in M\setminus K$ and $\tau\in(-\iota_0r(x)^2,\,0)$, let $$U(\tau):=(-\tau)^{-1}(\tr_{\omega_0+(-\tau)\rho_{\omega_0}}\omega_0-\tr_{\omega_0}\omega_0).$$
Choose an orthonormal frame based at $x$ with respect to the Riemannian metric induced by $\omega_0$ such that $\rho_{\omega_0}$ is diagonal with eigenvalues $\lambda_{1},\ldots,\lambda_{n}$.
Then $U(\tau)=-\sum_{i=1}^{n}\frac{\lambda_i}{1-\tau\lambda_i}$ so that for all $k\geq 1$, we see that
\begin{equation*}
\begin{split}
U^{(k)}(\tau)&=-\sum_{i\,=\,1}^{n}k!\frac{\lambda_i^{k+1}}{(1-\tau\lambda_i)^{k+1}}.
\end{split}
\end{equation*}
Multiplying \eqref{persistence} across by $r(x)^{2j+2}$ and noting that $ |\tau|\leq \iota_0r(x)^2$ therefore leads to the bounds
\begin{equation*}
\sup_{s\,\in\,[0,\,1]}\sup_{x\,\in\,M\setminus K}\sup_{P_{r_x}(x)}r_x^{2j+2}\cdot|\partial_{\tau}^{(j)}\partial_{\tau}\tilde{F}|<\infty,
\end{equation*}
because $\rho_{\omega_0}=O(r^{-2})$ with $g_0$-derivatives.

The corresponding result for $i\geq 1$ follows by differentiating \eqref{persistence} $i$-times, which can be expressed schematically in tensorial form as
$$\partial_{\tau}^{(j)}\partial_{\tau}\tilde{F}(\tau)=\underbrace{\rho_{\omega_0}\ast...\ast \rho_{\omega_0}}_{\text{$(j+1)$-times}}\ast \underbrace{\omega(\tau)^{-1}...\ast \omega(\tau)^{-1}}_{\text{$(j+1)$-times}},$$
where $\ast$ denotes linear combinations of tensorial contractions with respect to the cone  metric $g_0$ (induced by $\omega_0$). This completes the proof of the first statement.

The second statement concerning $F$ follows from the observation that the radial derivative $X \cdot F$ decays quadratically at spatial infinity, which in particular implies that $F$ is bounded on $M$. A similar argument to above applies to the covariant derivatives $\rho^{\frac{i}{2}}\cdot|\nabla^{g,\,i} F|_{g}$ for any $i \geq 1$.
\end{proof}

Recall from Section \ref{sec-path-reparam} that
\begin{equation}\label{G-concrete-expr}
G_s=\log\left(1+\frac{se^{c_{0}}}{1+s(e^{c_{0}}-1)}(e^{F-c_{0}}-1)\right)-\log\left(\frac{\omega_{s}^{n}}{\omega^{n}}\right),
\end{equation}
where $\omega_s=\omega+i\partial\bar{\partial}\Phi_s$. It now follows from Claim \ref{claim-partial-der-F} and Fa\`a di Bruno's formula applied to the universal smooth function $U(y):=\log\left(1+\frac{se^{c_{0}}}{1+s(e^{c_{0}}-1)}(e^{y}-1)\right),\,y\in\R,$ in a neighborhood of $0\in\R$ that the first term on the right-hand side of \eqref{G-concrete-expr} lies uniformly in $C^{2k,\,2\alpha}_{X,\,2}(M)$. The same also holds true for the second term of the right-hand side of \eqref{G-concrete-expr}. Indeed, since $\varphi_{\tau}^*\left(i\partial\bar{\partial}\Phi_s\right)$ is independent of $\tau$ on $M\setminus K$ and decays quadratically with $g_0$-derivatives, one has that
 \begin{equation*}
\varphi_{\tau}^*\left(\frac{\omega_{s}^{n}}{\omega^{n}}\right)=\frac{\left(\omega_0+(-\tau)\left(\rho_{\omega_0}+i\partial\bar{\partial}\Phi_s\right)\right)^{n}}
{(\omega_0+(-\tau)\rho_{\omega_0})^n}=1+\sum_{k\,=\,0}^{n-1}{n \choose k}(-\tau)^{n-k}\frac{\omega_0^{k}\wedge\left(\rho_{\omega_0}+i\partial\bar{\partial}\Phi_s\right)^{n-k}}{(\omega_0+(-\tau)\rho_{\omega_0})^n},
\end{equation*}
from which it can be deduced that $\frac{\omega_{s}^{n}}{\omega^{n}}-1\in C^{2k,\,2\alpha}_{X,\,2}(M)$ uniformly in $s$. A further application of Fa\`a di Bruno's formula to $U(y):=\log (1+y),\, y\in\R,$ in a neighborhood of $0\in\R$ ensures that $\log \left(\frac{\omega_{s}^{n}}{\omega^{n}}\right) \in C^{2k,\,2\alpha}_{X,\,2}(M)$ uniformly in $s$. From this, the result follows.
\end{proof}

Finally, we arrive at the promised sharp form of the weighted $C^{k}$-estimate.
\begin{corollary}[Sharp weighted $C^k$-bounds]\label{coro-bds-Ck-weight}
Let $(\vartheta_s)_{0\,\leq\, s\,\leq\, 1}$ be a path of solutions in $ C^{\infty}_{X}(M)$ to \eqref{starstar-s}. For each $k\geq 0$ and $\alpha\in\left(0,\,\frac{1}{2}\right)$, there exists a positive constant $C_{k,\,\alpha}$ such that for all $s\in[0,\,1]$,
\begin{equation*}\label{bds-Ck-weight}
\begin{split}
\|\vartheta_s\|_{C_{X,\,0}^{2k+2,\,2\alpha}(M)}\leq C_{k,\,\alpha}.
\end{split}
\end{equation*}
\end{corollary}

\begin{proof}
We linearize \eqref{starstar-s} by applying \eqref{equ:taylor-exp} with $\tilde{\omega}:=\omega_s$ and $\psi:=\vartheta_s$ to obtain the equality
\begin{equation}
\begin{split}\label{lin-eqn-final-coro}
\Delta_{h_s,\,X}\vartheta_s=2\int_0^1\int_0^{u}\arrowvert \partial\bar{\partial}\vartheta_s\arrowvert^2_{(g_s)_{\tau\vartheta_s}}\,d\tau\,du+2G_s=:\overline{G}_s,
\end{split}
 \end{equation}
where, for $\tau\in[0,\,1]$, $(g_s)_{\tau\vartheta_s}$ is the Riemannian metric induced by the K\"ahler form $\omega_s+\tau(i\partial\bar{\partial}\vartheta_s)$.
For each $k\geq 0$, we then see from Proposition \ref{prop-bds-Ck-weight} and Lemma \ref{lemma-check-data-F-s-fct-space} that $\overline{G}_s$ satisfies the bounds
 \begin{equation}\label{data-new-look-good}
 f^{1+\frac{k}{2}}|\nabla^{g_s,\,k}\overline{G}_s|_{g_s}\leq C_k.
 \end{equation}
In particular, interior estimates for the drift Laplacian $\Delta_{g_s,\,X}$ now tell us that for each $k\geq 0$,
  \begin{equation}\label{ber-shi-k=1}
 f^{\frac{k}{2}}|\nabla^{g_s,\,k}\vartheta_s|_{g_s}\leq C_k.
 \end{equation}
Indeed, let us demonstrate \eqref{ber-shi-k=1} for $k=1$ ($k=0$ has already been established in combining Propositions \ref{prop-bd-abo-uni-psi} and \ref{prop-bd-bel-uni-psi}).

The proof is purely Riemannian. The Bochner formula for functions asserts that
 \begin{equation*}
 \begin{split}
\Delta_{g_s,\,X}|\nabla^{g_s}\vartheta_s|^2_{g_s}&=2|\nabla^{g_s,\,2}\vartheta_s|^2_{g_s}+2\left(\Ric(g_s)+\frac{1}{2}\mathcal{L}_Xg_s\right)(\nabla^{g_s}\vartheta_s,\nabla^{g_s}\vartheta_s)+2g_s(\nabla^{g_s}\Delta_{g_s,\,X}\vartheta_s,\nabla^{g_s}\vartheta_s)\\
&\geq 2|\nabla^{g_s,\,2}\vartheta_s|^2_{g_s}+\left(2-\frac{C}{f}\right)|\nabla^{g_s}\vartheta_s|^2_{g_s}-\frac{C}{f^{\frac{3}{2}}}|\nabla^{g_s}\vartheta_s|_{g_s}\\
&\geq 2|\nabla^{g_s,\,2}\vartheta_s|^2_{g_s}+\left(2-\frac{C}{f}\right)|\nabla^{g_s}\vartheta_s|^2_{g_s}-\frac{C}{f^{2}},
\end{split}
\end{equation*}
where $C$ is a uniform positive constant that may vary from line to line. Here, we have used the asymptotics of $g_s$,
\eqref{data-new-look-good} in the second line, and Young's inequality in the last line.
An adaptation of the proof of Lemma \ref{lemma-dec-sub-fct} then ensures that \eqref{ber-shi-k=1} holds for $k=1$. The proof for higher values of $k$ can be proved in a similar fashion and so we omit the details.

Now,  $X\cdot \vartheta_s=\Delta_{h_s}\vartheta_s-\overline{G}_s$ by \eqref{lin-eqn-final-coro}. Therefore, after recalling that $\widetilde{\vartheta_s}(\tau):=(\varphi_{\tau}^X)^*\vartheta_s$,
we see that for each $k\geq 0$,
  \begin{equation*}
 f^{1+\frac{k}{2}}|\nabla^{g_s,\,k}\left(\partial_{\tau}\vartheta_s|_{\tau=-1}\right)|_{g_s}\leq C_k.
 \end{equation*}
Similarly, one can prove (by differentiating \eqref{lin-eqn-final-coro} iteratively) that for all $j\geq 2$ and $k\geq 0$,
\begin{equation*}
 f^{j+\frac{k}{2}}|\nabla^{g_s,\,k}\left(\partial_{\tau}^{(j)}\vartheta_s|_{\tau=-1}\right)|_{g_s}\leq C_{k,\,j}.
\end{equation*}
This completes the proof of the corollary, and correspondingly, of the weighted estimates.
\end{proof}

\subsection{Proof of Theorem \ref{mainthm2}: Existence when $t=0$}\label{sec-proof-main-thm}

We are now able to prove Theorem \ref{mainthm2}. To this end, set
\begin{equation*}
\begin{split}
S&:=\left\{s\in[0,\,1]\,|\,\textrm{there exists $\psi_s\in\mathcal{M}_{X}^{\infty}$
satisfying \eqref{star-s}}\right\}.
\end{split}
\end{equation*}
Note that $S\neq\emptyset$ because $0\in S$ (take $\psi_{0}=0$).

We first claim that $S$ is open. Indeed, this follows from Theorem \ref{Imp-Def-Kah-Ste}: if $s_0\in S$, then by Theorem \ref{Imp-Def-Kah-Ste}, there exists $\epsilon_{0}>0$ such that for all $s\in(s_{0}-\epsilon_{0},\,s_{0}+\epsilon_{0})$, there exists a solution $\psi_{s}\in \mathcal{M}^{4,\,2\alpha}_{X}(M)$
to $\eqref{star-s}$ with data $F_{s}\in\left(\mathcal{C}^{2,\,2\alpha}_{X}(M)\right)_{\omega,\,0}$. Since the data $F_{s}$ lies in $\mathcal{C}^{\infty}_{X}(M)$ by Lemma \ref{lemma-check-data-F-s-fct-space},
Theorem \ref{Imp-Def-Kah-Ste} ensures that for each
$s$ in this interval, $\psi_{s}\in \mathcal{M}^{\infty}_{X}(M)$.
It follows that $(s_{0}-\epsilon_{0},\,s_{0}+\epsilon_{0})\cap[0,\,1]\subseteq S$.

We next claim that $S$ is closed. To see this, take a sequence
$(s_k)_{k\,\geq\,0}$ in $S$ converging to some $s_{\infty}\in S$.
Then for $F_k:=F_{s_{k}}$, $k\geq 0$, the corresponding solution $\psi_{s_k}=:\psi_k$,
$k\geq 0$, of \eqref{star-s} satisfies
 \begin{equation}
(\omega+i\partial\bar{\partial}\psi_k)^n=
e^{F_{k}+\frac{X}{2}\cdot\psi_k}\omega^{n},\qquad k\geq 0.\label{MA-seq}
\end{equation}
Thanks to Lemma \ref{lemma-check-data-F-s-fct-space}, the sequence $(F_{{k}})_{k\,\geq\,0}$ is uniformly bounded in $\mathcal{C}^{2,\,2\alpha}_{X}(M)$.
 As a consequence, the sequence $(\psi_k)_{k\,\geq\,0}$ is
 uniformly bounded in $\mathcal{M}^{4,\,2\alpha}_{X}(M)$
by Corollary \ref{coro-bds-Ck-weight}. Indeed, recall the correspondence between solutions of \eqref{star-s} and \eqref{starstar-s}: $\psi_k$ is a solution to \eqref{star-s} if and only if $\vartheta_{s_k}=\psi_{s_k}-\Phi_{s_k}$ is a solution to \eqref{starstar-s}.  The
Arzel\`a-Ascoli theorem therefore allows us to pull out a subsequence of
 $(\psi_k)_{k\,\geq\,0}$ that converges to some $\psi_{\infty}\in C^{4,\,2\alpha'}_{\operatorname{\operatorname{loc}}}(M)$,
 $\alpha'\in(0,\alpha)$. As $(\psi_k)_{k\,\geq\,0}$ is uniformly bounded in
  $\mathcal{M}^{4,\,2\alpha}_{X}(M)$,
  $\psi_{\infty}$ will also lie in $\mathcal{M}^{4,\,2\alpha}_{X}(M)$.
  We need to show that\linebreak $(\omega+i\partial\bar{\partial}\psi_{\infty})(x)>0$
at every point $x\in M$. For this, it
 suffices to show that
  $(\omega+i\partial\bar{\partial}\psi_{\infty})^n(x)>0$
  for every $x\in M$. This is seen to hold true by letting $k$ tend to $+\infty$
   (up to a subsequence) in \eqref{MA-seq}. The fact that $\psi_{\infty}\in\mathcal{M}^{\infty}_{X}(M)$
   follows from Corollary \ref{coro-bds-Ck-weight}.

Finally, as an open and closed non-empty subset of $[0,\,1]$, connectedness of
$[0,\,1]$ implies that $S=[0,\,1]$. This completes the proof of the Theorem \ref{mainthm2}.

\section{Proof of Theorem \ref{mainthm3}: Openness at $t=0$}\label{opennesss}

In this section, we prove Theorem \ref{mainthm3} which states that for $t\in[0,\,1)$ sufficiently small, there exists a torus-invariant K\"ahler metric $\omega_t:=\omega+i\partial\bar{\partial}\varphi_t$ on $M$ with
\begin{equation}\label{bak-eme-cond}
\rho_{\omega_{t}}+\frac{1}{2}\mathcal{L}_{X}\omega_{t}=t\omega_{t}+(1-t)\omega.
\end{equation}
In order to do this, it suffices to solve the following complex Monge-Amp\`ere equation with data $F=O(r^{-2})$ with $g_{0}$-derivatives:
\begin{equation}\tag{$\ast_{t}$}
\log\left(\frac{(\omega+i\partial\bar{\partial}\varphi_{t})^{n}}{\omega^{n}}\right)=F+\frac{X}{2}\cdot\varphi_{t}-t\varphi_{t},\quad
\textrm{$\varphi_{t}\in C^{\infty}(M)$ torus-invariant},\quad\omega+i\partial\bar{\partial}\varphi_{t}>0,\quad t\in[0,\,1),
\end{equation}
where recall from Proposition \ref{mainprop}(i) that $F$ satisfies
\begin{equation*}\label{pot-Ric}
i\partial\bar{\partial}F=\rho_{\omega}+\frac{1}{2}\mathcal{L}_{X}\omega-\omega.
\end{equation*}

As in the setting of a Fano manifold, the main issue in proving Theorem \ref{mainthm3}
is that the linearisation of \eqref{ast-t-bis} at $t=0$ is not an isomorphism without further constraints on the function spaces involved, as the statement of Theorem \ref{iso-sch-Laplacian-pol} alludes to. In addition, a solution to \eqref{ast-t-bis} can be shown to grow polynomially at infinity under rather mild assumptions, as demonstrated below.
To overcome these issues, one has to work with function spaces that are strictly larger than those that were introduced previously in Section \ref{linear-theory-section}.
To this end, let $\Psi:=\psi_1$ ($=\varphi_{0}$) be the solution to $(\star_{1})$ ($=(\ast_{0})$) dictated by Theorem \ref{mainthm2} and \textbf{define}
$\tilde{\omega}:=\omega+i\partial\bar{\partial}\Psi$. This is the K\"ahler form of our new background metric $\tilde{g}$. Since $\tilde{\omega}$ coincides with $\sigma_{1}$ from before, thanks to Lemma \ref{lemma-tr-star-star}, we have a normalised Hamiltonian potential $\tilde{f}:=f_{\sigma_{1}}$ of the real holomorphic vector field $X$ on $M$ with respect to $\sigma_{1}$.
Let $\tilde{\rho}$ be any strictly positive $JX$-invariant smooth function on $M$ equal to $\tilde{f}$ outside a compact subset and bounded below by $1$.

We introduce the following map:
\begin{equation*}
\begin{split}
MA:(t,\psi)\in\,&[0,1]\times\left\{\varphi\in C^2_{\operatorname{\operatorname{loc}}}(M)\,|\,\tilde{\omega}_{\varphi}:=\tilde{\omega}+i\partial\bar{\partial}\varphi>0\right\}\mapsto\log\left(\frac{\tilde{\omega}_{\psi}^n}{\tilde{\omega}^n}\right)
-\frac{X}{2}\cdot\psi+t\psi+t\varphi_0\in\mathbb{R}.
\end{split}
\end{equation*}
Observe that solving \eqref{ast-t-bis} with $\varphi_{t}$ is equivalent to solving $MA(t,\psi_t)=0$ with $\psi_t:=\varphi_t-\varphi_0$.
For any $\psi\in C_{\operatorname{\operatorname{loc}}}^{2}(M)$, let $\tilde{g}_{\psi}$ (respectively $\tilde{g}_{t\psi}$) denote the K\"ahler metric associated to the K\"ahler form $\tilde{\omega}_{\psi}$
(resp.~$\tilde{\omega}_{t\psi}$ for any $t\in[0,\,1]$). We summarize the main local properties of the map $MA$. Brute force computations show that
\begin{equation}
\begin{split}
MA(0,0)&=0,\nonumber\\
D_{0}MA(0,\,\cdot)(u)&=\Delta_{\tilde{\omega}}u-\frac{X}{2}\cdot u,\quad u\in C^2_{\operatorname{\operatorname{loc}}}(M),\nonumber\\
\frac{d^2}{d\tau^2}\left(MA(0,\tau\psi)\right)&=\frac{d}{d\tau}(\Delta_{\tilde{\omega}_{\tau\psi}}\psi)=-\arrowvert\partial\bar{\partial}\psi\arrowvert^2_{\tilde{g}_{\tau\psi}}\quad\textrm{for $\tau\in[0,\,1]$},\label{equ:sec-der-bis}
 \end{split}
 \end{equation}
\begin{equation}\label{equ:taylor-exp-bis}
 \begin{split}
 MA(t,\psi)&=t\psi+t\varphi_0+\left.\frac{d}{d\tau}\right|_{\tau\,=\,0}MA(0,\tau\psi)+\int_0^1\int_0^{u}\frac{d^2}{d\tau^2}(MA(0,\tau\psi))\,d\tau\,du\\
 &=\Delta_{\tilde{\omega}}\psi-\frac{X}{2}\cdot\psi+t\psi+t\varphi_0-\int_0^1(1-\tau)\arrowvert \partial\bar{\partial}\psi\arrowvert^2_{\tilde{g}_{\tau\psi}}\,d\tau.
 \end{split}
 \end{equation}

\subsection{Function spaces}\label{fuckionspaces}
We next introduce the function spaces within which we will work.
\begin{itemize}
\item For $(x,\,\tau)\in M\times(-\infty,0)$, define the parabolic neighborhood $P_{r}(x,\,\tau)$ of $(x,\tau)$ of radius $r>0$ by
\begin{equation*}
P_{r}(x,\tau):=B_{\tilde{g}}(x,r)\times(-r^2+\tau,\tau].
\end{equation*}
Let $(\varphi^{X}_{\tau})_{\tau<0}$ denote the flow of $\frac{X}{2(-\tau)}$ such that $\varphi^{X}_{\tau}\big|_{\tau=-1}=\Id_M$. Define for a real-valued function $u:M\rightarrow\R$ the following time-dependent function:
\begin{equation*}
\tilde{u}(x,\tau):=(\varphi^X_{\tau})^*u(x),\qquad x\in M,\qquad\tau<0.
\end{equation*}
\item For $r=\iota_0\in (0,\inj_{\tilde{g}}M)$ and $\tau_0\in (-1,\min\{0,-1+\iota_0^2\})$ fixed once and for all, denote for $k\geq 0$ and $\alpha\in\left(0,\frac{1}{2}\right)$ the standard $C^{2k,\,2\alpha}$-norm on $P_{\iota_0}(x,\tau_0)$ by $\|\cdot\|_{C^{2k,\,2\alpha}(P_{\iota_0}(x))}$. For $\beta\in \R$, define $C_{X,\,\beta}^{2k,\,2\alpha}(M)$ to be the space of $JX$-invariant continuous functions $u$ on $M$ such that
\begin{equation*}
\norm{u}_{C^{2k,\,2\alpha}_{X,\,\beta}(M)} :=\sup_{x\,\in\,M}\tilde{\rho}(x)^{\frac{\beta}{2}}\|\tilde{u}\|_{C^{2k,\,2\alpha}(P_{\iota_0}(x))} < \infty.
\end{equation*}
\item For $\beta\in\R$, the source function space is defined as
\begin{equation*}
\mathcal{D}_{X,\,\beta}^{4,\,2\alpha}(M):=\left\{u\in C^{4,\,2\alpha}_{\operatorname{loc}}(M)\,|\,\tilde{\rho}^{\frac{j}{2}}(\nabla^{\tilde{g}})^ju\in C^{2,2\alpha}_{X,\,\beta}(M),\quad j=0,1,2,\quad\textrm{and}\quad X\cdot u\in C^{2,\,2\alpha}_{X,\,\beta}(M)\right\},
\end{equation*}
endowed with the norm
\begin{equation*}
\|u\|_{\mathcal{D}_{X,\,\beta}^{4,\,2\alpha}(M)}:=\sum_{j\,=\,0}^2\|\tilde{\rho}^{\frac{j}{2}}(\nabla^{\tilde{g}})^ju\|_{C^{2,2\alpha}_{X,\,\beta}(M)}+\|X\cdot u\|_{C^{2,2\alpha}_{X,\,\beta}(M)}.
\end{equation*}
\item For $\beta\in\R$, the target function space is defined as
\begin{equation*}
\mathcal{C}_{X,\,\beta}^{2,\,2\alpha}(M):=\left\{u\in C^{2,\,2\alpha}_{\operatorname{loc}}(M)\,|\,\tilde{\rho}^{\frac{j}{2}}(\nabla^{\tilde{g}})^ju\in C^{0,\,2\alpha}_{X,\,\beta}(M),\quad j=0,1,2,\quad\textrm{and}\quad X\cdot u\in C^{0,\,2\alpha}_{X,\,\beta}(M)\right\},
\end{equation*}
endowed with the norm
\begin{equation*}
\|u\|_{\mathcal{C}_{X,\,\beta}^{2,\,2\alpha}(M)}:=\sum_{j\,=\,0}^2\|\tilde{\rho}^{\frac{j}{2}}(\nabla^{\tilde{g}})^ju\|_{C^{0,\,2\alpha}_{X,\,\beta}(M)}+\|X\cdot u\|_{C^{0,\,2\alpha}_{X,\,\beta}(M)}.
\end{equation*}

\begin{remark}
  Note that the condition on $X \cdot u$ here is redundant because it is already implied by the $j = 1$ term in the first summand on the right-hand side. However, for consistency with the function spaces defined in Section \ref{function-spaces-subsection}, we choose to keep this term in the definition.
\end{remark}

\item Denote by $\pi_0$ the $L^2(e^{-\tilde{f}}\tilde{\omega}^n)$-projection on the kernel of the drift Laplacian, i.e.,
$$\pi_0(u):=\frac{1}{\int_{M}e^{-\tilde{f}}\tilde{\omega}^n}\int_{M}u\,e^{-\tilde{f}}\tilde{\omega}^n.$$
\item Finally, we define the space
\begin{equation*}
\begin{split}
\mathcal{M}^{4,\,2\alpha}_{X,\,\beta}(M)&:=\left\{\varphi\in C^2_{\operatorname{\operatorname{loc}}}(M)\,|\,\tilde{\omega}+i\partial\bar{\partial}\varphi>0\right\}\bigcap \mathcal{D}^{4,\,2\alpha}_{X,\,\beta}(M).
\end{split}
\end{equation*}
\end{itemize}

We need to show that $\mathcal{M}^{4,\,2\alpha}_{X,\,\beta}(M)$ is an open convex subspace of $\mathcal{D}^{4,\,2\alpha}_{X,\,\beta}(M)$. This is true for $\beta>-2$.
\begin{lemma}\label{nice-cosy-lemma}
If $\beta>-2$, then the space of K\"ahler potentials $\mathcal{M}^{4,\,2\alpha}_{X,\,\beta}(M)$ is an open convex subspace of $\mathcal{D}^{4,\,2\alpha}_{X,\,\beta}(M)$.
\end{lemma}

\begin{proof}
Let $\psi_0\in \mathcal{D}_{X,\,\beta}^{4,\,2\alpha}(M)$ be such that $\tilde{\omega}+i\partial\bar{\partial}\psi_{0}>0$. 
We need to show that there exists a neighborhood $U_{0}$ of $\psi_0$ in $\mathcal{D}_{X,\,\beta}^{4,\,2\alpha}(M)$ such that $\tilde{\omega}+i\partial\bar{\partial}\psi>0$ 
for all $\psi\in U_{0}$. To this end, we estimate the difference between the $(1,\,1)$-forms $\tilde{\omega}+i\partial\bar{\partial}\psi$ and $\tilde{\omega}+i\partial\bar{\partial}\psi_0$ in the following way:
\begin{equation*}
\begin{split}
\tilde{\omega}+i\partial\bar{\partial}\psi-(\tilde{\omega}+i\partial\bar{\partial}\psi_0)&=i\partial\bar{\partial}(\psi-\psi_0).
\end{split}
\end{equation*}
 By virtue of the very definition of the norm on $\mathcal{D}_{X,\,\beta}^{4,\,2\alpha}(M)$, we see that
\begin{equation*}|i\partial\bar{\partial}(\psi-\psi_0)|_{\tilde{\omega}}\leq C\tilde{\rho}^{-\frac{\beta}{2}-1}\|\psi-\psi_0\|_{\mathcal{D}_{X,\,\beta}^{4,\,2\alpha}(M)}\leq C\|\psi-\psi_0\|_{\mathcal{D}_{X,\,\beta}^{4,\,2\alpha}(M)}
\end{equation*}
for some uniform positive constant $C>0$. Here we have used the fact that $\beta>-2$ and  $\tilde{\rho}\geq 1$ on $M$. Thus, if $\|\psi-\psi_0\|_{\mathcal{D}_{X,\,\beta}^{4,\,2\alpha}(M)}(M)$ is small enough, then $\tilde{\omega}+i\partial\bar{\partial}\psi$ defines a K\"ahler metric on $M$.
\end{proof}

\subsection{Fredholm properties of the linearised operator}
The main technical result of this section is that the modified drift Laplacian of $\tilde{g}$ is an isomorphism between polynomially weighted function spaces.
\begin{prop}\label{iso-sch-Laplacian-pol-bis}
Let $\alpha\in\left(0,\,\frac{1}{2}\right)$ and $\beta< 0$. Then the map
\begin{equation*}
\begin{split}
\Delta_{\tilde{g},\,X}+\pi_0:\mathcal{D}^{4,\,2\alpha}_{X,\,\beta}(M)
\rightarrow \mathcal{C}^{2,\,2\alpha}_{X,\,\beta}(M)
\end{split}
\end{equation*}
is an isomorphism of Banach spaces.
\end{prop}

\begin{remark}
Contrary to Theorem \ref{iso-sch-Laplacian-pol}, observe that both the source and the target spaces are modelled on functions with the same polynomial growth at infinity.
In order to produce solutions to \eqref{bak-eme-cond} for small $t>0$, we will further impose the restriction
$\beta\in (-2,\,0)$ so that the non-linear terms make sense within these function spaces. 
This also explains why we work with functions that are four times differentiable, since the standard parabolic regularity theory at the local scale $\iota_0>0$ does not improve the decay of higher derivatives.
\end{remark}

\begin{remark}
The choice of the coefficient $1$ in front of the map $\pi_0$ is arbitrary. Any non-zero real number would work.
\end{remark}

\begin{proof}[Proof of Proposition \ref{iso-sch-Laplacian-pol-bis}]
We first show that the map in question is well-defined and continuous between the stated Banach spaces.

\begin{claim}\label{fuckme}
  $\Delta_{\tilde{g},\,X}+\pi_0:\mathcal{D}^{4,\,2\alpha}_{X,\,\beta}(M)
\rightarrow \mathcal{C}^{2,\,2\alpha}_{X,\,\beta}(M)$ is a well-defined continuous map.
\end{claim}

\begin{proof}[Proof of Claim \ref{fuckme}]
By finiteness of the weighted volume of $M$, the projection map $\pi_0$ restricted to $\mathcal{D}^{4,\,2\alpha}_{X,\,\beta}(M)$ is well-defined. Being constant for $u\in L^2(e^{-\tilde{f}}\tilde{\omega}^n)$, $\pi_0(u)$ lies in $\mathcal{C}^{2,\,2\alpha}_{X,\,\beta}(M)$ as soon as $\beta\leq 0$. The continuity of $\pi_0$ is a consequence of
the following sequence of elementary (in)equalities:
\begin{equation*}
\begin{split}
\|\pi_0(u)\|_{\mathcal{C}^{2,\,2\alpha}_{X,\,\beta}(M)}&=\norm{\pi_0(u)}_{C^{0}_{X,\,\beta}(M)}=\sup_{x\,\in\,M}\tilde{\rho}(x)^{\frac{\beta}{2}}\|\widetilde{\pi_0(u)}\|_{C^{0}(P_{\iota_0}(x))}\\
&\leq |\pi_0(u)|\leq \underbrace{C\left(\int_M\tilde{\rho}(x)^{-\frac{\beta}{2}}\,e^{-\tilde{f}}d\mu_{\tilde{g}}\right)}_{=:C'}\left(\sup_{x\,\in\,M}\tilde{\rho}(x)^{\frac{\beta}{2}}|u(x)|\right)\\
&\leq C'\left(\sup_{x\,\in\,M}\tilde{\rho}(x)^{\frac{\beta}{2}}\|\tilde{u}\|_{C^{0}(P_{\iota_0}(x))}
\right)\\
&\leq C'\|u\|_{\mathcal{D}^{4,\,2\alpha}_{X,\,\beta}(M)}.
\end{split}
\end{equation*}
Here, we have used the fact that $\tilde{u}(x,\tau)=u(x)$ at $\tau=-1$ in the penultimate line.

We next show that $u\in \mathcal{D}^{4,\,2\alpha}_{X,\,\beta}(M)\mapsto X\cdot u\in \mathcal{C}^{2,\,2\alpha}_{X,\,\beta}(M)$ is a well-defined continuous map. The proof that $u\in \mathcal{D}^{4,\,2\alpha}_{X,\,\beta}(M)\mapsto \Delta_{\tilde{\omega}} u\in \mathcal{C}^{2,\,2\alpha}_{X,\,\beta}(M)$ is a well-defined continuous map is similar and will therefore 
be omitted.

First observe that $X\cdot u\in C^{0,\,2\alpha}_{X,\,\beta}(M)$ by definition of $\mathcal{D}^{4,\,2\alpha}_{X,\,\beta}(M)$, since $X\cdot u\in C^{2,\,2\alpha}_{X,\,\beta}(M)$.
We henceforth denote the flow of $\varphi^X_{\tau}$ by $\varphi_{\tau}$. As $\widetilde{X\cdot(X\cdot u)}=(-2\tau)\partial_{\tau}(\varphi_{\tau}^*(X\cdot u))=(-2\tau)\partial_{\tau}(\widetilde{X\cdot u})$, we see that
\begin{equation*}
\|X\cdot(X\cdot u)\|_{C^{0,\,2\alpha}_{X,\,\beta}(M)}\leq C\sup_{x\,\in\,M}\tilde{\rho}(x)^{\frac{\beta}{2}}\|\widetilde{X\cdot u}\|_{C^{2,2\alpha}(P_{\iota_0}(x))}\leq C\|u\|_{\mathcal{D}^{4,\,2\alpha}_{X,\,\beta}(M)},
\end{equation*}
where $C$ depends on an upper bound of $|\tau_0|$ (the definition of which one can recall in Section \ref{fuckionspaces}).

We next consider the term $\tilde{\rho}^{\frac{1}{2}}\nabla^{\tilde{g}}(X\cdot u)$:
\begin{equation*}
\begin{split}
\mathcal{L}_X\left(\tilde{\rho}^{\frac{1}{2}}\nabla^{\tilde{g}}u\right)&=\left(X\cdot \tilde{\rho}^{\frac{1}{2}}\right)\nabla^{\tilde{g}}u+\tilde{\rho}^{\frac{1}{2}}[\mathcal{L}_X\left(\nabla^{\tilde{g}}\right)]u+\tilde{\rho}^{\frac{1}{2}}\nabla^{\tilde{g}}(X\cdot u).
\end{split}
\end{equation*}
Now, one can write schematically: $[\mathcal{L}_X\left(\nabla^{\tilde{g}}\right)]u=\tilde{g}^{-1}\ast \mathcal{L}_X\tilde{g}\ast \nabla^{\tilde{g}}u.$
In particular,
\begin{equation*}
\begin{split}
\varphi_{\tau}^*\left(\tilde{\rho}^{\frac{1}{2}}\nabla^{\tilde{g}}(X\cdot u)\right)=(-2\tau)\partial_{\tau}\left(\widetilde{\tilde{\rho}^{\frac{1}{2}}\nabla^{\tilde{g}}u}\right)+\varphi_{\tau}^*\left(-X\cdot\log \tilde{\rho}^{\frac{1}{2}}+\tilde{\rho}^{-\frac{1}{2}}\tilde{g}^{-1}\ast \mathcal{L}_X\tilde{g}\,\ast\right)\widetilde{\tilde{\rho}^{\frac{1}{2}}\nabla^{\tilde{g}}u}.
\end{split}
\end{equation*}
As $\widetilde{\tilde{\rho}^{\frac{1}{2}}\nabla^{\tilde{g}}u}\in C^{2,\,2\alpha}_{X,\,\beta}(M)$ by assumption and since the expression in front of $\widetilde{\tilde{\rho}^{\frac{1}{2}}\nabla^{\tilde{g}}u}$ on the right-hand side of the previous identity
lies in $C^{0,\,2\alpha}_{X,\,0}(M)$ thanks to Proposition \ref{mainprop}, it follows that there exists $C>0$ such that $\|\tilde{\rho}^{\frac{1}{2}}\nabla^{\tilde{g}}(X\cdot u)\|_{C^{0,\,2\alpha}_{X,\,\beta}(M)}\leq C\|u\|_{\mathcal{D}^{4,\,2\alpha}_{X,\,\beta}(M)}.$
Similar reasoning shows that $\tilde{\rho}\nabla^{\tilde{g},\,2}(X\cdot u)\in C^{0,\,2\alpha}_{X,\,\beta}(M)$ which concludes the desired claim on the map $u\in \mathcal{D}^{4,\,2\alpha}_{X,\,\beta}(M)\rightarrow X\cdot u\in \mathcal{C}^{2,\,2\alpha}_{X,\,\beta}(M)$.
\end{proof}

In order to show that the map $\Delta_{\tilde{g},\,X}+\pi_0$ is injective, note the following: if $u\in \mathcal{D}^{4,\,2\alpha}_{X,\,\beta}(M)$ satisfies $\Delta_{\tilde{g},\,X}(u)+\pi_0(u)=0$, then an integration by parts with respect to $e^{-\tilde{f}}\tilde{\omega}^n$, justified by the asymptotics of $u$ and its derivatives up to order $2$, leads to the fact that $\pi_0(u)=0$. Multiplying the equation $\Delta_{\tilde{g},\,X}(u)=0$ by $u$ and integrating by parts once again, one also sees that $\|\nabla^{\tilde{g}}u\|_{L^2(e^{-\tilde{f}}\tilde{\omega}^n)}=0$, i.e., that $u$ is constant on $M$. Since the weighted mean value of $u$ is $0$, this results in the desired conclusion, namely that $u=0$ identically.

Finally, we deal with the surjectivity of the map $\Delta_{\tilde{g},\,X}+\pi_0$. If $u\in \mathcal{D}^{4,\,2\alpha}_{X,\,\beta}(M)$, then we proceed as at the beginning of the proof of Theorem \ref{iso-sch-Laplacian-pol}: given $F\in \mathcal{C}^{2,2\alpha}_{X,\,\beta}(M)$, there exists a solution $\overline{u}$ to \eqref{love-drift-lap} with right-hand side given by $F-\pi_0(F)$ with $\pi_0(\overline{u})=0$.

The following claim is analogous to the proof of Claim \ref{claim-first-rough-growth} and therefore its proof will be omitted.
\begin{claim}\label{claim-first-rough-growth-bis}
There exists a positive constant $C=C(\tilde{\omega},n)$ such that $$|\bar{u}(x)|
\leq Ce^{ \frac{\tilde{f}(x)}{2}}\|F\|_{C^{0}_{X,\,\beta}(M)},\qquad x\in M.$$
\end{claim}

The next claim echoes Claim \ref{claim-sec-rough-growth}.
\begin{claim}\label{claim-sec-rough-growth-bis}
There exists a positive constant $A=A(\tilde{\omega},\,n)$ such that
$$|\bar{u}(x)|\leq A\,\tilde{\rho}(x)^{-\frac{\beta}{2}}\|F\|_{C^0_{X,\,\beta}(M)},\qquad x\in M.$$
\end{claim}

 \begin{proof}[Proof of Claim \ref{claim-sec-rough-growth-bis}]
Since the proof follows the same lines as that of Claim \ref{claim-sec-rough-growth}, we provide only a sketch.
It suffices to use a barrier function of the form $Af^{-\frac{\beta}{2}}+\varepsilon e^{\delta f}$ outside a compact set for $\delta\in\left(\frac{1}{2},\,1\right)$ fixed and where $\varepsilon>0$ and $A>0$ are arbitrary. This one can do by twice applying Lemma \ref{lemma-sub-sol-barrier}. Moreover, Claim \ref{claim-first-rough-growth-bis} ensures that the function $\overline{u}-Af^{-\frac{\beta}{2}}-\varepsilon e^{\delta f}$ is a proper function bounded from above. For this, we crucially rely on the fact that $\beta<0$.
\end{proof}

\begin{claim}\label{claim-first-step-u-overline}
The function $u:=\overline{u}+\pi_0(F)$ lies in $\mathcal{C}^{2,2\alpha}_{X,\,\beta}(M)$ and there exists a positive uniform constant $C>0$ such that
\begin{equation*}
\|u\|_{C^{2,2\alpha}_{X,\,\beta}(M)}\leq C\|F\|_{\mathcal{C}^{2,2\alpha}_{X,\,\beta}(M)}.
\end{equation*}
\end{claim}

\begin{proof}[Proof of Claim \ref{claim-first-step-u-overline}]
Observe that $\tilde{u}$ satisfies the parabolic equation
\begin{equation*}\label{self-sim-eqn}
\partial_{\tau}\tilde{u}-\frac{1}{2}\Delta_{\tilde{g}(\tau)}\tilde{u}=-\frac{1}{2\tau}\left(\widetilde{F}-\pi_0(F)\right).
\end{equation*}
We apply standard interior parabolic Schauder estimates as stated for instance in \cite[Theorem $8.12.1$]{Kry-Boo} to a parabolic ball $P_{\iota_0}(x)$. This yields a positive constant 
$C>0$ such that for all $x\in M$,
\begin{equation*}
\|\tilde{u}\|_{C^{2,2\alpha}(P_{\iota_0/2}(x))}\leq C\left(\|-\widetilde{F}+\pi_0(F)\|_{C^{0,2\alpha}(P_{\iota_0}(x))}+\|\widetilde{u}\|_{C^{0}(P_{\iota_0}(x))}\right).
\end{equation*}
Multiplying across by $\tilde{\rho}(x)^{\frac{\beta}{2}}$ and considering the supremum over $x\in M$ then leads to the bound
\begin{equation*}
\begin{split}
\|\tilde{u}\|_{C^{2,2\alpha}_{X,\,\beta}(M)}&\leq C\left(\|-\widetilde{F}+\pi_0(F)\|_{C^{0,2\alpha}_{X,\,\beta}(M)}+\|\widetilde{u}\|_{C^{0}_{X,\,\beta}(M)}\right)\\
&\leq C\|\widetilde{F}\|_{C^{0,\,2\alpha}_{X,\,\beta}(M)}.
\end{split}
\end{equation*}
Here we have used Claim \ref{claim-sec-rough-growth-bis} in the last line. The term $\pi_0(F)$ has been handled in the same way as $\pi_0(u)$ was in the proof
of Claim \ref{fuckme}. Recalling the definition of the norm on $\mathcal{C}^{2,2\alpha}_{X,\,\beta}(M)$, we reach the desired result.
\end{proof}

\begin{claim}\label{claim-second-step-u-overline}
    There exists a positive uniform constant $C$ such that
\begin{equation*}
\sum_{j\,=\,1}^2\|\tilde{\rho}^{\frac{j}{2}}\nabla^{\tilde{g},\,j}u\|_{C^{2,2\alpha}_{X,\,\beta}(M)}+\|X\cdot u\|_{C^{2,2\alpha}_{X,\,\beta}(M)}\leq C\|F\|_{\mathcal{C}^{2,2\alpha}_{X,\,\beta}(M)}.
\end{equation*}
\end{claim}

\begin{proof}[Proof of Claim \ref{claim-second-step-u-overline}]
We compute the following commutator for a given function $v\in C^4_{\operatorname{loc}}(M)$:
\begin{equation*}
\begin{split}
\left[\Delta_{\tilde{g},\,X},\tilde{\rho}^{\frac{j}{2}}\nabla^{\tilde{g},\,j}\right]v&=\left(\Delta_{\tilde{g},\,X}\tilde{\rho}^{\frac{j}{2}}\right)\nabla^{\tilde{g},\,j}v+j\nabla^{\tilde{g}}_{\nabla^{\tilde{g}}\log\tilde{\rho}}\left(\tilde{\rho}^{\frac{j}{2}}\nabla^{\tilde{g},\,j}v\right)-\frac{j^2}{2}|\nabla^{\tilde{g}}\log\tilde{\rho}|^2_{\tilde{g}}\left(\tilde{\rho}^{\frac{j}{2}}\nabla^{\tilde{g},\,j}v\right)\\
&\quad+\tilde{\rho}^{\frac{j}{2}}\left[\Delta_{\tilde{g},\,X},\nabla^{\tilde{g},\,j}\right]v.
\end{split}
\end{equation*}
Thanks to commutation formulae for the rough Laplacian acting on tensors, one has that
\begin{equation*}
\begin{split}
\left[\Delta_{\tilde{g}},\nabla^{\tilde{g},\,j}\right]v&=\sum_{k\,=\,0}^j\nabla^{\tilde{g},\,j-k}\Rm(\tilde{g})\ast_{\tilde{g}} \nabla^{\tilde{g},\,k}v,\\
\left[\nabla^{\tilde{g}}_{X},\nabla^{\tilde{g},\,j}\right]v&=-j\nabla^{\tilde{g},\,j}v+\sum_{k\,=\,0}^j\nabla^{\tilde{g},\,j-k}\Rm(\tilde{g})\ast_{\tilde{g}} \nabla^{\tilde{g},\,k}v+\sum_{k\,=\,0}^j\nabla^{\tilde{g},\,j-k}\left(\mathcal{L}_{\frac{X}{2}}\tilde{g}-\tilde{g}\right)\ast_{\tilde{g}} \nabla^{\tilde{g},\,k}v.
\end{split}
\end{equation*}
The last identity can be proved by induction on $j\geq 1$ after invoking Proposition \ref{mainprop}. 
In particular, the tensors $\tilde{\rho}^{\frac{j}{2}}\nabla^{\tilde{g},\,j}u$, $j=1,2$, satisfy
\begin{equation*}
\begin{split}
\Delta_{\tilde{g},\,X}\left(\tilde{\rho}^{\frac{j}{2}}\nabla^{\tilde{g},\,j}u\right)&=\tilde{\rho}^{\frac{j}{2}}\nabla^{\tilde{g},\,j}F+\left(-\frac{j^2}{2}|\nabla^{\tilde{g}}\log\tilde{\rho}|^2_{\tilde{g}}+O(\tilde{\rho}^{-1})+\left(\mathcal{L}_{\frac{X}{2}}\tilde{g}-\tilde{g}\right)
\ast_{\tilde{g}}+\Rm(\tilde{g})\ast_{\tilde{g}}\right)\tilde{\rho}^{\frac{j}{2}}\nabla^{\tilde{g},\,j}u\\
&\quad+j\nabla^{\tilde{g}}_{\nabla^{\tilde{g}}\log\tilde{\rho}}\left(\tilde{\rho}^{\frac{j}{2}}\nabla^{\tilde{g},\,j}u\right)
+\sum_{k\,=\,0}^{j-1}\tilde{\rho}^{\frac{j-k}{2}}\nabla^{\tilde{g},\,j-k}\Rm(\tilde{g})\ast_{\tilde{g}} \left(\tilde{\rho}^{\frac{k}{2}}\nabla^{\tilde{g},\,k}u\right)\\
&\quad+\sum_{k\,=\,0}^{j-1}\tilde{\rho}^{\frac{j-k}{2}}\nabla^{\tilde{g},\,j-k}\left(\mathcal{L}_{\frac{X}{2}}\tilde{g}-\tilde{g}\right)\ast_{\tilde{g}} \left(\tilde{\rho}^{\frac{k}{2}}\nabla^{\tilde{g},\,k}v\right),\\
\end{split}
\end{equation*}
where the term $O(\tilde{\rho}^{-1})$ holds with $g$-derivatives thanks to Lemma \ref{lemma-sub-sol-barrier} applied to $\delta:=-\frac{j}{2}$.

By arguing as in the proof of Claim \ref{claim-first-step-u-overline}, one arrives at the desired estimate on the $C^{2,2\alpha}_{X,\,\beta}$-norm of $\tilde{\rho}^{\frac{j}{2}}\nabla^{\tilde{g},\,j}u$ thanks to Proposition \ref{mainprop} and Theorem \ref{mainthm2}, both of which we use to handle the lower order terms involving $\Rm(\tilde{g})$ and $\mathcal{L}_{\frac{X}{2}}\tilde{g}-\tilde{g}$. In order to obtain the estimate on the $C^{2,2\alpha}_{X,\,\beta}$-norm of $X\cdot u$, one first derives the equation satisfied by $\Delta_{\tilde{g},\,X}(X\cdot u)$ and then 
concludes by invoking standard parabolic estimates as in the proof of Claim \ref{claim-first-step-u-overline}.
\end{proof}

Combining Claims \ref{claim-first-step-u-overline} and \ref{claim-second-step-u-overline} shows that the map $\Delta_{\tilde{g},\,X}+\pi_0:\mathcal{D}^{4,\,2\alpha}_{X,\,\beta}\rightarrow  \mathcal{C}^{2,\,2\alpha}_{X,\,\beta}(M)$ is surjective, as asserted.
\end{proof}

\subsection{Small perturbations at $t=0$}
In this section we show, using the implicit function theorem, that the
invertibility of the modified drift Laplacian given by Proposition \ref{iso-sch-Laplacian-pol-bis}
allows for small perturbations in polynomially weighted function spaces
of solutions to the complex Monge-Amp\`ere equation at $t=0$.
This demonstrates openness at $t=0$, and in turn, the existence of solutions to \eqref{bak-eme-cond} for small values of $t>0$, thereby proving Theorem
\ref{mainthm3}.

\begin{theorem}
For $\alpha\in\left(0,\,\frac{1}{2}\right)$ and $\beta\in(-2,0)$, the map
\begin{equation*}
\Phi:(t,\psi)\in [0,\,1]\times\mathcal{M}^{4,\,2\alpha}_{X,\,\beta}(M)
\rightarrow  MA(t,\psi)+\frac{1}{2}\pi_0(\psi)\in\mathcal{C}^{2,\,2\alpha}_{X,\,\beta}(M)
\end{equation*}
is a well-defined $C^1$-map. In particular, there exists $t_0\in (0,\,1)$ and a neighborhood $U_0$ of $0\in \mathcal{D}^{4,\,2\alpha}_{X,\,\beta}(M)$ such that for all $t\in (0,t_0)$, the equation $\Phi(t,\psi_{t})=0$ has a unique solution $\psi_{t}$ in $U_0$.
\end{theorem}

\begin{proof}
We first show that the map $\Phi$ is well-defined provided that $\beta\in(-2,\,0)$. Lemma \ref{nice-cosy-lemma} already restricts $\beta >-2$ to ensure that $\mathcal{M}^{4,\,2\alpha}_{X,\,\beta}(M)$ is a convex open subspace of $\mathcal{D}^{4,\,2\alpha}_{X,\,\beta}(M)$.

By \eqref{equ:taylor-exp-bis} and Proposition \ref{iso-sch-Laplacian-pol-bis}, it suffices to prove the following claim.
\begin{claim}\label{claim-well-def-Phi}
If $\psi\in \mathcal{M}^{4,\,2\alpha}_{X,\,\beta}(M)$, then $$\psi,\qquad \varphi_0,\qquad \int_0^1(1-\tau)\arrowvert \partial\bar{\partial}\psi\arrowvert^2_{\tilde{g}_{\tau\psi}}\,d\tau\in \mathcal{C}^{2,\,2\alpha}_{X,\,\beta}(M).$$
\end{claim}

\begin{proof}[Proof of Claim \ref{claim-well-def-Phi}]
By definition of $ \mathcal{M}^{4,\,2\alpha}_{X,\,\beta}(M)$, it is clear that $\psi\in \mathcal{C}^{2,\,2\alpha}_{X,\,\beta}(M)$ 
for every value of $\beta$. Since $\varphi_0$ lies in $\mathcal{M}^{\infty}_{X}(M)$ (as defined in Section \ref{function-spaces-subsection}) thanks to Theorem \ref{mainthm2},
it is also clear that $\varphi_0\in \mathcal{C}^{2,\,2\alpha}_{X,\,\beta}(M)$ \textit{provided that} $\beta<0$ because of the logarithmic term appearing in the definition of the spaces $\mathcal{D}^{2k+2,\,2\alpha}_{X}(M)$ (again as defined in Section \ref{function-spaces-subsection}). The last term can be written schematically as
\begin{equation*}
\int_0^1\tilde{g}_{s\psi}^{-1}\ast\tilde{g}_{s\psi}^{-1}\ast \nabla^{\tilde{g},\,2}\psi\ast \nabla^{\tilde{g},\,2}\psi\,ds.
\end{equation*}
By definition of the function space $\psi$ lies in, the above expression grows at most like $\tilde{\rho}^{-\beta-2}$, which is $O(\tilde{\rho}^{-\frac{\beta}{2}})$ provided that $\beta\geq -4$, but $\beta>-2$ in any case. A similar reasoning applies to higher derivatives. This concludes the proof of the claim.
\end{proof}

To show that the map $\Phi$ is $C^1$, consider the differential of $\Phi$ at $(t_0,\psi_0)\in[0,\,1]\times \mathcal{M}^{4,\,2\alpha}_{X,\,\beta}(M)$ which is given by 
\begin{equation*}
D_{(t_0,\psi_0)}\Phi(s,\psi)=\Delta_{\tilde{\omega}_{\psi_0},\,X}\psi+\frac{1}{2}\pi_0(\psi)+s(\psi_0+\varphi_0)+t_0\psi,\quad (s,\psi)\in \R\times \mathcal{D}^{4,\,2\alpha}_{X,\,\beta}(M).
\end{equation*}
For $(t_i,\psi_i)\in[0,1]\times \mathcal{M}^{4,\,2\alpha}_{X,\,\beta}(M)$, $i=0,1$, and $(s,\psi)\in \R\times \mathcal{D}^{4,\,2\alpha}_{X,\,\beta}(M)$, we then see that
\begin{equation*}
\begin{split}
\|&\left(D_{(t_1,\psi_1)}\Phi-D_{(t_0,\psi_0)}\Phi\right)(s,\psi)\|_{\mathcal{C}^{2,\,2\alpha}_{X,\,\beta}(M)}=\left\|\left(\Delta_{\tilde{\omega}_{\psi_1}}-\Delta_{\tilde{\omega}_{\psi_0}}\right)\psi+s(\psi_1-\psi_0)+(t_1-t_0)\psi\right\|_{\mathcal{C}^{2,\,2\alpha}_{X,\,\beta}(M)}\\
&\leq \left\|\left(\tilde{g}_{\psi_1}^{-1}-\tilde{g}_{\psi_0}^{-1}\right)\ast i\partial\bar{\partial}\psi\right\|_{\mathcal{C}^{2,\,2\alpha}_{X,\,\beta}(M)}+|s|\|\psi_1-\psi_0\|_{\mathcal{C}^{2,\,2\alpha}_{X,\,\beta}(M)}+|t_1-t_0|\|\psi\|_{\mathcal{C}^{2,\,2\alpha}_{X,\,\beta}(M)}\\
&\leq C\left(\|\psi_1-\psi_0\|_{\mathcal{D}^{4,\,2\alpha}_{X,\,\beta}(M)}+|t_1-t_0|\right)\left(|s|+\|\psi\|_{\mathcal{D}^{4,\,2\alpha}_{X,\,\beta}(M)}\right).
\end{split}
\end{equation*}
This proves the desired regularity property on $\Phi$ by definition of the operator norm on continuous linear maps $\R\times\mathcal{D}^{4,\,2\alpha}_{X,\,\beta}(M)\mapsto\mathcal{C}^{2,\,2\alpha}_{X,\,\beta}(M)$. The implicit function theorem applied to $\Phi$ at $(0,0)\in [0,1]\times\mathcal{M}^{4,\,2\alpha}_{X,\,\beta}(M)$ now gives the desired result, since $D_{(0,0)}\Phi=\Delta_{\tilde{\omega},\,X}+\frac{1}{2}\pi_0$ is an isomorphism of Banach spaces thanks to Proposition \ref{iso-sch-Laplacian-pol-bis}.
\end{proof}

\newpage

\bibliographystyle{amsalpha}

\bibliography{ref2}

\end{document}